\documentclass[11pt]{amsart}

\usepackage{etex}
\reserveinserts{28}

\usepackage{geometry}
\usepackage{listings} 
\usepackage{algorithmic,algorithm}
\usepackage{multirow}
\usepackage{enumerate}

\usepackage[pdftex]{graphicx}
\usepackage[pdftex]{color} % black,white,red,green,blue,cyan,magen ta,yellow
\usepackage[pdftex,colorlinks]{hyperref}
\usepackage{graphicx,pst-eps,epstopdf}
\usepackage{lscape}
\usepackage{indentfirst}
\usepackage{latexsym}
\usepackage{amsmath, amsfonts, amssymb,mathrsfs}
\usepackage{subfigure,pstricks,pst-node}
\usepackage{pst-eps,epstopdf}
\usepackage{verbatim}
\usepackage{tikz}
\usepackage{caption}
\usepackage{nicefrac}
\usepackage{slashbox}
\usepackage{float}
\usepackage{bm}

\usepackage{mathpazo}
\usepackage[mathpazo]{flexisym}
\hypersetup{
    bookmarks=true,         % show bookmarks bar?
    unicode=true,          % non-Latin characters in Acrobat?????s bookmarks
    pdftoolbar=true,        % show Acrobat?????s toolbar?
    pdfmenubar=true,        % show Acrobat?????s menu?
    pdffitwindow=true,      % page fit to window when opened
    pdftitle={},    % title
    pdfauthor={Xiaozhe Hu},     % author
    pdfsubject={},   % subject of the document
    pdfnewwindow=true,      % links in new window
    pdfkeywords={}, % list of keywords
    colorlinks=true,       % false: boxed links; true: colored links
    linkcolor=red,          % color of internal links
    citecolor=blue,        % color of links to bibliography
    filecolor=magenta,      % color of file links
    urlcolor=cyan           % color of external links
}

\newtheorem{theorem}{Theorem}[section]

\newtheorem{remark}{Remark}[section]
\newtheorem{lemma}{Lemma}[section]

\numberwithin{equation}{section} \numberwithin{table}{section}
\numberwithin{figure}{section}
\numberwithin{algorithm}{section}

\DeclareMathOperator*{\esssup}{ess\,sup}
\everymath{\displaystyle}

\begin{document}

\title{Robust Preconditioners for Incompressible MHD Models}

\author{Yicong Ma}
\address{Department of Mathematics,The Pennsylvania State University, University Park, PA 16802, USA}
\email{yxm147@psu.edu}
%\thanks{}

%    author two information
\author{Kaibo Hu}
\address{Beijing International Center for Mathematical Research , Peking University, Beijing 100871, P. R. China}
\email{hukaibo02@gmail.com}
%\thanks{}

%    author three information
\author{Xiaozhe Hu}
\address{Department of Mathematics, Tufts University, Medford, MA 02155, USA}
\email{xiaozhe.hu@tufts.edu}
%\thanks{}

%    author four information
\author{Jinchao Xu}
\address{Department of Mathematics,The Pennsylvania State University, University Park, PA 16802, USA}
\email{xu@math.psu.edu}
%\thanks{}

\date{}

\begin{abstract}
In this paper, we develop two classes of robust preconditioners for the structure-preserving discretization of the incompressible magnetohydrodynamics (MHD) system. By studying the well-posedness of the discrete system, we design block preconditioners for them and carry out rigorous analysis on their performance. We prove that such preconditioners are robust with respect to most physical and discretization parameters. In our proof, we improve the existing estimates of the block triangular preconditioners for saddle point problems by removing the scaling parameters, which are usually difficult to choose in practice. This new technique is not only applicable to the MHD system, but also to other problems. Moreover,  we prove that Krylov iterative methods with our preconditioners preserve the divergence-free condition exactly, which complements the structure-preserving discretization. Another feature is that we can directly generalize this technique to other discretizations of the MHD system. We also present preliminary numerical results to support the theoretical results and demonstrate the robustness of the proposed preconditioners.
\end{abstract}


\maketitle

%\tableofcontents

%\newpage

% !TEX root = mhd.tex
% This tex file contains all the content of the paper.

% introduction.
\section{Introduction}

The incompressible Magnetohydrodynamics (MHD) system models the interactions between electromagnetic fields and conducting fluids. It consists of the incompressible Navier-Stokes equation for the fluids and the (reduced) Maxwell's equation for the electro-magnetic fields.  MHD systems of different scales are used in different fields,
such as astrophysics, engineering related to liquid metal, controlled thermonuclear fusion.  There is a vast literature on the study of various aspects of MHD systems. In this work, we concentrate on the following incompressible MHD system in both $2$D and $3$D.  We assume that {\small $\Omega \subset \mathbb{R}^2$} or {\small $\Omega \subset \mathbb{R}^3$} is a simply connected bounded domain with a Lipchitz boundary. In $3$D case, the model is 
\begin{small}
\begin{align}
& \frac{\partial \bm{u}}{\partial t} 
+ ( \boldsymbol{u} \cdot \nabla) \boldsymbol{u}
- \frac{1}{Re} \Delta \boldsymbol{u}
- s \bm{j} \times \boldsymbol{B}
+ \nabla p
= \boldsymbol{f} , 
\label{eq:dimensionless1} \\
& \frac{\partial \bm{B}}{\partial t}
+ \nabla \times \boldsymbol{E} = \boldsymbol{0} , 
\label{eq:dimensionless2} \\
& \bm{j} - \frac{1}{Rm} \nabla \times \mu_{r}^{-1} \boldsymbol{B}
 = \boldsymbol{0} , 
 \label{eq:dimensionless3} \\
& \sigma_{r} ( \bm{E} + \bm{u} \times \bm{B} ) = \bm{j}, 
\label{eq:dimensionless4} \\
& \nabla \cdot \boldsymbol{u} = 0.
\label{eq:dimensionless5}
\end{align}
\end{small}
Here, {\small $\boldsymbol{u}$} is the velocity of fluid, {\small $p$} is the pressure, {\small $\boldsymbol{B}$} is the magnetic field, {\small $\boldsymbol{j}$} is the volume current density, and {\small $\boldsymbol{E}$} is the electric field.  The physical parameters are the fluid Reynolds number {\small $Re$}, the magnetic Reynolds number {\small $Rm$}, the coupling number $s$, the relative electrical conductivity {\small $\sigma_{r}$}, and the relative magnetic permeability {\small $\mu_{r}$}. The initial conditions for the fluid velocity, and the magnetic field are
\begin{small}
\begin{align*}
& \bm{u}( \bm{x}, 0 ) = \bm{u}_{0}(\bm{x}),  
\quad
\bm{B}( \bm{x}, 0 ) = \bm{B}_{0}(\bm{x}), 
\quad 
\forall \bm{x} \in \Omega.
\end{align*}
\end{small}
And the boundary conditions are
\begin{small}
\begin{align*}
& \bm{u} = 0,  
\quad
\bm{n} \cdot \bm{B} = 0, 
\quad
\bm{n} \times \bm{E} = \bm{0}, 
\quad
\forall \bm{x} \in \partial \Omega, 
\quad t > 0.
\end{align*}
\end{small}
The primary unknown physical variables in the model are {\small $\boldsymbol{u}$}, {\small $p$} and {\small $\boldsymbol{B}$}. These quantities, once known, uniquely determine {\small $\boldsymbol{E}$} and {\small $\boldsymbol{j}$}. To discretize this system, we follow the scheme proposed in our previous work \cite{Hu.K;Ma.Y;Xu.J.2014a}, which ensures {\small $\bm{B}$} is divergence-free exactly on the discrete level (Such property is referred to as structure-preserving in the following sections). Therefore, we solve {\small $\boldsymbol{u}$}, {\small $p$}, {\small $\boldsymbol{B}$}, and {\small $\boldsymbol{E}$} simultaneously. 

%There is also incompressible MHD model in $2D$:
In $2$D case, the MHD model \eqref{eq:dimensionless1}-\eqref{eq:dimensionless5} becomes: 
\begin{small}
\begin{align}
& \frac{\partial \boldsymbol{u}}{\partial t} 
+ ( \boldsymbol{u} \cdot \nabla) \boldsymbol{u}
- \frac{1}{Re} \Delta \boldsymbol{u}
- s j \times \bm{B}
+ \nabla p
= \boldsymbol{f} , 
\label{eq:dimensionless-2D1} \\
& \frac{\partial \boldsymbol{B}}{\partial t}
+ \mathrm{curl} E = \boldsymbol{0} , 
\label{eq:dimensionless-2D2} \\
& j - \frac{1}{Rm} \mathrm{rot} ~ \mu_{r}^{-1} \boldsymbol{B}
 = \boldsymbol{0} , 
 \label{eq:dimensionless-2D3} \\
& \sigma_{r} ( E + \bm{u} \times \bm{B} )  = j, 
\label{eq:dimensionless-2D4} \\
& \nabla \cdot \boldsymbol{u} = 0.
\label{eq:dimensionless-2D5}
\end{align}
\end{small}
Here, {\small $\mathrm{rot} \bm{u} = \frac{\partial u_{2} }{\partial x} - \frac{\partial u_{1}}{\partial y}$} for any vector {\small $\bm{u} = \left( u_{1}, u_{2} \right)^{T}$}, and {\small $\mathrm{curl} u = \left( \frac{\partial u}{\partial y}, - \frac{\partial u}{\partial x} \right)^{T}$} for any scalar {\small $u$}.

Note that we can directly apply the analysis of $3$D model to $2$D case because we can write velocity {\small $\bm{u} = ( u_{1}, u_{2} )^{T} \in \mathbb{R}^{2}$} as {\small $\bm{u} = ( u_{1}, u_{2}, 0 )^{T} \in \mathbb{R}^{3}$}, the magnetic field {\small $\bm{B}=(B_1, B_2)^T$} as {\small $\bm{B} = (B_1, B_2, 0)^T$}, and the electric field {\small $E$} as {\small $\bm{E} = (0, 0, E)^{T} \in \mathbb{R}^{3}$}. The cross product and derivative operators are all well-defined by rewriting those $2$D variables in the $3$D fashion. Therefore, we focus on the analysis in $3D$ case in the rest of the paper.  

For the MHD system, solving the linear systems obtained after linearization is usually the most challenging and time-consuming part in the overall simulation, which is due to the large-scale, multi-physical, and indefinite properties of the resulting linear systems.  In order to improve the efficiency of the numerical simulations, there have been a lot of studies on the development of efficient solvers for various MHD systems.  

% block preconditioners.
% Chacon: physics-based preconditioning.
Due to the block structure of the resulting linear systems, many block preconditioners have been developed in the literature for the MHD system. 
% Shadid.
Shadid and his collaborators have developed a series of robust and scalable Newton-Krylov solvers for the MHD system \cite{Shadid.J;Cyr.E;Pawlowski.R;Tuminaro.R;Chacon.L;Lin.P.2010a,Shadid.J;Pawlowski.R;Banks.J;Chacon.L;Lin.P;Tuminaro.R.2010a,Cyr.E;Shadid.J;Tuminaro.R;Pawlowski.R;Chacon.L.2013a,Phillips.E;Elman.H;Cyr.E;Shadid.J;Pawlowski.R.2014a}. In \cite{Shadid.J;Pawlowski.R;Banks.J;Chacon.L;Lin.P;Tuminaro.R.2010a}, they propose a robust, efficient, fully-coupled stabilized finite element formulation for resistive MHD, which enables both fully-implicit and direct-to-steady-state solutions. They investigate the performance of one-level Schwarz method and also a new fully coupled algebraic multilevel method in that paper. In \cite{Shadid.J;Cyr.E;Pawlowski.R;Tuminaro.R;Chacon.L;Lin.P.2010a,Cyr.E;Shadid.J;Tuminaro.R;Pawlowski.R;Chacon.L.2013a,Phillips.E;Elman.H;Cyr.E;Shadid.J;Pawlowski.R.2014a}, they explore a class of robust and scalable parallel preconditioners for Newton-Krylov solver based on the physical-based approximate block factorization (ABF) technique. They employ block factorization and approximate the resulting Schur complement recursively based on special techniques, for example, operator commutativity \cite{Elman.H;Howle.V;Shadid.J;Shuttleworth.R;Tuminaro.R.2006a}. Numerical experiments and benchmark tests demonstrate the efficiency and scalability of their preconditioners.
Chac\'{o}n and his collaborators also contribute to developing ``physics-based" block preconditioning strategies for fully-implicit Newton-Krylov solvers for MHD system \cite{Chacon.L;Knoll.D;Finn.J.2002a,Chacon.L;Knoll.D.2003a,Chacon.L.2008b,Chacon.L.2008a}. They use physical-based ABF approach to design preconditioners for the linearized MHD system. With algebraic techniques and recursive approximations of the Schur complements,  they successfully convert the complicated problems into several Poisson-like equations and design efficient ABF preconditioners for the linearized MHD system. The implementation of ABF preconditioners consists of solving a sequence of Poisson-like equations, for which multigrid (MG) methods, especially algebraic multigrid (AMG) methods, can be effectively applied.  Numerical experiments and benchmark tests demonstrate the efficiency and scalability of ABF preconditioners. Moreover, T\'{o}th et. al \cite{Toth.G;Keppens.R;Botchev.M.1998a,Keppens.R;Toth.G;Botchev.M.1999a} use a block incomplete LU (ILU) factorization to precondition the MHD system. And Badia et. al \cite{Badia.S;Martin.A;Planas.R.2014a} propose a recursive version block ILU preconditioner recently.

% Additive Schwarz.  Efstathiou.E;Gander.M.2003a Tuminaro.R;Tong.C;Shadid.J;Devine.K.2002a
In addition to the block preconditioners, there are also many works on other preconditioning strategies for the MHD system, such as the additive Schwarz methods \cite{Quarteroni.A;Valli.A.1999a,Cai.X;Sarkis.M.1999a,Ovtchinnikov.S;Dobrain.F;Cai.X;Keyes.D.2007a,Reynolds.D;Samtaney.R;Tiedeman.H.2012a}, 
% Operator splitting.
Operator splitting method \cite{Reynolds.D;Samtaney.R;Woodward.C.2010a,Reynolds.D;Samtaney.R;Tiedeman.H.2012a}.
% Reynolds et. al propose a preconditioner based on splitting the operator by coordinate directions, which is especially attractive in the structured grid calculations. 

In this paper, we develop robust block preconditioners especially for the linearized system arising from the structure-preserving discretizations \cite{Hu.K;Ma.Y;Xu.J.2014a}. We precondition it by converting coupled MHD systems into subsystems for which effective preconditioners exist. Different from the aforementioned preconditioners that have been studied mostly from algebraic point of view, our preconditioners are motivated from the perspective of functional and PDE analysis following a framework summarized by Mardal and Winther in \cite{Mardal.K;Winther.R.2011a}. In essence, we study the mapping property of the linearized operator between appropriate Sobol\'{e}v spaces equipped with carefully chosen norms. So we can derive robust block diagonal preconditioners based on proper norms straightforwardly. Such block diagonal preconditioners are often known as norm-equivalent preconditioners and use them in combination with minimal residual (MINRES) method. Moreover, we can design block triangular preconditioners, and theoretically prove that they are Field-of-values- (FOV-) equivalent preconditioners based on the mapping properties \cite{Loghin.D;Wathen.A.2004a}. And we use them in combination with general minimal residual (GMRES) method.

In the analysis of FOV-equivalent preconditioners for saddle point problems, we improve the estimates by Loghin and Wathen in \cite{Loghin.D;Wathen.A.2004a} by removing scaling parameters in front of the diagonal blocks. It is observed that such scaling parameters are difficult to choose and unnecessary in practice. By choosing appropriate norms in the analysis, we are able to get rid of these scaling parameters, which is consistent with the practical implementations and observations.  While this new technique is originally motivated for FOV-equivalent preconditioners for the MHD systems, it is expected to be applicable to other saddle point type problems. 
  
One special feature of our work is that we pay special attention to the structure-preserving property. We design our preconditioners in such a way that the resulting preconditioned Krylov iterative methods inherit this property even if the linear system is solved inexactly.  This feature makes our preconditioner structure-preserving and suitable for accurate and efficient MHD simulations. 

The efficiency of our preconditioners depends on how the several relevant subsystems are solved. The subsystem for velocity field is Poisson-like, for which multigrid methods are effective preconditioners. The subsystem for pressure is well-conditioned and hence it can be easily solved. For the electric and magnetic fields, we need to solve subsystems involving {\small $\mathrm{curl} \ \mathrm{curl}$} and {\small $\mathrm{grad} \ \mathrm{div}$} operators. We adopt the HX-preconditioner \cite{Hiptmair.R;Xu.J.2007a}, which is developed based on the auxiliary space preconditioning \cite{Xu.J.1996a}. Taking advantage of those efficient sub-problem solvers, we develop practical and scalable preconditioners for the structure-preserving discretization of the MHD system.    

The rest of the paper is organized as follows. We revisit the structure-preserving finite element discretization introduced in \cite{Hu.K;Ma.Y;Xu.J.2014a} in \S \ref{sec:magnetohydrodynamics_model} and \S\ref{sec:fem_discretization}. We carry out the analysis using different weighted norms to ensure the robustness of preconditioners with respect to the physical and discretization parameters. Then we propose and analyze these preconditioners in \S \ref{sec:uniform_pc} and discuss their generalizations to other discretization schemes in \S \ref{sec:general}.  Finally, we present results of numerical experiments in \S \ref{sec:numeric_experiments} to demonstrate the robustness of these new preconditioners. 

% model.
\section{Magnetohydrodynamics model}\label{sec:magnetohydrodynamics_model}
Following \cite{Hu.K;Ma.Y;Xu.J.2014a}, we use the following set of notation.  First, {\small $(\cdot, \cdot)$} and {\small$\Vert \cdot \Vert$} denotes {\small $L^{2}$} inner product and {\small $L^2$} norm
\begin{small}
$$
(u,v)=\int_{\Omega}u\cdot v \mathrm{d}x,
\quad
\|u\| =\left(\int_{\Omega}|u|^2 \mathrm{d}x\right)^{1/2} = \sqrt{(u,u)}. 
$$
\end{small}
With a slight abuse of notation, we use {\small $L^2(\Omega)$} to denote
both the scalar and vector {\small $L^2$} space.  Given a linear operator {\small $D$}, we define
\begin{small}
$$
H(D,\Omega) = \left\{v\in L^2(\Omega), ~ Dv\in L^2(\Omega) \right\} ,
$$
\end{small}
and 
\begin{small}
$$
H_0(D,\Omega) = \left\{v\in H(D, \Omega), ~ t_{D}v=0 \mbox{ on } \partial\Omega \right\}. 
$$
\end{small}
Here, {\small $t_{D}$} is the trace operator defined by
\begin{small}
$$
t_{D}v=
\left\{
  \begin{array}{cc}
    v, & D=\mathrm{grad},\\
    v\times n, & D=\mathrm{curl},\\
    v\cdot n, & D=\mathrm{div},
  \end{array}
\right.
$$
\end{small}
where $n$ is the outer normal direction of {\small $\partial \Omega$}. We note that {\small $L^2(\Omega)$} can be viewed as {\small $H(id,\Omega)$} where {\small $id$} denotes the identity operator and we
often use the following notation:
\begin{small}
$$
L^2_0(\Omega) = \left\{v\in L^2(\Omega), ~ \int_\Omega v=0 \right\}.
$$
\end{small}
When {\small $D=\mathrm{grad}$}, we often use the notation:
\begin{small}
$$
H^1(\Omega)=H(\mathrm{grad}, \Omega), \quad
H^1_0(\Omega)=H_0(\mathrm{grad}, \Omega).
$$
\end{small}
We also use the space {\small $L^p$} and {\small $H^{-1}$} with their canonical norms
\begin{small}
$$
\|v\|_{0,p}= \left( \int_\Omega|v|^p \right)^{1/p},
\quad
\|v\|_{0,\infty}= \esssup_{x \in \Omega} \lvert v(x) \rvert ,
\quad \|v\|_{H^{-1}}=\sup_{\phi\in
  H^1_0(\Omega)}\frac{( v, \phi )}{\|\nabla\phi\|},
$$
\end{small}
% where {\small $\langle \cdot, \cdot \rangle$} is the duality pair between the functional space and its dual space.
As we will see later, it is convenient to introduce the following spaces
\begin{small}
\begin{align*} %\label{space}
& \boldsymbol{X} = H_{0}^{1}(\Omega)^{3}\times {H}_{0}(\mathrm{div}; \Omega)\times
{H}_{0}(\mathrm{curl}; \Omega), ~ Q=L^2_0(\Omega), \\
& \boldsymbol{V} = H_{0}^{1}(\Omega)^{3}, ~
\boldsymbol{V}^{d} = H_{0}(\mathrm{div}; \Omega) \mbox{ and }
\boldsymbol{V}^{c} = H_{0}(\mathrm{curl}; \Omega), \\
& \bm{V}^{d,0} = H_0(\mathrm{div}0, \Omega)=\{ \bm{C} \in \bm{V}^d,~ \nabla \cdot \bm{C} = 0 \}.
\end{align*}
\end{small}
We use {\small $\boldsymbol{W}^{\ast}$} ({\small $\boldsymbol{W} = \boldsymbol{V}$}, {\small $\boldsymbol{V}^{d}$} or {\small $\boldsymbol{V}^{c}$}) to denote the dual space of {\small $\boldsymbol{W}$}, and {\small $\bm{W}_{h}$} the corresponding finite element space of {\small $\bm{W}$}.  Moreover, we assume that both {\small $\mu_{r}$} and {\small $\sigma_{r}$} are positive continuous functions only depending on {\small $x \in \Omega$}, which induce weighted {\small $L^{2}$}-norms
\begin{small}
\begin{align*}
& \Vert \boldsymbol{x} \Vert_{\sigma_{r}}^{2} 
= ( \sigma_{r} \boldsymbol{x}, \boldsymbol{x} ), ~
 \Vert \boldsymbol{x} \Vert_{\mu_{r}^{-1}}^{2}
= ( \mu_{r}^{-1} \boldsymbol{x}, \boldsymbol{x} ).
\end{align*}
\end{small}
For the sake of convenience, we also define a tri-linear form
\begin{small}
\begin{align*}
& d( \bm{w}, \bm{u}, \bm{v} ) = 
\frac{1}{2} \left[ (\boldsymbol{w}\cdot \nabla \boldsymbol{u}, \boldsymbol{v})
  - ( \boldsymbol{w}\cdot \nabla 
\boldsymbol{v}, \boldsymbol{u}) \right].
\end{align*}
\end{small}
%
%\vskip-10pt
Based on these notation, we consider the variational formulation for the incompressible MHD system \eqref{eq:dimensionless1}-\eqref{eq:dimensionless5}: Find {\small $(\boldsymbol{u}, \boldsymbol{B}, \boldsymbol{E})\in \boldsymbol{X}$} and {\small $ p\in Q$} such that for any {\small $(\boldsymbol{v}, \boldsymbol{C}, \boldsymbol{F})\in \boldsymbol{X} $} and {\small $q\in Q$},
\begin{small}
\begin{align}
\begin{cases}
&  \left( \frac{\partial \boldsymbol{u}}{\partial t},\boldsymbol{v} \right) 
 + d( \bm{u}, \bm{u}, \bm{v} )
+ k^{-1} ( \nabla \cdot \boldsymbol{u}, \nabla \cdot \boldsymbol{v} ) 
+ \frac{1}{Re} (\nabla \boldsymbol{u}, \nabla \boldsymbol{v}) \\
& \qquad
- s ( \sigma_{r} \boldsymbol{E}\times \boldsymbol{B},\boldsymbol{v} ) 
+ s ( \sigma_{r} \boldsymbol{u} \times \boldsymbol{B}, \boldsymbol{v} \times \boldsymbol{B} ) 
- (p,\nabla\cdot \boldsymbol{v}) 
 = (\boldsymbol{f},\boldsymbol{v}), \\
& \left( \mu_{r}^{-1} \frac{\partial \boldsymbol{B}}{\partial t}, \boldsymbol{C} \right) 
 + ( \mu_{r}^{-1} \nabla\times \boldsymbol{E}, \boldsymbol{C}) = 0,   \\
& ( \sigma_{r} \boldsymbol{E},\boldsymbol{F}) 
 + ( \sigma_{r} \boldsymbol{u} \times \boldsymbol{B}, \boldsymbol{F} )
- \frac{1}{Rm} (\mu_{r}^{-1} \boldsymbol{B}, \nabla\times \boldsymbol{F}) 
 = 0,  \\
& (\nabla\cdot \boldsymbol{u}, q)  =  0.
\end{cases}\label{eq:variational_form}
\end{align}
\end{small}

\begin{remark}
Note that the special treatment of the nonlinear convection term in \eqref{eq:variational_form} is based on the following identity, i.e. if {\small $\nabla\cdot\boldsymbol{u}=0$} and {\small $\bm{u} = 0$} on {\small $\partial \Omega$}, 
\begin{small}
\begin{equation}\label{modified-convection}
( \boldsymbol{u}\cdot \nabla \boldsymbol{u}, \boldsymbol{v})
= d( \bm{u}, \bm{u}, \bm{v} ) ,
\end{equation}
\end{small}
This is a classical stabilization technique, c.f. \cite{Teman.R.1977a}.
\end{remark}

% fem
\section{Finite element discretization for the MHD system}\label{sec:fem_discretization}
In this section, we revisit the structure-preserving discretization of the incompressible MHD system \cite{Hu.K;Ma.Y;Xu.J.2014a}. For the temporal discretization, we adopt the backward Euler scheme. And similar spatial discretization is also applicable to other temporal discretization schemes, such as Crank-Nicolson or backward differentiation formula (BDF). We first introduce the full discretizations and then revisit the well-posedness of the linearized problem with different weighted norms from \cite{Hu.K;Ma.Y;Xu.J.2014a}.

\subsection{Full discretization scheme} 
Before going into the details of discretizations, we first introduce the finite element spaces for the velocity {\small $\boldsymbol{u}$}, the pressure {\small $p$}, the magnetic field {\small $\boldsymbol{B}$}, and the electric field {\small $\boldsymbol{E}$}, respectively.  

For the velocity and pressure, we choose a standard stable Stokes pair such that {\small $\boldsymbol{V}_h \subset H_{0}^{1}(\Omega)^3$}, {\small $Q_h \subset L_{0}^{2}(\Omega)$}, and it satisfies the well-known inf-sup condition:
\begin{small}
\begin{equation}
  \label{inf-sup}
\inf_{q_h\in  Q_h}\sup_{\boldsymbol{v}_h\in \boldsymbol{V}_h}\frac{(\nabla\cdot \boldsymbol{v}_h,q_h)}{\|\nabla \boldsymbol{v}_h\|\;\|q_h\|}\ge \beta>0,
\end{equation}
\end{small}
where the positive constant {\small $\beta$} is independent of mesh size {\small $h$}.  Many stable Stokes pairs are available, such as Taylor-Hood element \cite{Boffi.D;Brezzi.F;Fortin.M.2013a}.  For the magnetic field, we use Raviart-Thomas elements and denote the finite element space as  {\small $\boldsymbol{V}_h^d \subset H_0(\mathrm{div}; \Omega)$}. And we employ N\'{e}d\'{e}lec edge elements to discretize the electric field and denote the finite element space as {\small $\boldsymbol{V}^c_{h} \subset H_0(\mathrm{curl}; \Omega)$}. Therefore, we can define the finite element space of {\small $\bm{X}$} 
\begin{small}
\begin{equation*}
\boldsymbol{X}_h  =  \boldsymbol{V}_h \times \boldsymbol{V}_h^d \times \boldsymbol{V}_h^c.
\end{equation*}
\end{small}
Based on the above finite element spaces, we have the following full discretization scheme based on the Picard linearization. We discretize the convection term explicitly because it leads to a symmetric linear system, which facilitates the analysis of solvers. The solvers proposed in this paper also work for other discretizations in \cite{Hu.K;Ma.Y;Xu.J.2014a}.

\textbf{Symmetric Picard linearization.} Find {\small $
(\boldsymbol{u}_h^{n},\boldsymbol{B}_h^{n},\boldsymbol{E}_h^{n},p_h^{n})\in \boldsymbol{X}_h \times Q_h
$} such that for any {\small $(\boldsymbol{v}_h,\boldsymbol{C}_h,\boldsymbol{F}_h, q_h)\in \boldsymbol{X}_h \times Q_h$},
\begin{small}
\begin{align}
\begin{cases}
& k^{-1} \left( \boldsymbol{u}_h^{n}-\boldsymbol{u}_h^{n-1}, \boldsymbol{v}_{h} \right) 
+  d( \bm{u}_h^{n-1},  \bm{u}_h^{n-1},  \bm{v}_h ) 
+ k^{-1} ( \nabla \cdot \bm{u}_{h}, \nabla \cdot \bm{v}_{h} ) \\
& \qquad
+ {1 \over Re} \left(\nabla \boldsymbol{u}_h^{n}, \nabla \boldsymbol{v}_h \right) 
 - s  \left(\boldsymbol{j}^{n}_{h,n-1}\times \boldsymbol{B}_h^{n-1},\boldsymbol{v}_h \right) 
 - \left( p_h^{n},\nabla\cdot \boldsymbol{v}_h \right) 
= \left( \boldsymbol{f}_h^{n},\boldsymbol{v}_h \right),   \\
& - k^{-1} \alpha \left( \mu_{r}^{-1} \left( \boldsymbol{B}_h^{n} - \boldsymbol{B}_h^{n-1} \right), 
\boldsymbol{C}_{h} \right) 
- \alpha \left( \mu_{r}^{-1} \nabla \times \boldsymbol{E}_h^{n}, \boldsymbol{C}_{h} \right) = 0,  \\
& s \left( \boldsymbol{j}^{n}_{h,n-1},\boldsymbol{F}_h \right) 
- \alpha \left( \mu_{r}^{-1} \boldsymbol{B}_h^{n}, \nabla\times \boldsymbol{F}_h \right) =0, \\
& \left( \nabla\cdot \boldsymbol{u}_h^{n}, q_h \right) =0.
\end{cases}\label{eq:fully_picard1}
\end{align}
\end{small}
where {\small $\boldsymbol{j}^{n}_{h, n-1} =   
\sigma_{r} ( \boldsymbol{E}_h^{n}+\boldsymbol{u}_h^{n}\times \boldsymbol{B}_h^{n-1} ) $}, and {\small $\alpha = s/Rm$}.  Here, the second equation is multiplied by $-\alpha$ and the third equation is multiplied by $s$.  This is because we would like to make the resulting linear system symmetric. The above discretization has nice properties, for example, energy estimate, structure-preserving, as analyzed in \cite{Hu.K;Ma.Y;Xu.J.2014a}.

\subsection{Well-posedness} 
Now we discuss about the well-posedness of scheme \eqref{eq:fully_picard1}, which is the foundation of the preconditioners we propose.  

For the sake of simplicity, we rewrite {\small $\boldsymbol{x}_{h}^{n-1}$} ( {\small $\boldsymbol{x} = \boldsymbol{u}$}, {\small $\boldsymbol{B}$}, {\small $\boldsymbol{E}$} or {\small $p$}, so is the {\small $\boldsymbol{x}$} mentioned afterward in this paragraph) as {\small $\boldsymbol{x}^{-}$} and {\small $\boldsymbol{x}_{h}^{n}$} as {\small $\boldsymbol{x}$}.  We keep the subscript {\small $h$} for the finite element spaces in consideration of clarity.  Then we can write the symmetric Picard linearization as: Find {\small $( \boldsymbol{u}, \boldsymbol{B}, \boldsymbol{E}, p ) \in \boldsymbol{X}_{h} \times Q_{h}$} such that for any {\small $( \boldsymbol{v}, \boldsymbol{C}, \boldsymbol{F}, q) \in \boldsymbol{X}_{h} \times Q_{h}$},
\begin{small}
\begin{align}
\begin{cases}
&k^{-1} ( \boldsymbol{u}, \boldsymbol{v} )
+ Re^{-1} ( \nabla \boldsymbol{u}, \nabla \boldsymbol{v} )
+ k^{-1} ( \nabla \cdot \boldsymbol{u}, \nabla \cdot \boldsymbol{v} ) 
- s ( \sigma_{r} \boldsymbol{E} \times \boldsymbol{B}^{-}, \boldsymbol{v} ) \\
& \qquad
+ s ( \sigma_{r} \boldsymbol{u} \times \boldsymbol{B}^{-}, \boldsymbol{v} \times \boldsymbol{B}^{-} ) 
- ( p, \nabla \cdot \boldsymbol{v} ) 
= ( \widetilde{ \boldsymbol{f} }, \boldsymbol{v}), \\
& - k^{-1} \alpha ( \mu_{r}^{-1} \boldsymbol{B}, \boldsymbol{C}) 
- \alpha ( \mu_{r}^{-1} \nabla \times \boldsymbol{E}, \boldsymbol{C})
= - k^{-1} \alpha ( \mu_{r}^{-1} \boldsymbol{B}^{-}, \boldsymbol{C}), \\
& s ( \sigma_{r} \boldsymbol{E}, \boldsymbol{F})
+ s ( \sigma_{r} \boldsymbol{u} \times \boldsymbol{B}^{-}, \boldsymbol{F} )
- \alpha (\mu_{r}^{-1} \boldsymbol{B}, \nabla \times \boldsymbol{F}) 
=  \boldsymbol{0}, \\
& (\nabla \cdot \boldsymbol{u}, q) = 0,
\end{cases} \label{eqn:symmetric-picard}
\end{align}
\end{small}
with
{\small $\widetilde{\boldsymbol{f} } = \boldsymbol{f}
	+ k^{-1}  \boldsymbol{u}^{-}
	- d( \bm{u}^{-}, \bm{u}^{-}, \bm{v} )$}.
For {\small $\boldsymbol{\xi} = (\boldsymbol{u}, \boldsymbol{B}, \boldsymbol{E}) \in \bm{X}_{h}$}, {\small $\boldsymbol{\eta} = (\boldsymbol{v}, \boldsymbol{C}, \boldsymbol{F}) \in \bm{X}_{h}$}, and {\small $p, ~q \in Q_{h}$}, we define a bilinear form {\small $\boldsymbol{a}_{0}(\cdot, \cdot)$} on {\small $\boldsymbol{X}_{h} \times \boldsymbol{X}_{h}$} and {\small $\boldsymbol{b}(\cdot, \cdot)$} on {\small $\boldsymbol{X}_{h} \times Q_{h}$} by
\begin{small}
\begin{align*}
\boldsymbol{a}_{0} (\boldsymbol{\xi}, \boldsymbol{\eta}) & = 
k^{-1} ( \boldsymbol{u}, \boldsymbol{v} )
+ Re^{-1} ( \nabla \boldsymbol{u}, \nabla \boldsymbol{v} )
+ k^{-1} ( \nabla \cdot \boldsymbol{u}, \nabla \cdot \boldsymbol{v} )
- s ( \sigma_{r} \boldsymbol{E} \times \boldsymbol{B}^{-}, \boldsymbol{v} ) \\
& 
+ s ( \sigma_{r} \boldsymbol{u} \times \boldsymbol{B}^{-}, \boldsymbol{v} \times \boldsymbol{B}^{-} ) 
- k^{-1} \alpha (\mu_{r}^{-1} \boldsymbol{B}, \boldsymbol{C}) 
- \alpha ( \mu_{r}^{-1} \nabla \times \boldsymbol{E}, \boldsymbol{C}) \\
& 
+ s ( \sigma_{r} \boldsymbol{E}, \boldsymbol{F})
+ s ( \sigma_{r} \boldsymbol{u} \times \boldsymbol{B}^{-}, \boldsymbol{F} )
- \alpha (\mu_{r}^{-1} \boldsymbol{B}, \nabla \times \boldsymbol{F}),
\end{align*}
\end{small}
and
\begin{small}
$$
	\boldsymbol{b}( \boldsymbol{\eta}, q ) = ( \nabla \cdot \boldsymbol{v} , q ).
$$
\end{small}
Therefore, we can write \eqref{eqn:symmetric-picard} as: Find {\small $\boldsymbol{\xi} \in \boldsymbol{X}_{h}$} and {\small $p \in Q_{h}$} such that
\begin{small}
\begin{align}
\begin{cases}
& \boldsymbol{a}_{0} (\boldsymbol{\xi}, \boldsymbol{\eta}) + 
\boldsymbol{b} (\boldsymbol{\eta}, p) 
= \langle \boldsymbol{h}, \boldsymbol{\eta} \rangle, 
~ \forall \eta \in \boldsymbol{X}_{h},  \\
& \boldsymbol{b}(\boldsymbol{\xi}, q) = \langle g, q \rangle,
~ \forall q \in Q_{h}. 
\end{cases}\label{eqn:problem-a0-1}
\end{align}
\end{small}
In order to analyze the well-posedness of this problem, we analyze an auxiliary problem first. Define a bilinear form 
\begin{small}
\begin{align*}
\boldsymbol{a} (\boldsymbol{\xi}, \boldsymbol{\eta}) & = 
a_{0} ( \bm{\xi}, \bm{\eta} )
- \alpha ( \mu_{r}^{-1} \nabla \cdot \boldsymbol{B}, \nabla \cdot \boldsymbol{C} ).
\end{align*}
\end{small}
The corresponding auxiliary problem is: Find {\small $\boldsymbol{\xi} \in \boldsymbol{X}_{h}$} and {\small $p \in Q_{h}$} such that
\begin{small}
\begin{align}
\begin{cases}
& \boldsymbol{a} (\boldsymbol{\xi}, \boldsymbol{\eta}) + 
\boldsymbol{b} (\boldsymbol{\eta}, p) 
= \langle \boldsymbol{h}, \boldsymbol{\eta} \rangle, 
~ \forall \eta \in \boldsymbol{X}_{h},  \\
& \boldsymbol{b}(\boldsymbol{\xi}, q) = \langle g, q \rangle,
~ \forall q \in Q_{h}. 
\end{cases}\label{eqn:problem-aux-1}
\end{align}
\end{small}

As shown in \cite{Hu.K;Ma.Y;Xu.J.2014a}, the mixed formulations \eqref{eqn:problem-a0-1} and \eqref{eqn:problem-aux-1} are equivalent if {\small $\boldsymbol{h} = ( \boldsymbol{f}, \boldsymbol{l}, \boldsymbol{r} ) \in \boldsymbol{V}_{h}^{\ast} \times [ \boldsymbol{V}_{h}^{d} ]^{\ast} \times [\boldsymbol{V}_{h}^{c}]^{\ast} $} such that {\small $
	\langle \boldsymbol{l}, \boldsymbol{C} \rangle 
	= (\boldsymbol{l}_{R}, \boldsymbol{C} ), ~ 
	\forall \boldsymbol{C} \in \boldsymbol{V}^{d}$} for some {\small $ \boldsymbol{l}_{R} \in \boldsymbol{V}^{d, 0}$}. The well-posedness of the mixed formulation \eqref{eqn:problem-a0-1} follows directly from that of the auxiliary problem \eqref{eqn:problem-aux-1}. Therefore, we focus on proving the well-posedness of the auxiliary problem.  

As usual, we use Brezzi's theorem to analyze the well-posedness of the above auxiliary problem.  As discussed in \cite{Mardal.K;Winther.R.2011a}, ensuring that the constants appearing in the inf-sup conditions are independent of the physical and discretize parameters is crucial to design robust block preconditioners for coupled systems.  Therefore, we introduce the weighted norms
\begin{small}
\begin{align}
& \Vert ( \boldsymbol{v}, \boldsymbol{C}, \boldsymbol{F} ) \Vert_{\boldsymbol{X}}^{2} 
= \Vert \boldsymbol{v} \Vert_{\mathcal{H}_{1} }^{2} 
+ \Vert \boldsymbol{C} \Vert_{\mathcal{H}_{3} }^{2}
+ \Vert \boldsymbol{F} \Vert_{\mathcal{H}_{ 4 } }^{2}
,~
\Vert q \Vert_{Q } = \Vert q \Vert_{\mathcal{H}_{2} }, 
\nonumber \\
& \Vert ( \bm{v}, q, \bm{C}, \bm{F} ) \Vert^{2}_{\mathcal{H}}
= \Vert \boldsymbol{v} \Vert_{\mathcal{H}_{1} }^{2} 
+ \Vert q \Vert_{\mathcal{H}_{2} }^{2}
+ \Vert \boldsymbol{C} \Vert_{\mathcal{H}_{3}}^{2}
+ \Vert \boldsymbol{F} \Vert_{\mathcal{H}_{ 4 } }^{2}
, \label{eq:Hnorm}
\end{align} 
\end{small}
with
\begin{small}
\begin{align*}
& \Vert \boldsymbol{v} \Vert_{\mathcal{H}_{1} }^{2}
= k^{-1} \Vert \boldsymbol{v} \Vert^{2} 
+ Re^{-1} \Vert \nabla \boldsymbol{v} \Vert^{2}
+ k^{-1} \Vert \nabla \cdot \boldsymbol{v} \Vert^{2}
+ s \Vert \bm{v} \times \bm{B}^{-} \Vert_{\sigma_{r}}^{2}, \\
& \Vert q \Vert^{2}_{\mathcal{H}_{2}} = k \Vert q \Vert^{2}, \\
& \Vert \boldsymbol{C} \Vert_{\mathcal{H}_{3} }^{2}
= k^{-1} \alpha \Vert \boldsymbol{C} \Vert^{2}_{\mu_{r}^{-1}}
+ \alpha \Vert \nabla \cdot \boldsymbol{C} \Vert^{2}_{\mu_{r}^{-1}}, \\
& \Vert \boldsymbol{F} \Vert_{\mathcal{H}_{4} }^{2}
= s \Vert \boldsymbol{F} \Vert^{2}_{\sigma_{r}}
+ k \alpha \Vert \nabla \times \boldsymbol{F} \Vert^{2}_{\mu_{r}^{-1}}, 
\end{align*}
\end{small}
where {\small $\mathcal{H}_{i}$} ({\small $i = 1$}, {\small $2$}, {\small $3$} or {\small $4$}) is a symmetric positive operator (SPD) such that {\small $\Vert \bm{x} \Vert_{\mathcal{H}_{i} }^{2} = ( \mathcal{H}_{i} \bm{x}, \bm{x} )$}.

Results similar to the following theorem can be found in \cite{Hu.K;Ma.Y;Xu.J.2014a}, by introducing more carefully chosen weighted norms, this theorem provides better bounds than the previous one.

\begin{theorem}\label{thm:wellposed_a}
Given {\small $\boldsymbol{h} \in \boldsymbol{X}^{\ast}_{h}$} and {\small $g \in Q^{\ast}_{h}$}, the auxiliary problem \eqref{eqn:problem-aux-1} is well-posed if {\small $k$} is sufficiently small, i.e. {\small $k \leq k_{0}$}, where
\begin{small} 
\begin{align}
& k_{0} = \frac{1}{8 s} \Vert \sqrt{\sigma_{r}} \boldsymbol{B}^{-} \Vert_{0, \infty}^{-2}.
\label{eq:def_k0}
\end{align}
\end{small}
That is, for any {\small $\boldsymbol{h} \in \boldsymbol{X}^{\ast}_{h}$} and {\small $g \in Q^{\ast}_{h}$}, there exists a unique {\small $( \boldsymbol{\xi}, p ) = ( \boldsymbol{u}, \boldsymbol{B}, \boldsymbol{E}, p ) \in \boldsymbol{X}_{h} \times Q_{h}$} that solves the above problem and satisfies:
\begin{small}
$$
	\Vert \boldsymbol{u} \Vert_{\mathcal{H}_{1} } 
	+ \Vert \boldsymbol{B} \Vert_{\mathcal{H}_{3} }
	+ \Vert \boldsymbol{E} \Vert_{\mathcal{H}_{4} }
	+ \Vert p \Vert_{\mathcal{H}_{2} } 
	\leq C\left( \Vert \boldsymbol{h} \Vert_{\boldsymbol{X}^{\ast}} + \Vert g \Vert_{Q^{\ast}} \right),
$$
%\lesssim
where the constant {\small $C$} does not depend on the mesh size $h$, the time  step size $k$, and physical parameters {\small $Rm$}, {\small $s$}, {\small $\mu_r$}, and {\small $\sigma_r$}.
\end{small}
\end{theorem}

Proof of Theorem \ref{thm:wellposed_a} is similar to that in \cite{Hu.K;Ma.Y;Xu.J.2014a}. We step by step prove the following properties.
\begin{enumerate}
\item {\small $\boldsymbol{a} (\cdot, \cdot)$} and {\small $\boldsymbol{b}(\cdot, \cdot)$} are bounded;
\item inf-sup condition holds for {\small $\boldsymbol{b}(\cdot, \cdot)$};
\item inf-sup conditions hold for {\small $\boldsymbol{a}(\cdot, \cdot)$} in the kernel of the operator induced by {\small $\boldsymbol{b}(\cdot, \cdot)$}.
\end{enumerate}

%%% boundedness of a( , ) & b( , ).
\begin{lemma}\label{lemma:boundness}
Both {\small $\boldsymbol{a}(\cdot, \cdot)$} and {\small $\boldsymbol{b}(\cdot, \cdot)$} are bounded operators. That is, 
\begin{small}
\begin{align*}
& \boldsymbol{a} ( \boldsymbol{\xi}, \boldsymbol{\eta} )
\leq C \Vert \boldsymbol{\xi} \Vert_{\boldsymbol{X}}
\Vert \boldsymbol{\eta} \Vert_{\boldsymbol{X}}, \\
& \boldsymbol{b} ( \boldsymbol{\eta}, q )
\leq C \Vert \boldsymbol{\eta} \Vert_{\boldsymbol{X}}
\Vert q \Vert_{Q},
\end{align*} 
\end{small}
where {\small $C$} is a constant independent of the mesh size $h$, the time  step size $k$, and physical parameters {\small $Rm$}, {\small $s$}, {\small $\mu_r$}, and {\small $\sigma_r$}.
\end{lemma}

% proof.
\begin{proof}
It is enough to show that every term in {\small $\boldsymbol{a}(\cdot, \cdot)$} and {\small $\boldsymbol{b}(\cdot, \cdot)$} is bounded.
For the Navier-Stokes equation part,
\begin{small}
$$
	k^{-1} \lvert ( \boldsymbol{u}, \boldsymbol{v} ) \rvert
	+ Re^{-1} \lvert ( \nabla \boldsymbol{u}, \nabla \boldsymbol{v} ) \rvert
	+ k^{-1} \lvert ( \nabla \cdot \bm{u} , \nabla \cdot \bm{v} ) \rvert
	\leq C
	\Vert \boldsymbol{u} \Vert_{\mathcal{H}_{1} } 
	\Vert \boldsymbol{v} \Vert_{\mathcal{H}_{1} }.
$$
\end{small}
with a constant {\small $C$} independent of {\small $k$} and {\small $h$}. For the nonlinear term,
\begin{small}
$$
	\lvert s ( \sigma_{r} \boldsymbol{u} \times \boldsymbol{B}^{-}, 
	\boldsymbol{E} )  \rvert
	\leq s \Vert \boldsymbol{u} 
	\times \boldsymbol{B}^{-} \Vert_{\sigma_{r}}
	\Vert \boldsymbol{E} \Vert_{\sigma_{r}}
	\leq \Vert \boldsymbol{u} \Vert_{\mathcal{H}_{1} } 
	\Vert \boldsymbol{E} \Vert_{\mathcal{H}_{4} },
$$
\end{small}
and
\begin{small}
$$
	\lvert s ( \sigma_{r} \boldsymbol{u} \times \boldsymbol{B}^{-}, 
	\boldsymbol{v} \times \boldsymbol{B}^{-} ) \rvert
	\leq s \Vert \boldsymbol{u} 
	\times \boldsymbol{B}^{-} \Vert_{\sigma_{r}}
	\Vert \boldsymbol{v} 
	\times \boldsymbol{B}^{-} \Vert_{\sigma_{r}}
	\leq 
	\Vert \boldsymbol{u} \Vert_{\mathcal{H}_{1} } 
	\Vert \boldsymbol{v} \Vert_{\mathcal{H}_{1} }.
$$
\end{small}
Moreover, we have
\begin{small}
$$
	\lvert \alpha ( \mu_{r}^{-1} \nabla \times \boldsymbol{E}, 
	\boldsymbol{C} ) \rvert
	\leq \alpha
	\Vert \nabla \times \boldsymbol{E} \Vert_{\mu_{r}^{-1}}
	\Vert \boldsymbol{C} \Vert_{\mu_{r}^{-1}}
	\leq \Vert \boldsymbol{E} \Vert_{\mathcal{H}_{4} }
	\Vert \boldsymbol{C} \Vert_{\mathcal{H}_{3} }.
$$
\end{small}
Note that, 
\begin{small}
\begin{align*}
& \lvert s ( \sigma_{r} \boldsymbol{E}, \boldsymbol{F} ) \rvert
	\leq s \Vert \boldsymbol{E} \Vert_{\sigma_{r}}
	\Vert \boldsymbol{F} \Vert_{\sigma_{r}}
	\leq \Vert \boldsymbol{E} \Vert_{\mathcal{H}_{4} }
	\Vert \boldsymbol{F} \Vert_{\mathcal{H}_{4} }, \\
& \lvert  k^{-1} \alpha ( \mu_{r}^{-1} \boldsymbol{B}, 
	\boldsymbol{C} ) \rvert
	\leq k^{-1} \alpha \Vert \boldsymbol{B} \Vert_{\mu_{r}^{-1}}
	\Vert \boldsymbol{C} \Vert_{\mu_{r}^{-1}}
	\leq \Vert \boldsymbol{B} \Vert_{\mathcal{H}_{3} }
	\Vert \boldsymbol{C} \Vert_{\mathcal{H}_{3} }, 
\end{align*}
\end{small}
and 
\begin{small}
$$
	\lvert ( \nabla \cdot \boldsymbol{v}, q ) \rvert
	\leq \Vert \boldsymbol{v} \Vert_{\mathcal{H}_{1} } \Vert q \Vert_{\mathcal{H}_{2} }. 
$$ 
\end{small}
These estimates conclude the lemma.
\end{proof}

%%% inf-sup of b( , ).
\begin{lemma}\label{lemma:b_infsup}
If {\small $k \leq k_{0}$}, defined in \eqref{eq:def_k0}, there exists a constant {\small $\zeta> 0$} such that
\begin{small}
$$
	\inf\limits_{q \in Q} \sup\limits_{\boldsymbol{\eta} \in \boldsymbol{X}}
	\frac{ \boldsymbol{b}(\boldsymbol{\eta}, q) }
	{ \Vert \boldsymbol{\eta} \Vert_{\boldsymbol{X}}
	\Vert q \Vert_{Q} } \geq \zeta > 0, 
$$
\end{small}
and {\small $\zeta$} is independent of the mesh size $h$, the time  step size $k$, and physical parameters {\small $Rm$}, {\small $s$}, {\small $\mu_r$}, and {\small $\sigma_r$}.
\end{lemma}

\begin{proof}
It is well-known that the inf-sup condition of {\small $\boldsymbol{b}(\cdot, \cdot)$} for stable finite element pairs holds, i,e.
\begin{small}
$$	
	\inf\limits_{q \in Q} \sup\limits_{\boldsymbol{v} \in \boldsymbol{X}}
	\frac{ (\nabla \cdot \boldsymbol{v}, q) }
	{ \Vert \boldsymbol{v} \Vert_{1}
	\Vert q \Vert_{0} } \geq \zeta > 0,
$$
\end{small}
where {\small $\zeta$} is a constant independent of {\small $k$}. Therefore, for any {\small $q \in Q$}, we can choose {\small $\boldsymbol{\eta} = ( \boldsymbol{v}, \boldsymbol{0}, \boldsymbol{0} ) \in \boldsymbol{X}$} such that
\begin{small}
$$	
	\frac{ \boldsymbol{b} (\boldsymbol{\eta}, q) }
	{ \Vert \boldsymbol{v} \Vert_{1}
	\Vert q \Vert_{0} } \geq \zeta > 0.
$$
\end{small}
Note that
\begin{small}
$$
	\Vert \boldsymbol{v} \Vert_{\mathcal{H}_{1} } \Vert q \Vert_{\mathcal{H}_{2} } 
	\leq C \Vert \boldsymbol{v} \Vert_{1} \Vert q \Vert_{0},
$$
\end{small}
and the constant {\small $C$} is independent of {\small $k$}, which completes the proof.
\end{proof}

%%% inf-sup of a( , ).
\begin{lemma}\label{lemma:a_infsup}
If {\small $k \leq k_{0}$}, defined in \eqref{eq:def_k0}, there exists a constant {\small $\beta > 0$} such that
\begin{small}
\begin{align*}
& \inf\limits_{0 \neq \boldsymbol{\xi} \in 
\boldsymbol{X}^{0, \boldsymbol{u} }}
\sup \limits_{0 \neq \boldsymbol{\eta} \in 
\boldsymbol{X}^{0, \boldsymbol{u} }}
\frac{ \boldsymbol{a}(\boldsymbol{\xi}, \boldsymbol{\eta}) }{\Vert \boldsymbol{\xi} \Vert_{\boldsymbol{X}} \Vert \boldsymbol{\eta} \Vert_{\boldsymbol{X}} } \geq \beta > 0, \\
& \inf\limits_{0 \neq \boldsymbol{\eta} \in 
\boldsymbol{X}^{0, \boldsymbol{u} }}
\sup \limits_{0 \neq \boldsymbol{\xi} \in 
\boldsymbol{X}^{0, \boldsymbol{u} }}
\frac{ \boldsymbol{a}(\boldsymbol{\xi}, \boldsymbol{\eta}) }{\Vert \boldsymbol{\xi} \Vert_{\boldsymbol{X}} \Vert \boldsymbol{\eta} \Vert_{\boldsymbol{X}} } \geq \beta > 0,
\end{align*}
\end{small}
and {\small $\beta$} is independent of the mesh size $h$, the time  step size $k$, and physical parameters {\small $Rm$}, {\small $s$}, {\small $\mu_r$}, and {\small $\sigma_r$}. Here,
{\small $$\boldsymbol{X}^{0, \boldsymbol{u} } = \left\{ \boldsymbol{u} \in \boldsymbol{V}_{h}, ~ (\nabla \cdot \boldsymbol{u}, q) = 0, ~ \forall q \in Q_{h} \right\}.$$}
\end{lemma}

\begin{proof}
Take {\small $\boldsymbol{v} = \boldsymbol{u}$, $\boldsymbol{F} = \boldsymbol{E}$}, {\small $\boldsymbol{C} = - \frac{1}{2} ( \boldsymbol{B} + k \nabla \times \boldsymbol{E} )$}, then
\begin{small}
\begin{align*}
\boldsymbol{a}( \boldsymbol{\xi}, \boldsymbol{\eta} ) 
& = 
k^{-1} \Vert \boldsymbol{u} \Vert^{2}
+ Re^{-1} \Vert \nabla \boldsymbol{u} \Vert^{2}
+ k^{-1} \Vert \nabla \cdot \bm{u}\Vert^{2}
+ s \Vert \boldsymbol{E} \Vert^{2}_{\sigma_{r}}
+ \frac{\alpha k}{2} \Vert \nabla \times \boldsymbol{E} \Vert^{2}_{\mu_{r}^{-1} } \\
& + \frac{\alpha k^{-1}}{2} \Vert \boldsymbol{B} \Vert^{2}_{\mu_{r}^{-1} } 
+ 2 s ( \sigma_{r} \boldsymbol{u} \times \boldsymbol{B}^{-}, \boldsymbol{E} )
+ s \Vert \boldsymbol{u} \times \boldsymbol{B}^{-} \Vert_{\sigma_{r} }^{2}
+ \frac{\alpha}{2} \Vert \nabla \cdot \boldsymbol{B} \Vert^{2}_{\mu_{r}^{-1}}.
\end{align*}
\end{small}
Since
\begin{small}
\begin{align*}
2 ( \sigma_{r} \boldsymbol{u} \times \boldsymbol{B}^{-}, \boldsymbol{E} )
& \leq
2 \Vert \boldsymbol{E} \Vert_{\sigma_{r} } 
\Vert \boldsymbol{u} \times \boldsymbol{B}^{-} \Vert_{\sigma_{r} } 
\leq 
\frac{1}{2} \Vert \boldsymbol{E} \Vert^{2}_{\sigma_{r}}
+ 2 \Vert \boldsymbol{u} \times \boldsymbol{B}^{-} \Vert^{2}_{\sigma_{r}}.
\end{align*}
\end{small}
When {\small $k \leq k_{0}$}, we get {\small $
	s \Vert \boldsymbol{u} \times \boldsymbol{B}^{-} \Vert^{2}_{\sigma_{r}} 
	\leq \frac{1}{2} k^{-1} \Vert \boldsymbol{u} \Vert^{2}
$}. Therefore,
\begin{small}
$$
	k^{-1} \Vert \boldsymbol{u} \Vert^{2}
	- s \Vert \boldsymbol{u} \times \boldsymbol{B}^{-} \Vert^{2}_{\sigma_{r}}
	\geq \frac{1}{2} k^{-1} \Vert \boldsymbol{u} \Vert^{2}.
$$
\end{small}
Hence, {\small $
	\boldsymbol{a} ( \boldsymbol{\xi}, \boldsymbol{\eta} ) \geq \beta \Vert \boldsymbol{\xi} \Vert_{ \bm{X} }^{2}
$}. Obviously, we have {\small $\Vert \boldsymbol{\eta} \Vert_{ \bm{X} } \leq C \Vert \boldsymbol{\xi} \Vert_{ \bm{X} }$}, then the inf-sup condition holds. And neither {\small $C$} nor {\small $\beta$} depends on {\small $k$}, {\small $h$}, {\small $Rm$}, {\small $s$}, {\small $\sigma_{r}$} or {\small $\mu_{r}$}. The other inf-sup condition can be proved in the same way.

\end{proof}

Mixed formulations \eqref{eqn:problem-a0-1} and \eqref{eqn:problem-aux-1} are equivalent, and \eqref{eqn:problem-aux-1} is well-posed. As a result, we have the following well-posedness for the mixed formulation \eqref{eqn:problem-a0-1}.

\begin{theorem}\label{thm:wellposed_a0}
Given {\small $\boldsymbol{h} = ( \boldsymbol{f}, \boldsymbol{l}, \boldsymbol{r} ) \in \boldsymbol{V}^{\ast}_{h} \times [ \boldsymbol{V}^{d}_{h} ]^{\ast} \times [\boldsymbol{V}^{c}_{h} ]^{\ast} $} such that
\begin{small}
$$
	\langle \boldsymbol{l}, \boldsymbol{C} \rangle 
	= (\boldsymbol{l}_{R}, \boldsymbol{C} ), ~ 
	\forall \boldsymbol{C} \in \boldsymbol{V}^{d}
	\mbox{ for some } 
	\boldsymbol{l}_{R} \in \boldsymbol{V}^{d, 0},
$$  
\end{small}
and {\small $g \in Q_{h}^{\ast}$}, where {\small $\langle \cdot, \cdot \rangle$} is the duality pair between {\small $\boldsymbol{V}^{d}_{h}$} and {\small $[\boldsymbol{V}^{d}_{h}]^{\ast}$}, the mixed formulation \eqref{eqn:problem-a0-1} is well-posed if {\small $k \leq k_{0}$}.
\end{theorem}

\begin{proof}
By similar argument in Lemma 1 and Theorem 8 in \cite{Hu.K;Ma.Y;Xu.J.2014a}, it is straight-forward to reach the conclusion.
\end{proof}

% preconditioner.
\section{Robust preconditioner}\label{sec:uniform_pc}

In this section, we develop robust preconditioners based on the well-posedness of the discrete MHD system. We discuss two types of preconditioners, one is norm-equivalent preconditioners, which lead to block diagonal preconditioners, and the other is FOV-equivalent preconditioners, which lead to block triangular preconditioners. Here, we follow the classification proposed in \cite{Loghin.D;Wathen.A.2004a}.

\subsection{Norm-equivalent preconditioner}
Following the notation and approaches in \cite{Mardal.K;Winther.R.2011a}, we consider a general model problem on a Hilbert space {\small $\bm{H}$}, which is equipped with an inner product {\small $(\cdot, \cdot)_{\mathcal{H}} $} and an induced norm {\small $\| x \|^2_{\mathcal{H}} = (x, x)_{\mathcal{H}} $}: Find {\small $x \in \bm{H}$} such that 
\begin{small}
\begin{equation}\label{eqn:L-weak}
L(x, y) = \langle f, y \rangle, \ \forall y \in \bm{H},
\end{equation}
\end{small}
where {\small $L(\cdot, \cdot)$} is a symmetric bilinear form, and {\small $\langle \cdot, \cdot \rangle$} is the duality pair between {\small $\bm{H}$} and {\small $\bm{H}^{\ast}$}.  We assume the model problem \eqref{eqn:L-weak} is well-posed and satisfies
\begin{small}
\begin{equation}\label{ine:L-bound}
\lvert L(x,y) \rvert \leq \beta \Vert x\Vert_{\mathcal{H}} \Vert y \Vert_{\mathcal{H}}, \ \forall x, \ y \in \bm{H},
\end{equation}
\end{small} 
and 
\begin{small}
\begin{equation}\label{ine:L-infsup}
\inf_{0 \neq x \in \bm{H} } \sup_{0 \neq y \in \bm{H}} \frac{L(x, y)}{\| x \|_{\mathcal{H}} \| y \|_{\mathcal{H}}} 
\geq \alpha > 0.
\end{equation}
\end{small}
We define an operator {\small $\mathcal{A}: \bm{H} \rightarrow \bm{H}^* $} by
\begin{small}
\begin{equation*}
\langle \mathcal{A} x, y \rangle = L(x, y), \ \forall x,\ y \in \bm{H}.
\end{equation*}
\end{small}
Hence, the operator form of~\eqref{eqn:L-weak}~is
\begin{small}
\begin{equation}\label{eqn:L-operator}
\mathcal{A}x = f.
\end{equation}
\end{small}
Assume that an SPD operator {\small $\mathcal{M}: \bm{H}^* \rightarrow \bm{H}$} is a preconditioner of this system.  Based on {\small $\mathcal{M}$}, we can define an inner product {\small $(x, y)_{\mathcal{M}^{-1}} = \langle \mathcal{M}^{-1}x, y \rangle$} on {\small $\bm{H}$}, which induces a norm {\small $\| x \|_{\mathcal{M}^{-1}}^2 = (x, x)_{\mathcal{M}^{-1}}$}. As
\begin{small}
\begin{equation*}
(\mathcal{M} \mathcal{A}x, y)_{\mathcal{M}^{-1}} = \langle \mathcal{A}x, y \rangle = L(x,y) = (x, \mathcal{M}\mathcal{A}y)_{\mathcal{M}^{-1}},
\end{equation*}
\end{small}
{\small $\mathcal{MA}: \bm{H} \rightarrow \bm{H}$} is symmetric with respect to {\small $(\cdot, \cdot)_{\mathcal{M}^{-1}}$}.  Therefore, we can use preconditioned MINRES method \cite{Paige.C;Saunders.M.1975a} to solve \eqref{eqn:L-operator} with {\small $\mathcal{M}$} as the preconditioner and {\small $(\cdot, \cdot)_{\mathcal{M}^{-1}}$} as the inner product. It is proved \cite{Greenbaum.A.1997a} that if {\small $x^{m}$} is the $m$-th iteration of MINRES method and {\small $x$} is the exact solution, then there exists a constant {\small $\delta \in (0,1)$}, only depending on the condition number {\small $\kappa(\mathcal{MA})$}, such that 
\begin{small}
\begin{equation}
\langle \mathcal{MA} (x - x^m), \mathcal{A}(x-x^m) \rangle^{1/2} \leq 2 \delta^m \langle \mathcal{MA} (x - x^0), \mathcal{A}(x-x^0) \rangle^{1/2}.
\end{equation}
\end{small}
Moreover, an estimate leads to
\begin{small}
\begin{equation*}
\delta = \frac{\kappa(\mathcal{MA}) - 1}{\kappa(\mathcal{MA}) +1}.
\end{equation*}
\end{small}
Therefore, in order to estimate the performance of the preconditioner {\small $\mathcal{M}$}, we need to estimate the condition number {\small $\kappa(\mathcal{MA})$}.  Mardal and Winther  \cite{Mardal.K;Winther.R.2011a} prove that if \eqref{ine:L-bound} and \eqref{ine:L-infsup} hold, and the operator {\small $\mathcal{M}$} satisfies
\begin{small}
\begin{equation} \label{ine:M=H}
c_1 (x, x)_{\mathcal{H}} \leq (x, x)_{\mathcal{M}^{-1}} \leq c_2 (x,x)_{\mathcal{H}}, \quad \forall \ x \in \bm{H},
\end{equation}
\end{small}
then the condition number satisfies
\begin{small}
\begin{equation*}  %\label{ine:cond-MA}
1 \leq \kappa(\mathcal{MA}) \leq \frac{c_2 \beta}{ c_1 \alpha}.
\end{equation*}
\end{small}
They also give an estimate of convergence rate
\begin{small}
\begin{equation*} %\label{ine:MIN-MA} 
\delta \leq \frac{c_2 \beta - c_1 \alpha}{c_1 \alpha + c_2 \beta}.
\end{equation*}
\end{small}
This available analysis gives important guidance on designing preconditioners for MINRES method. That is, as long as we find an SPD operator {\small $\mathcal{M}$} that satisfies the assumption \eqref{ine:M=H} with constants {\small $c_1$} and {\small $c_2$} independent of the physical and discretize parameters, we can use it to precondition MINRES and the convergence rate is uniform with respect to the discretization and physical parameters. According to \cite{Loghin.D;Wathen.A.2004a}, condition \eqref{ine:M=H} means that operator {\small $\mathcal{H}$} and {\small $\mathcal{M}^{-1}$} are \emph{norm-equivalent}, and the corresponding preconditioners are known as \emph{norm-equivalent preconditioners}. Next, we discuss how to apply such an idea to design practical and robust preconditioners for the MHD system.  

\subsubsection{Application to the MHD System}
From the discussion in the previous section, we know that based on the well-posedness of the problem, 
robust and effective norm-equivalent preconditioners can be obtained as long as the condition \eqref{ine:M=H} is satisfied. In this section, we discuss how to use this framework to design preconditioners, which satisfy this condition, for the structure-preserving schemes \eqref{eqn:problem-a0-1} and \eqref{eqn:problem-aux-1}.

{\bf Preconditioner for auxiliary scheme \eqref{eqn:problem-aux-1}.}
Following the proof of the well-posedness of the symmetric Picard linearization \eqref{eq:fully_picard1}, we first design preconditioner for the auxiliary scheme with the stabilization term {\small $( \mu_{r}^{-1} \nabla \cdot \boldsymbol{B}, \nabla \cdot \boldsymbol{C})$}.  The operator form of \eqref{eqn:problem-aux-1} is
\begin{small} 
\begin{equation}\label{eqn:A-with-stable}
\tilde{\mathcal{A} } \bm{x} = \bm{F}  \Rightarrow
\begin{pmatrix}
\mathcal{A}_{1} & -\mathrm{div}^{\ast} & 0 & \mathcal{F}^{\ast} \\
-\mathrm{div} & 0 & 0 & 0\\
0 & 0 & - \alpha ( {k^{-1}}  \mu_{r}^{-1} \mathcal{I}_{3} + \mathrm{div}^{\ast} \mu_{r}^{-1}  \mathrm{div} ) & - \alpha \mu_{r}^{-1} \mathrm{curl}\\
\mathcal{F} & 0 & -\alpha \mathrm{curl}^{\ast} \mu_{r}^{-1} & s \sigma_{r} \mathcal{I}_{4}
\end{pmatrix}
\begin{pmatrix}
\bm{u} \\ 
p \\
\bm{B}\\ 
\bm{E} 
\end{pmatrix}
= 
\begin{pmatrix}
\bm{h}_{1} \\ 
-g \\
- \bm{h}_{2} \\ 
\bm{h}_{3} 
\end{pmatrix},
\end{equation}
\end{small}
where 
\begin{small}
\begin{align*}
& \mathcal{A}_{1} \bm{u} = 
k^{-1} \bm{u} - Re^{-1} \Delta\bm{u} + k^{-1} \mathrm{div}^{\ast} \mathrm{div} \bm{u} 
- s \sigma_{r} (\bm{u}\times {\bm{B}} ^{-})\times {\bm{B}} ^{-}, 
~ \forall \bm{u}\in \bm{V}_{h}, \\
& \mathcal{F} \bm{u} =
s \sigma_{r} \bm{u} \times \bm{B}^{-},
~ \forall \bm{u}\in \bm{V}_{h}, 
\end{align*}
\end{small}
and {\small $\bm{h}_1 = \widetilde{\bm{f}}$}, {\small $\bm{h}_2 = \alpha k^{-1} \mu_r^{-1} \bm{B}^0$}, {\small $\bm{h}_3 = \bm{0}$}, and {\small $g=0$}.  Note that, we have {\small $\mathrm{div} \bm{h}_2 = 0$} and {\small $\mathcal{A}_{1} = \mathcal{H}_{1}$}.

Based on the well-posedness of the scheme (see Theorem \ref{thm:wellposed_a}), {\small $\tilde{\mathcal A}$} is an isomorphism from {\small $\bm{X}_{h} \times Q_h $} to {\small $ ( \bm{X}_{h} \times Q_h)^* $}.  Therefore, one simple choice of the preconditioner is the operator induced by the inner product {\small $(\cdot, \cdot)_{\mathcal{H}}$} which satisfies \eqref{ine:M=H} with constants {\small $c_1 = c_2 = 1$}, i.e.  
{\small $
\tilde{\mathcal{D}} = \text{diag}\left( \mathcal{H}_{1}^{-1}, \mathcal{H}_{2}^{-1}, \mathcal{H}_{3}^{-1}, \mathcal{H}_{4}^{-1} \right)
$}. This is because for the MHD system, {\small $\Vert \cdot \Vert_{\mathcal{H}}^{2}$} is defined by \eqref{eq:Hnorm}. More precisely, operator {\small $\tilde{\mathcal{D}}$} is a natural isomorphism from {\small $ ( \bm{X}_{h} \times Q_h)^* $} to {\small $\bm{X}_{h} \times Q_h $} with operator form
\begin{small}
\begin{equation}\label{eqn:B-with-stable}
\tilde{\mathcal{D}} = 
\begin{pmatrix}
\mathcal{A}_{1} & 0  & 0            & 0 \\
0				       & 	k \mathcal{I}_{2}        & 0         & 0 \\
0					       &  0           & \alpha ( k^{-1} \mu_{r}^{-1}  \mathcal{I}_{3} + \mathrm{div}^{\ast} \mu_{r}^{-1} \mathrm{div} )&  0  \\
0					       &  0 		& 0 & s \sigma_{r} \mathcal{I}_{4} + k \alpha \mathrm{curl}^{\ast} \mu_{r}^{-1} \mathrm{curl} 
\end{pmatrix}^{-1}.
\end{equation}
\end{small}
Based on previous analysis, we can directly conclude that {\small $\tilde{\mathcal{D}}$} is a norm-equivalent preconditioner for {\small $\tilde{\mathcal{A}}$}.

A direct application of preconditioner {\small $\tilde{\mathcal{D}}$} in practice can be expensive and time-consuming because we need to invert each diagonal block of {\small $\tilde{\mathcal{D}}$}.  In order to make the preconditioner more practical, we replace the diagonal blocks of {\small $\tilde{\mathcal{D}}$} by their spectral equivalent SPD approximations, i.e.
\begin{small}
\begin{equation}\label{eqn:M-with-stable}
\tilde{\mathcal{M}} = 
\mathrm{diag}\left( \mathcal{Q}_{1}, \mathcal{Q}_{2}, \mathcal{Q}_{3}, \mathcal{Q}_{4} \right),
\end{equation}
\end{small}
where {\small
\begin{equation} \label{def:Qx}
c_{2, i } \left( \mathcal{Q}_{ i }  \bm{x}, \bm{x} \right)  
 \leq  \left( \mathcal{H}_{ i }^{-1} \bm{x}, \bm{x} \right) 
 \leq c_{1, i } \left( \mathcal{Q}_{i} \bm{x} , \bm{x} \right), ~ i = 1, 2, 3, 4.
 \end{equation}
 }
In the implementation, we define {\small $\mathcal{Q}_{1}$} by several MG cycles, {\small $\mathcal{Q}_{2}$} by simple iterative methods such as Jacobi methods, and {\small $\mathcal{Q}_{3}$}, {\small $\mathcal{Q}_{4}$} by the well-known HX-preconditioner \cite{Hiptmair.R;Xu.J.2007a}.  Then {\small $\tilde{\mathcal{M}}$} satisfies the condition \eqref{ine:M=H} with {\small $c_1 = \min\{  c_{1, 1}^{-1}, c_{1, 2}^{-1}, c_{1, 3}^{-1}, c_{1, 4}^{-1} \}$} and {\small $c_2 = \max\{  c_{2, 1}^{-1}, c_{2, 2}^{-1}, c_{2, 3}^{-1}, c_{2, 4}^{-1} \}$}.  Therefore, {\small $\tilde{\mathcal{M}}$} is a norm-equivalent preconditioner for {\small $\tilde{\mathcal{A}}$}.

{\bf Preconditioner for scheme \eqref{eqn:problem-a0-1}.}
Now we consider the original structure-preserving discretization without the stabilization term {\small $( \mu_{r}^{-1} \nabla \cdot \boldsymbol{B}, \nabla \cdot \boldsymbol{C})$}.  Similarly, the operator form of \eqref{eqn:problem-a0-1} is 
\begin{small} 
\begin{equation}\label{eqn:A-without-stable}
\mathcal{A} \bm{x} = \bm{F}  \Rightarrow
\begin{pmatrix}
\mathcal{A}_{1} & -\mathrm{div}^{\ast} & 0 & \mathcal{F}^{\ast} \\
-\mathrm{div} & 0 & 0 & 0\\
0 & 0 & - \alpha {k^{-1}} \mu_{r}^{-1} \mathcal{I}_{3} & - \alpha \mu_{r}^{-1} \mathrm{curl}\\
\mathcal{F} & 0 & -\alpha \mathrm{curl}^{\ast} \mu_{r}^{-1} & s \sigma_{r} \mathcal{I}_{4}
\end{pmatrix}
\begin{pmatrix}
\bm{u} \\ 
p \\
\bm{B}\\ 
\bm{E} 
\end{pmatrix}
= 
\begin{pmatrix}
\bm{h}_{1} \\ 
-g \\
- \bm{h}_{2} \\ 
\bm{h}_{3} 
\end{pmatrix}.
\end{equation}
\end{small}

Since {\small $\tilde{\mathcal{D}}$} is a uniform preconditioner for {\small $\tilde{\mathcal{A}}$} which has the stabilization term, it turns out that we can obtain a uniform preconditioner {\small $\mathcal{D}$} for {\small $\mathcal{A}$} by removing the stabilization term in {\small $\tilde{\mathcal{D}}$}, i.e. 
\begin{small}\begin{align}
& \mathcal{D} = \mathrm{diag}\left( \mathcal{H}_{1}^{-1}, \mathcal{H}_{2}^{-1}, \left[   \alpha k^{-1} \mu_{r}^{-1} \mathcal{I}_{3} \right]^{-1}, \mathcal{H}_{4}^{-1} \right).
\label{eqn:B-no-stable}
\end{align} \end{small}
Actually, using the fact that {\small $ \mathrm{div} \ \mathrm{curl} = 0 $}, we have {\small 
$$ 
 \left[\alpha ( k^{-1} \mu_{r}^{-1} \mathcal{I}_{3} + \mathrm{div}^{\ast} \mu_{r}^{-1} \mathrm{div} ) \right]^{-1}
 (- \alpha \mu_{r}^{-1} \mathrm{curl}) = -k \mathrm{curl} =
 \left[ \alpha k^{-1} \mu_{r}^{-1} \mathcal{I}_{3} \right]^{-1}(- \alpha \mu_{r}^{-1} \mathrm{curl}).
 $$
 } 
 Therefore, {\small $\tilde{\mathcal{D}} \tilde{\mathcal{A}} = \mathcal{D} \mathcal{A}$} and we conclude that {\small $\mathcal{D}$} is a norm-equivalent preconditioner for {\small $\mathcal{A}$}. Again, using {\small $\mathcal{D}$} in practice can be expensive and time-consuming.  We can replace the diagonal blocks of {\small $\mathcal{D}$} by their spectral equivalent SPD approximations that satisfy \eqref{def:Qx}, namely,
\begin{small}
\begin{equation}\label{eqn:M-no-stable}
\mathcal{M} = 
\mathrm{diag}\left( \mathcal{Q}_{1}, \mathcal{Q}_{2}, \mathcal{Q}_{3}, \mathcal{Q}_{4} \right).
\end{equation}
\end{small}
It is easy to see that {\small $\mathcal{M}$} is a norm-equivalent preconditioner for {\small $\mathcal{A}$}.

{\bf A divergence-free preserving iterative process.}
One important feature of the structure preserving discretization is that it preserves the Gauss's law of magnetic field exactly on the discrete level. Therefore, when designing preconditioners and iterative methods, we would like to inherit this property as well.  We now discuss how to preserve this property mainly for scheme \eqref{eqn:problem-a0-1} and briefly for scheme \eqref{eqn:problem-aux-1} briefly.

First, we consider a simple linear iteration based on operator {\small $\mathcal{D}$}:
\begin{small}
\begin{equation} \label{eqn:linear-iter-B}
\begin{pmatrix}
\bm{u}^{l+1} \\
p^{l+1} \\
\bm{B}^{l+1}  \\
\bm{E}^{l+1} 
\end{pmatrix}
= 
\begin{pmatrix}
\bm{u}^{l} \\
p^{l} \\
\bm{B}^{l}  \\
\bm{E}^{l}
\end{pmatrix}
+
\begin{pmatrix}
\bm{e}_{1} \\
\bm{e}_{2} \\
\bm{e}_{3} \\
\bm{e}_{4}  
\end{pmatrix},\quad
\begin{pmatrix}
\bm{e}_{1} \\
\bm{e}_{2}\\
\bm{e}_{3} \\
\bm{e}_{4} 
\end{pmatrix} = 
\mathcal{D} \left[  
\begin{pmatrix}
\bm{h}_{1} \\
-g \\ 
- \bm{h}_{2} \\ 
\bm{h}_{3} 
\end{pmatrix} - 
\mathcal{A}
\begin{pmatrix}
\bm{u}^{l} \\
p^{l} \\
\bm{B}^{l}  \\
\bm{E}^{l} 
\end{pmatrix}
\right].
\end{equation}
\end{small}
It is easy to check that, if {\small $\mathrm{div} \bm{B}^{l} = 0$} and {\small $\mathrm{div} \bm{h}_2 = 0$}, then {\small $\mathrm{div} \bm{e}_{3} = 0$}. Therefore, {\small $\mathrm{div} \bm{B}^{l+1} = 0$}. Namely, the linear iterative scheme based on {\small $\mathcal{B}$} preserves the divergence-free condition at each iteration.  

Now we consider the preconditioned MINRES method with {\small $\mathcal{D}$} as a preconditioner. The following theorem states that if the initial guess satisfies the divergence-free condition exactly, the solutions of all the iterations also satisfy this condition exactly.  

\begin{theorem} \label{thm:div-free-B}
If the initial guess {\small $\bm{x}^0 = (\bm{u}^{0}, p^{0}, \bm{B}^{0},  \bm{E}^{0})^T, $} satisfies the divergence-free condition exactly, i.e. {\small $\mathrm{div} \bm{B}^0 = 0$}, and {\small $\mathrm{div} \bm{h}_2 = 0$}, then all the iterations {\small $\bm{x}^l = (\bm{u}^{l}, p^l, \bm{B}^{l}, \bm{E}^{l})^T$} of the preconditioned MINRES method with preconditioner {\small $\mathcal{B}$} satisfy the divergence-free condition exactly, i.e. {\small $\mathrm{div} \bm{B}^l = 0$}.
\end{theorem} 

\begin{proof}
According to the definition of preconditioned MINRES method with preconditioner {\small $\mathcal{D}$}, we know that
\begin{small}
\begin{equation}\label{def:solu-krylov}
\bm{x}^l \in \bm{x}^0 + \mathcal{K}^l(\mathcal{D} \mathcal{A}, \bm{r}^0), 
\end{equation}
\end{small}
where {\small 
$$
\mathcal{K}^l(\mathcal{DA}, \bm{r}^0) = \text{span}\{ \bm{r}^0, \mathcal{DA}\bm{r}^0,  (\mathcal{DA})^2\bm{r}^0, \cdots, (\mathcal{DA})^{l-1}\bm{r}^0 \},
$$} 
with {\small $\bm{r}^0 = \bm{F} - \mathcal{DA} \bm{x}^0 $}, {\small $\bm{r}^0 = ( \bm{r}_{1}^{0},  \bm{r}_{2}^{0}, \bm{r}_{3}^{0},  \bm{r}_{4}^{0} )^T$}.  Note that {\small $\mathrm{div}\bm{r}_{3}^{0} = 0$} by its definition.  

Now, we denote {\small $(\mathcal{DA})^{m} \bm{r}^0$} by {\small $\bm{v}^{m} =  ( \bm{v}_{1}^{m}, \bm{v}_{2}^{m} , \bm{v}_{3}^{m},  \bm{v}_{4}^{m} )^T$} for {\small $m =  0, 1, 2, \cdots, l-1$}.  Since {\small $\bm{v}^m = \mathcal{DA} \bm{v}^{m-1}$} and 
\begin{small}
\begin{align} 
\bm{v}_{3}^{m} 
& = (\alpha k^{-1} \mu_{r}^{-1} \mathcal{I}_{3})^{-1} 
\left( - {\alpha}{k}^{-1} \mu_{r}^{-1} \mathcal{I}_{3} \bm{v}_{3}^{m-1}  
-\alpha \mu_{r}^{-1} \mathrm{curl} \bm{v}_{4}^{m-1} \right) 
\nonumber\\
& = - \bm{v}_{3}^{m-1} - k \mathrm{curl}\bm{v}_{4}^{ m-1},
\label{eqn:vBm}
\end{align}
\end{small}
{\small $\mathrm{div}\bm{v}_{3}^{m} = 0$} if {\small $\mathrm{div}\bm{v}_{3}^{m-1} = 0$}. Noticing that {\small $\mathrm{div}\bm{v}_{3}^{0} = \mathrm{div}\bm{r}_{3}^{0} = 0$}, by mathematical induction, we have {\small $\mathrm{div}\bm{v}_{3}^{m} = 0$}.

Finally, due to the fact \eqref{def:solu-krylov}, {\small $\bm{x}^{l}$} is a linear combination of {\small $\bm{v}^{m}$}, {\small $m = 0,1, \cdots, l-1$} and then {\small $\bm{B}^{l}$} is a linear combination of {\small $\bm{v}_{3}^{m}$}, {\small $m = 0,1, \cdots, l-1$}.  Because {\small $\mathrm{div}\bm{v}_{3}^{m} = 0$}, {\small $m=0,1,\cdots,l-1$}, we can conclude that {\small $\mathrm{div} \bm{B}^{l} = 0$}.
\end{proof}

Although using {\small $\mathcal{D}$} as a preconditioner preserves the divergence-free condition exactly, in general, it is not true when we use operator {\small $\mathcal{M}$} defined by \eqref{eqn:M-no-stable} as preconditioner.  In order to preserve this property, we need to use {\small $\mathcal{Q}_{3} = \left[ \alpha k^{-1} \mu_{r}^{-1} \mathcal{I}_{3} \right]^{-1}$} in the preconditioner {\small $\mathcal{M}$}, namely,
\begin{small}
\begin{equation}\label{eqn:M-div-free}
\mathcal{M} = 
\mathrm{diag} \left(
\mathcal{Q}_{1}, \mathcal{Q}_{2}, \left[ \alpha k^{-1} \mu_{r}^{-1} \mathcal{I}_{3} \right]^{-1}, \mathcal{Q}_{4} 
\right).
\end{equation}
\end{small}
\begin{remark}\label{remark:divB_free}
In terms of implementation, note that, for the magnetic field in {\small $\bm{v}^m = \mathcal{MA} \bm{v}^{m-1}$}, we still have \eqref{eqn:vBm}, which can be used to update {\small $\bm{v}_{3}^{m}$} at each preconditioning step without inverting the mass matrix.
\end{remark}

The following theorem states that if the initial guess satisfies the divergence-free condition exactly, then the solutions of all the iterations satisfy this condition exactly when {\small $\mathcal{M}$} defined in \eqref{eqn:M-div-free} is used as a preconditioner.

\begin{theorem} \label{thm:dive-free-M}
If the initial guess {\small $\bm{x}^0 = (\bm{u}^{0}, p^{0}, \bm{B}^{0}, \bm{E}^{0} )^T, $} satisfies the divergence-free condition exactly, i.e. {\small $\mathrm{div} \bm{B}^0 = 0$} and {\small $\mathrm{div} \bm{h}_2 = 0$}, then all the iterations {\small $\bm{x}^l = (\bm{u}^{l}, p^l, \bm{B}^{l},  \bm{E}^{l})^T$} of the preconditioned MINRES method with preconditioner {\small $\mathcal{M}$} defined in \eqref{eqn:M-div-free} satisfy the divergence-free condition exactly, i.e.  {\small $\mathrm{div} \bm{B}^l = \bm{0}$}.\end{theorem} 

\begin{proof}
The proof is the same as the proof of Theorem \ref{thm:div-free-B} with $\mathcal{D}$ replaced by $\mathcal{M}$.
\end{proof}

Next we consider the linear system {\small $\tilde{\mathcal{A} } \bm{x} = \bm{F} $ }. If we use {\small $\tilde{\mathcal{D}}$} as the preconditioner, the solutions of all the iterations satisfy the divergence-free condition exactly if the initial guess is divergence-free. The argument is exactly the same with that of Theorem \ref{thm:div-free-B} except that  \eqref{eqn:vBm} is replaced by
 \begin{small}
 \begin{equation}
 \bm{v}_{3}^{m}   =   \mathcal{H}_{3}^{-1} 
 \left(  - \mathcal{H}_{3} \bm{v}_{3}^{ m-1}   
 - \alpha \mu_{r}^{-1} \mathrm{curl} \bm{v}_{4}^{ m-1} \right) 
 = - \bm{v}_{3}^{m-1}  - k \mathrm{curl}\bm{v}_{4}^{ m-1}.
 \label{eqn:vBm_divdiv}
 \end{equation}
 \end{small}
And if we use {\small $\tilde{\mathcal{M}}$}, with {\small $\mathcal{Q}_{3} = \mathcal{H}_{3}^{-1} = \left( \alpha ( k^{-1} \mu_{r}^{-1} \mathcal{I}_{3} + \mathrm{div}^{\ast} \mu_{r}^{-1} \mathrm{div} ) \right)^{-1}$}, as the preconditioner, i.e. 
\begin{small}
\begin{equation}\label{eqn:tildeM-div-free}
\tilde{\mathcal{M}} = 
\mathrm{diag}\left(
\mathcal{Q}_{1}, 
\mathcal{Q}_{2}, 
\mathcal{H}_{3}^{-1}, 
\mathcal{Q}_{4}
\right),
\end{equation}
\end{small}
all the iterations satisfy the divergence-free condition exactly as long as the initial guess is divergence-free.

% lower preconditioner.
\subsection{FOV-equivalent preconditioner}
In this section, we recall the abstract framework for designing the \emph{FOV-equivalent preconditioners}, following \cite{Loghin.D;Wathen.A.2004a}. FOV-equivalent preconditioners are not necessary to be SPD, which makes it more general than norm-equivalent preconditioners. For example, we can design block triangular preconditioners, and use them to precondition the GMRES method.

Consider the model problem \eqref{eqn:L-operator}, now {\small $\mathcal{A}$} is not necessary to be symmetric.  We use a general operator {\small $\mathcal{M}_{\mathcal{L}}: \bm{H}^{\ast} \rightarrow \bm{H}$} to denote the preconditioner. Based on the inner product {\small $(\cdot, \cdot)_{\mathcal{M}^{-1}}$} and the norm {\small $\Vert \cdot \Vert_{\mathcal{M}^{-1}}$}, we can estimate the convergence rate of the preconditioned GMRES. It is proved \cite{Elman.H.1982a,Saad.Y.1996a} that if {\small $x^m$} is the {\small $m$}-iteration of GMRES method and {\small $x$} is the exact solution, then
\begin{small}
$$
	\frac{ \Vert \mathcal{M}_{\mathcal{L}} \mathcal{A}(x - x^m) \Vert_{\mathcal{M}^{-1}} }{ \Vert \mathcal{M}_{\mathcal{L}} \mathcal{A}(x- x^{0}) \Vert_{\mathcal{M}^{-1}} } 
	\leq
	\left( 1 - \frac{\gamma^{2}}{\Gamma^{2}} \right)^{m/2},
$$
\end{small}
where
\begin{small}
\begin{equation} \label{def:FOV}
	\gamma \leq \frac{ (x, \mathcal{M}_{\mathcal{L}} \mathcal{A} x )_{\mathcal{M}^{-1}} }{(x, x)_{\mathcal{M}^{-1}} }, 
	\quad
	\frac{ \Vert \mathcal{M}_{\mathcal{L}} \mathcal{A} x \Vert_{\mathcal{M}^{-1}} }{\Vert x \Vert_{\mathcal{M}^{-1}} } \leq \Gamma.
\end{equation}
\end{small}
According to the theory, we conclude that as long as we find an operator {\small $\mathcal{M}_{\mathcal{L}}$} and a proper inner product {\small $(\cdot, \cdot)_{\mathcal{M}^{-1}}$} such that condition \eqref{def:FOV} is satisfied with constants {\small $\gamma$} and {\small $\Gamma$} independent of the physical and discretize parameters, {\small $\mathcal{M}_{\mathcal{L}}$} is a uniform preconditioner for the GMRES method. Such preconditioners is usually referred to as FOV-equivalent preconditioners.  

% Abstract framework for FOV-equivalent preconditioners.
To give a flavor of the analysis of FOV-equivalent preconditioners, we first demonstrate the analysis with a $2$-by-$2$ block system. Application to the MHD system is discussed later. Assume that {\small $\mathcal{A} \bm{x} = \bm{F}$} is  in the form
\begin{small}
\begin{align}
& \begin{pmatrix}
\mathcal{A}_{1} & -\mathcal{B}^{\ast} \\
\mathcal{B} & 0
\end{pmatrix} \begin{pmatrix}
\bm{x}_{1} \\
\bm{x}_{2} 
\end{pmatrix} = \begin{pmatrix}
\bm{F}_{1} \\
\bm{F}_{2} 
\end{pmatrix}, \label{prob:toy_model}
\end{align}
\end{small}
where {\small $\mathcal{A}_{1}$} is a SPD operator. Based on the partition of the system, we assume that a splitting of the Hilbert space {\small $\bm{H}$} is {\small $\bm{H}_{1} \times \bm{H}_{2}$} such that {\small $\bm{x}_{1} \in \bm{H}_{1}$}, {\small $\bm{x}_{2} \in \bm{H}_{2}$}. Assume that the problem \eqref{prob:toy_model} is well-posed with respect to norm {\small $\Vert \cdot \Vert_{ \mathcal{M}^{-1} }$}, which is induced by {\small $\mathcal{M} = \mathrm{diag} \left( \mathcal{H}_{1}, \mathcal{H}_{2} \right)^{-1}$}. And we further assume that {\small $\mathcal{H}_{1} = \mathcal{A}_{1}$}. Therefore, the well-posedness implies that there exists a constant {\small $\zeta > 0$}, independent of physical and discretization parameters (depending on the problem) such that
\begin{small}
\begin{align}
& \inf \limits_{ \bm{x}_{2} \in \bm{H}_{2} } \sup \limits_{ \bm{x}_{1} \in \bm{H}_{1} }
\frac{ ( \mathcal{B} \bm{x}_{1}, \bm{x}_{2} ) }
{ \Vert \bm{x}_{1} \Vert_{ \mathcal{A}_{1} } \Vert \bm{x}_{2} \Vert_{ \mathcal{H}_{2} } }
\geq \zeta > 0. \label{eq:B_infsup}
\end{align}
\end{small}
\begin{theorem}\label{thm:toy_model_exact_solve}
If the condition \eqref{eq:B_infsup} holds, there exist constants {\small $\gamma$} and {\small $\Gamma$} such that for all {\small $\bm{x} \neq \bm{0}$}, the operator {\small $\mathcal{A}$} defined in \eqref{prob:toy_model} and the operator 
\begin{small}
\begin{align*}
& \mathcal{M}_{\mathcal{L}} = \begin{pmatrix}
\mathcal{A}_{1} & 0 \\
\mathcal{B} & \mathcal{H}_{2}
\end{pmatrix}^{-1}
\end{align*}
\end{small}
satisfy condition \eqref{def:FOV} with the norm {\small $\Vert \cdot \Vert_{ \mathcal{M}^{-1} }$} induced by {\small $\mathcal{M} = \mathrm{diag} \left( \mathcal{A}_{1}, \mathcal{H}_{2} \right)^{-1}$}.
\end{theorem}

\begin{proof}
By simple computation, we get
\begin{small}
\begin{align*}
& \mathcal{M}_{\mathcal{L}} \mathcal{A} = \begin{pmatrix}
\mathcal{I}_{1} & - \mathcal{A}_{1}^{-1} \mathcal{B}^{\ast} \\
0 & \mathcal{H}_{2}^{-1} \mathcal{B} \mathcal{A}_{1}^{-1} \mathcal{B}^{\ast}
\end{pmatrix}.
\end{align*}
\end{small}
Then for any {\small $\bm{x} = ( \bm{x}_{1}, \bm{x}_{2} )^{T}$}, we have
\begin{small}
\begin{align*}
( \bm{x}, \mathcal{M}_{\mathcal{L}} \mathcal{A} \bm{x}  )_{ \mathcal{M}^{-1} }
& = ( \bm{x}_{1} - \mathcal{A}_{1}^{-1} \mathcal{B}^{\ast} \bm{x}_{2}, \bm{x}_{1} )_{ \mathcal{A}_{1} }
+ ( \mathcal{B} \mathcal{A}_{1}^{-1} \mathcal{B}^{\ast} \bm{x}_{2}, \bm{x}_{2} ) \\
& = \Vert \bm{x}_{1} \Vert_{ \mathcal{A}_{1} }^{2} 
- ( \mathcal{B}^{\ast} \bm{x}_{2}, \bm{x}_{1} )
+ \Vert \mathcal{B}^{\ast} \bm{x}_{2} \Vert_{ \mathcal{A}_{1}^{-1} }^{2} \\
& \geq
\Vert \bm{x}_{1} \Vert_{ \mathcal{A}_{1} }^{2} 
- \Vert \bm{x}_{1} \Vert_{ \mathcal{A}_{1} } 
\Vert \mathcal{B}^{\ast} \bm{x}_{2} \Vert_{ \mathcal{A}_{1}^{-1} }
+ \Vert \mathcal{B}^{\ast} \bm{x}_{2} \Vert_{ \mathcal{A}_{1}^{-1} }^{2} \\
& = 
\begin{pmatrix}
\xi_{1} \\
\xi_{2}
\end{pmatrix}^{T}
\begin{pmatrix}
1 & -1/2 \\
-1/2 & 1
\end{pmatrix} 
\begin{pmatrix}
\xi_{1} \\
\xi_{2}
\end{pmatrix},
\end{align*}
\end{small}
where {\small $\xi_{1} = \Vert \bm{x}_{1} \Vert_{ \mathcal{A}_{1} }$}, {\small $\xi_{2} = \Vert \mathcal{B}^{\ast} \bm{x}_{2} \Vert_{ \mathcal{A}_{1}^{-1} }$}. Since the matrix in the middle is SPD, there exists {\small $\gamma_{0} > 0$} such that
\begin{small}
\begin{align*}
& ( \bm{x},  \mathcal{M}_{\mathcal{L}} \mathcal{A} \bm{x}  )_{ \mathcal{M}^{-1} } 
\geq \gamma_{0} \left( \Vert \bm{x}_{1} \Vert_{ \mathcal{A}_{1} }^{2}
+ \Vert \mathcal{B}^{\ast} \bm{x}_{2} \Vert_{ \mathcal{A}_{1}^{-1} }^{2} \right).
\end{align*}
\end{small}
Moreover,
\begin{small}
\begin{align*}
& \Vert \mathcal{B}^{\ast} \bm{x}_{2} \Vert_{ \mathcal{A}_{1}^{-1} }
= \sup \limits_{ \bm{x}_{1} \in \mathcal{H}_{1} } 
\frac{  ( \mathcal{B} \bm{x}_{1}, \bm{x}_{2} ) }{  \Vert \bm{x}_{1} \Vert_{ \mathcal{A}_{1} } }
\geq \zeta  \Vert \bm{x}_{2} \Vert_{ \mathcal{H}_{2} },
\end{align*} 
\end{small}
we get
\begin{small}
\begin{align*}
&  ( \bm{x},  \mathcal{M}_{\mathcal{L}} \mathcal{A} \bm{x}  )_{ \mathcal{M}^{-1} } 
\geq \gamma_{0} \Vert \bm{x}_{1} \Vert_{ \mathcal{A}_{1} }^{2}
+ \gamma_{0} \zeta^{2} \Vert \bm{x}_{2} \Vert_{ \mathcal{H}_{2} }^{2}
\geq
\min \left\{ \gamma_{0}, \gamma_{0} \zeta^{2} \right\}
(\bm{x}, \bm{x})_{ \mathcal{M}^{-1} }^{2},
\end{align*}
\end{small}
which leads to the lower bound {\small $\gamma$}. The upper bound {\small $\Gamma$} follows directly from the boundedness of each term.
\end{proof}

As we mentioned before, applying {\small $\mathcal{M}_{\mathcal{L}}$} defined in Theorem \ref{thm:toy_model_exact_solve} as preconditioner can be expensive and time-consuming. Therefore, we replace the diagonal blocks by their spectral equivalent SPD approximations. The following theorem states that under certain assumptions, such preconditioner is still robust.

\begin{theorem}\label{thm:toy_model_inexact}
If the condition \eqref{eq:B_infsup} holds, there exist constants {\small $\gamma$} and {\small $\Gamma$} such that for all {\small $\bm{x} \neq \bm{0}$}, the operator {\small $\mathcal{A}$} defined in \eqref{prob:toy_model} and the operator
\begin{small}
\begin{align*}
& \widehat{ \mathcal{M} }_{\mathcal{L}} = \begin{pmatrix}
\mathcal{Q}_{1}^{-1} & 0 \\
\mathcal{B} & \mathcal{Q}_{2}^{-1}
\end{pmatrix}^{-1}
\end{align*}
\end{small}
satisfy condition \eqref{def:FOV} with the norm {\small $\Vert \cdot \Vert_{ \mathcal{M}^{-1} }$} induced by {\small $\mathcal{M} = \mathrm{diag} \left( \mathcal{Q}_{1}, \mathcal{Q}_{2} \right)$} provided that 
\begin{enumerate}
\item {\small $c_{2, i } \left( \mathcal{Q}_{ i }  \bm{x}, \bm{x} \right)  
 \leq  \left( \mathcal{H}_{ i }^{-1} \bm{x}, \bm{x} \right) 
 \leq c_{1, i } \left( \mathcal{Q}_{i} \bm{x} , \bm{x} \right)$}, {\small $i = 1$} or {\small $2$},
\item {\small $\Vert \mathcal{I}_{1} - \mathcal{Q}_{1} \mathcal{A}_{1} \Vert_{ \mathcal{A}_{1} } \leq \rho$}, with {\small $0 \leq \rho < 1$}.
\end{enumerate}
\end{theorem}

\begin{proof}
By simple computation, we get
\begin{small}
\begin{align*}
& \widehat{ \mathcal{M} }_{\mathcal{L}} \mathcal{A} = 
\begin{pmatrix}
\mathcal{Q}_{1} \mathcal{A}_{1} & - \mathcal{Q}_{1} \mathcal{B}^{\ast} \\
\mathcal{Q}_{2} \mathcal{B} ( \mathcal{I}_{1} - \mathcal{Q}_{1} \mathcal{A}_{1} )
& \mathcal{Q}_{2} \mathcal{B} \mathcal{Q}_{1} \mathcal{B}^{\ast}
\end{pmatrix}.
\end{align*}
\end{small}
Then for any {\small $\bm{x} = ( \bm{x}_{1}, \bm{x}_{2} )^{T}$}, we have
\begin{small}
\begin{align*}
( \bm{x}, \widehat{ \mathcal{M} }_{\mathcal{L}} \mathcal{A} \bm{x} )_{ \mathcal{M}^{-1} } 
& = \Vert \bm{x}_{1} \Vert_{ \mathcal{A}_{1} }^{2}
- ( \mathcal{B}^{\ast} \bm{x}_{2}, \bm{x}_{1} )
+ ( \mathcal{B} ( \mathcal{I}_{1} - \mathcal{Q}_{1} \mathcal{A}_{1} ) \bm{x}_{1}, \bm{x}_{2} )
+ \Vert \mathcal{B}^{\ast} \bm{x}_{2} \Vert_{ \mathcal{Q}_{1} }^{2} \\
& = \Vert \bm{x}_{1} \Vert_{ \mathcal{A}_{1} }^{2}
- ( \mathcal{Q}_{1} \mathcal{A}_{1}  \bm{x}_{1}, \mathcal{B}^{\ast} \bm{x}_{2} )
+ \Vert \mathcal{B}^{\ast} \bm{x}_{2} \Vert_{ \mathcal{Q}_{1} }^{2} .
\end{align*}
\end{small}
As {\small $\Vert \mathcal{I}_{1} - \mathcal{Q}_{1} \mathcal{A}_{1} \Vert_{ \mathcal{A}_{1} } \leq \rho$} implies that
\begin{small}
\begin{align*}
& ( 1-\rho ) (\bm{x}_{1}, \bm{x}_{1})_{ \mathcal{A}_{1}^{-1} } 
\leq (\bm{x}_{1}, \bm{x}_{1})_{ \mathcal{Q}_{1} } 
\leq ( 1+\rho ) (\bm{x}_{1}, \bm{x}_{1})_{ \mathcal{A}_{1}^{-1} }, \\
& ( 1+\rho )^{-1} (\bm{x}_{1}, \bm{x}_{1})_{ \mathcal{A}_{1} } 
\leq (\bm{x}_{1}, \bm{x}_{1})_{\mathcal{Q}_{1}^{-1} } 
\leq ( 1-\rho )^{-1} (\bm{x}_{1}, \bm{x}_{1})_{ \mathcal{A}_{1} },
\end{align*}
\end{small}
we have
\begin{small}
\begin{align*}
- ( \mathcal{Q}_{1} \mathcal{A}_{1}  \bm{x}_{1}, \mathcal{B}^{\ast} \bm{x}_{2} ) 
& \leq  
\| \mathcal{A}_1 \bm{x}_1 \|_{\mathcal{Q}_1} \| \mathcal{B}^{\ast} \bm{x}_2 \|_{\mathcal{Q}_1} 
\leq (1+\rho) \| \mathcal{A}_1 \bm{x}_1 \|_{\mathcal{A}_1^{-1}} \| \mathcal{B}^{\ast} \bm{x}_2 \|_{\mathcal{Q}_1} \\  
& = (1+\rho) \| \bm{x}_1 \|_{\mathcal{A}_1} \| \mathcal{B}^{\ast} \bm{x}_2 \|_{\mathcal{Q}_1}, \\
\end{align*}
\end{small}
Therefore,
\begin{small}
\begin{align*}
( \bm{x}, \widehat{ \mathcal{M} }_{\mathcal{L}} \mathcal{A} \bm{x} )_{ \mathcal{M}^{-1} } 
& \geq 
 \Vert \bm{x}_{1} \Vert_{ \mathcal{A}_{1} }^{2}
- (1+ \rho) \Vert \bm{x}_{1} \Vert_{ \mathcal{A}_1 }
\Vert \mathcal{B}^{\ast} \bm{x}_{2} \Vert_{ \mathcal{Q}_{1} }
+ \Vert \mathcal{B}^{\ast} \bm{x}_{2} \Vert_{ \mathcal{Q}_{1} }^{2} \\
& = 
\begin{pmatrix}
\xi_{1} \\
\xi_{2}
\end{pmatrix}^{T}
\begin{pmatrix}
1 & -(1+\rho)/2 \\
-(1+\rho)/2 & 1
\end{pmatrix}
\begin{pmatrix}
\xi_{1} \\
\xi_{2}
\end{pmatrix},
\end{align*}
\end{small}
where {\small $\xi_{1} = \Vert \bm{x}_{1} \Vert_{ \mathcal{A}_{1} }$}, {\small $\xi_{2} = \Vert \mathcal{B}^{\ast} \bm{x}_{2} \Vert_{ \mathcal{Q}_{1} }$}. We can verify that the matrix in the middle is SPD when {\small $0 \leq \rho < 1$}. Therefore, there exists a constant {\small $\gamma_{0} > 0$} such that
\begin{small}
\begin{align*}
( \bm{x}, \widehat{ \mathcal{M} }_{\mathcal{L}} \mathcal{A} \bm{x} )_{ \mathcal{M}^{-1} } 
& \geq 
\gamma_{0} \left( \Vert \bm{x}_{1} \Vert_{ \mathcal{A}_{1} }^{2}
+ \Vert \mathcal{B}^{\ast} \bm{x}_{2} \Vert_{ \mathcal{Q}_{1} }^{2}  \right)
\geq
\gamma_{0} (1-\rho) \Vert \bm{x}_{1} \Vert_{ \mathcal{Q}_{1}^{-1} }^{2}
+ \gamma_{0} (1-\rho) \zeta^{2} \Vert \bm{x}_{2} \Vert_{ \mathcal{H}_{2} }^{2} \\
& \geq
\min\left\{ \gamma_{0} (1-\rho), \gamma_{0} (1-\rho) \zeta^{2} c_{1,2}^{-1} \right\} 
( \bm{x}, \bm{x} )_{ \mathcal{M}^{-1} },
\end{align*}
\end{small}
which leads to the lower bound {\small $\gamma$}. The upper bound {\small $\Gamma$} follows directly from the fact that each term is bounded.
\end{proof}

We comment that the second assumption in Theorem \ref{thm:toy_model_inexact} is reasonable as in practice we can achieve it by performing one or several steps of V-cycle multigrid method.

\subsubsection{Application to MHD system}
In this section, we discuss how to design FOV-equivalent preconditioners for the structure-preserving discretization \eqref{eqn:problem-a0-1}.  Instead of giving details of the FOV-equivalent preconditioners for the auxiliary problem \eqref{eqn:problem-aux-1}, we comment that similar preconditioners exist, and theoretical results follow by similar argument.

The operator form of the mixed formulation \eqref{eqn:problem-a0-1} is
\begin{small}
\begin{align*}
\mathcal{A}\bm{x} = \bm{F}  \Longrightarrow
& \begin{pmatrix}
\mathcal{A}_{1} & -\mathrm{div}^{\ast} & 0 & \mathcal{F}^{\ast} \\
\mathrm{div} & 0 & 0 & 0 \\
0 & 0 & \alpha k^{-1} \mu_{r}^{-1} \mathcal{I}_{3} & \alpha \mu_{r}^{-1} \mathrm{curl} \\
\mathcal{F} & 0 & - \alpha \mathrm{curl}^{\ast} \mu_{r}^{-1} & s \sigma_{r} \mathcal{I}_{4}
\end{pmatrix} 
\begin{pmatrix}
\bm{u} \\ p \\ \bm{B} \\ \bm{E}
\end{pmatrix} = \begin{pmatrix}
\bm{h}_{1} \\ g \\ \bm{h}_{2} \\ \bm{h}_{3}
\end{pmatrix},
\end{align*}
\end{small}
Here, we multiply {\small $-1$} to the second and the third equations because we do not require the system to be symmetric any more and such a small modification makes our exposition more clear. Moreover, we still have {\small $\mathcal{H}_{1} = \mathcal{A}_{1}$}.

Based on the well-posedness of scheme \eqref{eqn:problem-a0-1}, it is natural to propose the following block lower triangular preconditioner {\small $\mathcal{M}_{\mathcal{L}}$}:
\begin{small}
\begin{align}\label{eqn:L-no-stable}
& \mathcal{M}_{\mathcal{L}} = \begin{pmatrix}
\mathcal{A}_{1} & 0 & 0 & 0 \\
\mathrm{div} & k \mathcal{I}_{2} & 0 & 0 \\
0 & 0 & \alpha k^{-1} \mu_{r}^{-1} \mathcal{I}_{3} & 0 \\
\mathcal{F} & 0 & - \alpha \mathrm{curl}^{\ast} \mu_{r}^{-1} & \mathcal{H}_{4}
\end{pmatrix}^{-1}.
\end{align}
\end{small}
Note that the diagonal blocks of {\small $\mathcal{M}_{\mathcal{L}}^{-1}$} are the same as those of {\small $\mathcal{D}^{-1}$} defined by \eqref{eqn:B-no-stable}.  Thus, naturally we choose inner product {\small $(\cdot, \cdot)_{\mathcal{H}}$} and norm {\small $\| \cdot \|_{\mathcal{H}}$} \eqref{eq:Hnorm} to verify the condition \eqref{def:FOV}. We prove that operator {\small $\mathcal{M}_{\mathcal{L}}$} is a robust FOV-equivalent preconditioner in the following theorem. 

\begin{theorem}\label{thm:lower_exact_solve}
If {\small $k \leq k_{0}$}, there exist constants {\small $\gamma$} and {\small $\Gamma$} that are independent of the mesh size {\small $h$}, time step size {\small $k$}, and physical parameters {\small $Rm$}, {\small $s$}, {\small $\mu_r$}, and {\small $\sigma_r$}, such that for all {\small $\bm{x} \neq \bm{0} $}, the condition \eqref{def:FOV} holds with {\small $(\cdot, \cdot)_{\mathcal{M}^{-1} } = (\cdot, \cdot)_{ \mathcal{H} }$}.
\end{theorem}

\begin{proof}
By simple computation, we get
% L^{-1}.
%\begin{small}
%\begin{align*}
%& \mathcal{M}_{\mathcal{L}} = \begin{pmatrix}
%\mathcal{A}_{1}^{-1} & 0 & 0 & 0 \\
%- k^{-1} \mathrm{div} \mathcal{A}_{1}^{-1} & k^{-1} \mathcal{I}_{2} & 0 & 0 \\
%0 & 0 & \alpha^{-1} k \mu_{r} \mathcal{I}_{3} & 0 \\
%- \mathcal{H}_{4}^{-1} \mathcal{F} \mathcal{A}_{1}^{-1} & 0 & k \mu_{r} \mathcal{H}_{4}^{-1} \mathrm{curl}^{\ast} \mu_{r}^{-1} & \mathcal{H}_{4}^{-1}
%\end{pmatrix}.
%\end{align*}
%\end{small}
% end L^{-1}.
\begin{small}
\begin{align*}
& \mathcal{M}_{\mathcal{L}} \mathcal{A} = \begin{pmatrix}
\mathcal{I}_{1} & - \mathcal{A}_{1}^{-1} \mathrm{div}^{\ast} & 0 & \mathcal{A}_{1}^{-1} \mathcal{F}^{\ast} \\
0 & k^{-1} \mathrm{div} \mathcal{A}_{1}^{-1} \mathrm{div}^{\ast} & 0 & - k^{-1} \mathrm{div} \mathcal{A}_{1}^{-1} \mathcal{F}^{\ast} \\
0 & 0 & \mathcal{I}_{3} & k \mathrm{curl} \\
0 &  \mathcal{H}_{4}^{-1} \mathcal{F} \mathcal{A}_{1}^{-1} \mathrm{div}^{\ast} & 0 &  \mathcal{H}_{4}^{-1} \mathcal{S}_{4}
\end{pmatrix}
\end{align*}
\end{small}
where 
\begin{small}
$$
	\mathcal{S}_{4} = s \sigma_{r} \mathcal{I}_{4} 
	+ \alpha k \mathrm{curl}^{\ast} \mu_{r}^{-1} \mathrm{curl}
	- \mathcal{F} \mathcal{A}_{1}^{-1} \mathcal{F}^{\ast} .
$$ 
\end{small}
Then for any {\small $\bm{x} = ( \bm{u}, p, \bm{B}, \bm{E} )^{T}$}, we have
\begin{small}
\begin{align*}
( \bm{x},  \mathcal{M}_{\mathcal{L}} \mathcal{A} \bm{x} )_{\mathcal{H}}
& = ( \bm{u}, \bm{u} )_{\mathcal{A}_{1} } 
 - ( \bm{u}, \mathrm{div}^{\ast} p ) + ( \bm{u}, \mathcal{F}^{\ast} \bm{E} )
+  ( p, \mathrm{div} \mathcal{A}_{1}^{-1} \mathrm{div}^{\ast} p ) \\
& \quad -  ( p, \mathrm{div} \mathcal{A}_{1}^{-1} \mathcal{F}^{\ast} \bm{E} ) 
+ ( \bm{B}, \bm{B} )_{ \mathcal{H}_{3} } 
+ \alpha (\bm{B}, \mathrm{curl} \bm{E})_{ \mu_{r}^{-1} } \\
& \quad +  ( \bm{E}, \mathcal{F} \mathcal{A}_{1}^{-1} \mathrm{div}^{\ast} p )
+  (\bm{E}, \mathcal{S}_{4} \bm{E} ) \\
& \geq
\Vert \bm{u} \Vert_{\mathcal{A}_{1} }^{2}
-  \Vert \bm{u} \Vert_{\mathcal{A}_{1} } \Vert \mathrm{div}^{\ast} p \Vert_{ \mathcal{A}_{1}^{-1} }
-  \Vert \bm{u} \Vert_{\mathcal{A}_{1} } \Vert \mathcal{F}^{\ast} \bm{E} \Vert_{ \mathcal{A}_{1}^{-1} }
+ \Vert \mathrm{div}^{\ast} p \Vert_{ \mathcal{A}_{1}^{-1} }^{2}
+ \Vert \bm{B} \Vert_{ \mathcal{H}_{3} }^{2} \\
& \quad
- \Vert \bm{B} \Vert_{ \mathcal{H}_{3} }
\sqrt{\alpha k} \Vert \mathrm{curl} \bm{E} \Vert_{ \mu_{r}^{-1} } 
+ s \Vert \bm{E} \Vert_{ \sigma_{r} }^{2}
+ \alpha k \Vert \mathrm{curl} \bm{E} \Vert_{ \mu_{r}^{-1} }^{2}
- \Vert \mathcal{F}^{\ast} \bm{E} \Vert_{ \mathcal{A}_{1}^{-1} }^{2}.
\end{align*}
\end{small}
Since 
\begin{small}
\begin{align*}
& \Vert \mathcal{F}^{\ast} \bm{E} \Vert_{\mathcal{A}_{1}^{-1} }
= \sup\limits_{ \bm{v} \neq \bm{0} } 
\frac{ ( \mathcal{F}^{\ast} \bm{E}, \bm{v} ) }{ \Vert \bm{v} \Vert_{\mathcal{A}_{1} } }
= \sup\limits_{ \bm{v} \neq \bm{0} } 
\frac{ ( \bm{E},  \mathcal{F} \bm{v} ) }{ \Vert \bm{v} \Vert_{\mathcal{A}_{1} } }
\leq \sup\limits_{ \bm{v} \neq \bm{0} } 
\frac{  \sqrt{s} \Vert \bm{E} \Vert_{\sigma_r}  \sqrt{s} \Vert \bm{v} \times \bm{B}^{-} \Vert_{\sigma_r}   }{ \Vert \bm{v} \Vert_{\mathcal{A}_{1} } }
\leq \frac{1}{\sqrt{2}} \sqrt{s} \Vert \bm{E} \Vert_{ \sigma_r},
\end{align*}
\end{small}
we get
\begin{small}
\begin{align*}
& ( \bm{x},  \mathcal{M}_{\mathcal{L}} \mathcal{A} \bm{x} )_{\mathcal{H}} 
\geq
\Vert \bm{u} \Vert_{\mathcal{A}_{1} }^{2}
-  \Vert \bm{u} \Vert_{\mathcal{A}_{1} } \Vert \mathrm{div}^{\ast} p \Vert_{ \mathcal{A}_{1}^{-1} }
-  \frac{\sqrt{s}}{\sqrt{2} } \Vert \bm{u} \Vert_{\mathcal{A}_{1} } \Vert \bm{E} \Vert_{ \sigma_{r} }
+ \Vert \mathrm{div}^{\ast} p \Vert_{ \mathcal{A}_{1}^{-1} }^{2} 
+ \Vert \bm{B} \Vert_{ \mathcal{H}_{3} }^{2} \\
& \qquad \qquad \qquad
- \Vert \bm{B} \Vert_{ \mathcal{H}_{3} }
\sqrt{\alpha k} \Vert \mathrm{curl} \bm{E} \Vert_{ \mu_{r}^{-1} } 
+ \frac{s}{2} \Vert \bm{E} \Vert_{ \sigma_{r} }^{2}
+ \alpha k \Vert \mathrm{curl} \bm{E} \Vert_{ \mu_{r}^{-1} }^{2} \\
&\qquad \qquad  \quad
= 
\begin{pmatrix}
\xi_{1} \\
\xi_{2} \\
\xi_{3} \\
\xi_{4} \\
\xi_{5}
\end{pmatrix}^{T}
\begin{pmatrix}
1 & -1/2 & 0 & -1/2\sqrt{2} & 0 \\
-1/2 & 1 & 0 & 0 & 0 \\
0 & 0 & 1 & 0 & -1/2 \\
-1/2\sqrt{2} & 0 & 0 & 1/2 & 0 \\
0 & 0 & -1/2 & 0 & 1
\end{pmatrix}
\begin{pmatrix}
\xi_{1} \\
\xi_{2} \\
\xi_{3} \\
\xi_{4} \\
\xi_{5} 
\end{pmatrix},
\end{align*}
\end{small}
where {\small $\xi_{1} = \Vert \bm{u} \Vert_{\mathcal{A}_{1} } $}, {\small $\xi_{2} = \Vert \mathrm{div}^{\ast} p \Vert_{ \mathcal{A}_{1}^{-1} }$}, {\small $\xi_{3} = \Vert \bm{B} \Vert_{ \mathcal{H}_{3} }$}, {\small $\xi_{4} = \sqrt{s} \Vert \bm{E} \Vert_{ \sigma_{r} }$} and {\small $\xi_{5} = \sqrt{\alpha k} \Vert \mathrm{curl} \bm{E} \Vert_{ \mu_{r}^{-1} }$}. It is easy to verify that the matrix in the middle is SPD. Therefore, there exists a constant $\gamma_{0} > 0$ such that
\begin{small}
\begin{align*}
( \bm{x},  \mathcal{M}_{\mathcal{L}} \mathcal{A} \bm{x} )_{\mathcal{H}} 
& \geq
\gamma_{0} \left(
\Vert \bm{u} \Vert_{\mathcal{A}_{1} }^{2}
+ \Vert \mathrm{div}^{\ast} p \Vert_{ \mathcal{A}_{1}^{-1} }^{2}
+ \Vert \bm{B} \Vert_{ \mathcal{H}_{3} }^{2}
+ \Vert \bm{E} \Vert_{ \mathcal{H}_{4} }^{2}
\right)
\geq
\min \left\{
\gamma_{0}, \gamma_{0} \zeta^{2}
\right\}
( \bm{x}, \bm{x} )_{ \mathcal{M}^{-1} }.
\end{align*}
\end{small}
The last inequality comes from the fact that
\begin{small}
\begin{align*}
& \Vert \mathrm{div}^{\ast} p \Vert_{ \mathcal{A}_{1}^{-1} } 
= \sup\limits_{ \bm{v} \neq \bm{0} } \frac{ (\mathrm{div}^{\ast} p, \bm{v}) }{ \Vert \bm{v} \Vert_{\mathcal{A}_{1} } }
= \sup\limits_{ \bm{v} \neq \bm{0} } \frac{ ( p, \mathrm{div} \bm{v}) }{ \Vert \bm{v} \Vert_{\mathcal{A}_{1} } }
\geq \zeta \Vert p \Vert_{ \mathcal{H}_{2} }.
\end{align*}
\end{small}
The above estimate leads to the lower bound {\small $\gamma$}. On the other hand, the upper bound $\Gamma$ can be obtained directly by the boundedness of each term.
\end{proof}

To reduce the time and computation cost of {\small $\mathcal{M}_{\mathcal{L}}$}, we replace its diagonal blocks by their spectral equivalent SPD approximations except that of {\small $\bm{B}$}. The implementation is the same as that in remark \ref{remark:divB_free}. Such modification gives rise to the following operator
\begin{small}
\begin{align} \label{eqn:barL-no-stable}
& \widehat{ \mathcal{M} }_{\mathcal{L}} = \begin{pmatrix}
\mathcal{Q}_{1}^{-1} & 0 & 0 & 0 \\
\mathrm{div} &  \mathcal{Q}_{2}^{-1} & 0 & 0 \\
0 & 0 & \alpha k^{-1} \mu_{r}^{-1} \mathcal{I}_{3} & 0 \\
\mathcal{F} & 0 & - \alpha \mathrm{curl}^{\ast} \mu_{r}^{-1} &  \mathcal{Q}_{4}^{-1}
\end{pmatrix}^{-1}.
\end{align}
\end{small}  
We note that, for the magnetic field, we can apply the same approach with \eqref{eqn:vBm} at each preconditioning step in order to preserve divergent-free condition exactly on the discrete level. And we can prove a similar conclusion to that in Theorem \ref{thm:lower_exact_solve} for $\widehat{ \mathcal{M} }_{\mathcal{L}}$ using a different inner product in the analysis.

%%%%%%%%%%%
% new theorem and proof for inexact solve.
\begin{theorem}\label{thm:lower_inexact_solve}
If {\small $k \leq k_{0}$} and the condition \eqref{def:Qx} holds, then there exist constants {\small $\gamma$} and {\small $\Gamma$} that are independent of  the mesh size {\small $h$}, time step size {\small $k$}, and physical parameters {\small $Rm$}, {\small $s$}, {\small $\mu_r$}, and {\small $\sigma_r$}, such that for all {\small $\bm{x} \neq \bm{0}$}, the condition \eqref{def:FOV} holds with the inner product {\small $(x, y)_{\mathcal{M}^{-1}} $} induced by {\small $\mathcal{M} = 
\mathrm{diag} \left(
\mathcal{Q}_{1}, \mathcal{Q}_{2},\mathcal{H}_{3}^{-1}, \mathcal{Q}_{4} 
\right)$} provided that 
{\small $\Vert \mathcal{I}_{1} - \mathcal{Q}_{1} \mathcal{A}_{1}  \Vert_{ \mathcal{A}_{1} } \leq \rho$}, with {\small $0 \leq \rho < 0.289$}.
\end{theorem}

\begin{proof}
By calculation, it can be seen that
\begin{small}
\begin{align*}
& \widehat{\mathcal{M}}_{\mathcal{L}} \mathcal{A} = 
\begin{pmatrix}
\mathcal{Q}_{1} \mathcal{A}_{1}
& - \mathcal{Q}_{1}\mathrm{div}^{\ast} & 0 
& \mathcal{Q}_{1} \mathcal{F}^{\ast} \\
 \mathcal{Q}_{2} \mathrm{div} ( \mathcal{I}_{1} - \mathcal{Q}_{1}\mathcal{A}_{1} )
& \mathcal{Q}_{2} \widehat{ \mathcal{S} }_{2} 
& 0 & - \mathcal{Q}_{2} \mathrm{div} \mathcal{Q}_{1} \mathcal{F}^{\ast} \\
0 & 0 & \mathcal{I}_{3} & k \mathrm{curl} \\
 \mathcal{Q}_{4}\mathcal{F} (  \mathcal{I}_{1} -\mathcal{Q}_{1} \mathcal{A}_{1} )
&  \mathcal{Q}_{4} \mathcal{F} \mathcal{Q}_{1} \mathrm{div}^{\ast}
& 0
& \mathcal{Q}_{4}\widehat{ \mathcal{S} }_{4}
\end{pmatrix}.
\end{align*}
\end{small}
where {\small $\widehat{ \mathcal{S} }_{2} = \mathrm{div} \mathcal{Q}_{1}\mathrm{div}^{\ast}$}, {\small $\widehat{ \mathcal{S} }_{4} = s \sigma_{r} \mathcal{I}_{4}
+ \alpha k \mathrm{curl}^{\ast} \mu_{r}^{-1} \mathrm{curl} 
- \mathcal{F} \mathcal{Q}_{1} \mathcal{F}^{\ast}$}. Therefore, for any {\small $\bm{x} = ( \bm{u}, p, \bm{B}, \bm{E} )^{T}$},
\begin{small}
\begin{align*}
( \bm{x}, \widehat{ \mathcal{M} }_{\mathcal{L}} \mathcal{A} \bm{x} )_{\mathcal{M}^{-1}}
& = ( \bm{u}, \mathcal{A}_{1} \bm{u}) 
- ( \bm{u},  \mathrm{div}^{\ast} p) 
+ (\bm{u}, \mathcal{F}^{\ast} \bm{E}) 
+ (p,  \mathrm{div} ( \mathcal{I}_{1} - \mathcal{Q}_{1}\mathcal{A}_{1} ) \bm{u} ) 
+ (p, \widehat{ \mathcal{S} }_{2} p) \\
&\quad 
- (p, \mathrm{div} \mathcal{Q}_{1}\mathcal{F}^{\ast} \bm{E} ) 
+ ( \bm{B}, \bm{B} )_{ \mathcal{H}_{3} } 
+ ( \bm{B}, k \mathrm{curl} \bm{E} )_{ \mathcal{H}_{3} } 
+ ( \bm{E},  \mathcal{F} ( \mathcal{I}_{1} -  \mathcal{Q}_{1}\mathcal{A}_{1}) \bm{u}) \\
& \quad 
+ ( \bm{E}, \mathcal{F} \mathcal{Q}_{1} \mathrm{div}^{\ast} p) 
+ ( \bm{E}, \widehat{ \mathcal{S} }_{4} \bm{E} ) 
\\
& \geq
\Vert \bm{u} \Vert_{\mathcal{A}_{1}}^{2}
- \Vert \bm{u} \Vert_{\mathcal{A}_{1}}
\Vert \mathcal{F}^{\ast} \bm{E} \Vert_{\mathcal{A}_{1}^{-1}}
- \Vert \mathcal{A}_{1} \bm{u} \Vert_{\mathcal{Q}_{1}}
\Vert \mathrm{div}^{\ast} p \Vert_{\mathcal{Q}_{1} } 
+ \Vert \mathrm{div}^{\ast} p \Vert_{\mathcal{Q}_{1} }^{2} 
\\
& \quad
+ \Vert \bm{B} \Vert_{ \mathcal{H}_{3} }^{2}
- \Vert \bm{B} \Vert_{ \mathcal{H}_{3} }
\sqrt{\alpha k} \Vert \mathrm{curl} \bm{E} \Vert_{ \mu_{r}^{-1} }
- \rho
\Vert \bm{u} \Vert_{\mathcal{A}_{1}}
\Vert \mathcal{F}^{\ast} \bm{E} \Vert_{\mathcal{A}_{1}^{-1}} \\
& \quad
+ s \Vert \bm{E} \Vert_{ \sigma_{r} }^{2} 
+ \alpha k \Vert \mathrm{curl} \bm{E} \Vert_{ \mu_{r}^{-1} }^{2}
- \Vert \mathcal{F}^{\ast} \bm{E} \Vert_{\mathcal{Q}_{1} }^{2}.
\end{align*}
\end{small}
Since {\small $
\Vert \mathcal{I}_{1} - \mathcal{Q}_{1} \mathcal{A}_{1}  \Vert_{ \mathcal{A}_{1} } \leq \rho$} implies
\begin{small}
\begin{align*}
& (1-\rho) (\bm{v}, \bm{v})_{\mathcal{A}_{1}^{-1}} \leq (\bm{v}, \bm{v})_{\mathcal{Q}_{1}} 
\leq (1+\rho) (\bm{v}, \bm{v})_{\mathcal{A}_{1}^{-1}}, \\
& (1+\rho)^{-1} (\bm{v}, \bm{v})_{\mathcal{A}_{1} }
\leq (\bm{v}, \bm{v})_{\mathcal{Q}_{1}^{-1} }
\leq (1-\rho)^{-1} (\bm{v}, \bm{v})_{\mathcal{A}_{1}},
\end{align*}
\end{small}
and {\small $\Vert \mathcal{F}^{\ast} \bm{E}  \Vert_{\mathcal{A}_{1}^{-1}} \leq \frac{1}{\sqrt{2}} \sqrt{s} \Vert \bm{E} \Vert_{\sigma_r}$} (it is shown in the proof of Theorem \ref{thm:lower_exact_solve}), we get
\begin{small}
\begin{align*}
& ( \bm{x}, \widehat{ \mathcal{M} }_{\mathcal{L}} \mathcal{A} \bm{x} )_{\mathcal{M}^{-1}}
\geq
\Vert \bm{u} \Vert_{\mathcal{A}_{1}}^{2}
- \frac{\sqrt{s}}{\sqrt{2} } \Vert \bm{u} \Vert_{\mathcal{A}_{1}}
\Vert \bm{E} \Vert_{\sigma_{r} }
- (1+\rho)
\Vert \bm{u} \Vert_{\mathcal{A}_{1}}
\Vert \mathrm{div}^{\ast} p \Vert_{\mathcal{Q}_{1} } 
+ \Vert \mathrm{div}^{\ast} p \Vert_{\mathcal{Q}_{1} }^{2} \\
& \qquad \qquad \qquad
+ \Vert \bm{B} \Vert_{ \mathcal{H}_{3} }^{2}
- \Vert \bm{B} \Vert_{ \mathcal{H}_{3} }
\sqrt{\alpha k} \Vert \mathrm{curl} \bm{E} \Vert_{ \mu_{r}^{-1} }
- \rho \frac{\sqrt{s}}{\sqrt{2} }
\Vert \bm{u} \Vert_{\mathcal{A}_{1}}
\Vert \bm{E} \Vert_{\sigma_{r} } \\
& \qquad \qquad \qquad
+ \frac{1-\rho}{2} s \Vert \bm{E} \Vert_{ \sigma_{r} }^{2} 
+ \alpha k \Vert \mathrm{curl} \bm{E} \Vert_{ \mu_{r}^{-1} }^{2} \\
& \qquad
= \begin{pmatrix}
\xi_{1} \\
\xi_{2} \\
\xi_{3} \\
\xi_{4} \\
\xi_{5} 
\end{pmatrix}^{T}
\begin{pmatrix}
1 & -(1+\rho)/2 & 0 & -(1+\rho)/2\sqrt{2} & 0 \\
-(1+\rho)/2 & 1 & 0 & 0 & 0 \\
0 & 0 & 1 & 0 & -1/2 \\
-(1+\rho)/2\sqrt{2} & 0 & 0 & (1-\rho)/2 & 0 \\
0 & 0 & -1/2 & 0 & 1
\end{pmatrix}
\begin{pmatrix}
\xi_{1} \\
\xi_{2} \\
\xi_{3} \\
\xi_{4} \\
\xi_{5} 
\end{pmatrix},
\end{align*}
\end{small}
where {\small $\xi_{1} = \Vert \bm{u} \Vert_{\mathcal{A}_{1} } $}, {\small $\xi_{2} = \Vert \mathrm{div}^{\ast} p \Vert_{ \mathcal{Q}_{1} }$}, {\small $\xi_{3} = \Vert \bm{B} \Vert_{ \mathcal{H}_{3} }$}, {\small $\xi_{4} = \sqrt{s} \Vert \bm{E} \Vert_{ \sigma_{r} }$} and {\small $\xi_{5} = \sqrt{\alpha k} \Vert \mathrm{curl} \bm{E} \Vert_{ \mu_{r}^{-1} }$}. we can verify that the matrix in the middle is SPD when {\small $0 \leq \rho < 0.289$}. Therefore, there exists a constant {\small $\gamma_{0} > 0$} such that
\begin{small}
\begin{align*}
( \bm{x}, \widehat{ \mathcal{M} }_{\mathcal{L}} \mathcal{A} \bm{x} )_{\mathcal{M}^{-1} } 
& \geq
\gamma_{0} \left(
\Vert \bm{u} \Vert_{\mathcal{A}_{1} }^{2}
+ (1-\rho) \Vert \mathrm{div}^{\ast} p \Vert_{ \mathcal{A}_{1}^{-1} }^{2}
+ \Vert \bm{B} \Vert_{ \mathcal{H}_{3} }^{2}
+ \Vert \bm{E} \Vert_{ \mathcal{H}_{4} }^{2}
\right) \\
& \geq
\gamma_{0} \left(
\Vert \bm{u} \Vert_{\mathcal{A}_{1} }^{2}
+ (1-\rho) \zeta^{2} \Vert p \Vert_{ \mathcal{H}_{2} }^{2}
+ \Vert \bm{B} \Vert_{ \mathcal{H}_{3} }^{2}
+ \Vert \bm{E} \Vert_{ \mathcal{H}_{4} }^{2}
\right) \\
& \geq
\min \left\{
\gamma_{0} (1-\rho), 
\gamma_{0} (1-\rho) \zeta^{2} c_{1, 2}^{-1},
\gamma_{0} c_{1, 3}^{-1}, 
\gamma_{0} c_{1, 4}^{-1}
\right\}
( \bm{x}, \bm{x} )_{\mathcal{M}^{-1} },
\end{align*}
\end{small}
which leads to the lower bound {\small $\gamma$}.  On the other hand, the upper bound {\small $\Gamma$} can be obtained directly from the boundedness of each term.
\end{proof}

\begin{remark}
Condition {\small $
\Vert \mathcal{I}_{1} - \mathcal{Q}_{1} \mathcal{A}_{1}  \Vert_{ \mathcal{A}_{1} } \leq \rho
$}, with {\small $0 \leq \rho < 0.289$}, means that we should solve the velocity block accurately enough.  We comment that this condition can be relaxed by introducing parameters $\rho_1$ and $\rho_2$ in front of the diagonal block {\small $\mathcal{Q}_{2}$} and {\small $\mathcal{Q}_{4}$}, respectively.  In that case, the proof is very similar to that of the Theorem \ref{thm:lower_inexact_solve} by choosing different numbers in the Young's inequality and picking appropriate $\rho_1$ and $\rho_2$. Thus, we omit the proof.  
\end{remark}

% generalization.
\section{Generalizations to other discretizations} \label{sec:general}
In this section, we generalize the robust preconditioners based on the well-posedness to other discretizations of incompressible MHD systems that have been developed in the literature. As examples, we will give detailed discussions on the discretization proposed by Gunzburger et al. \cite{Gunzburger.M;Meir.A;Peterson.J.1991a} for the stationary incompressible MHD system and that for non-stationary incompressible MHD system given in \cite{Bris.C;Lelievre.T.2006a,Prohl.A.2008a}. Generalizations to other discretizations are similar.

\subsection{$H^{1}$ discretization for a stationary incompressible MHD system}
The stationary incompressible MHD equations discussed in \cite{Gunzburger.M;Meir.A;Peterson.J.1991a} is:
\begin{small}
\begin{align}
\begin{cases}
& \bm{u} \cdot \nabla \bm{u} - \frac{1}{Re} \Delta \bm{u} + \nabla p - \alpha ( \nabla \times \bm{B} ) \times \bm{B} = \bm{f},  \\
& \nabla \cdot \bm{u} = 0,  \\
& -\nabla \times ( \bm{u} \times \bm{B} ) + \frac{1}{Rm} \nabla \times \nabla \times \bm{B} = 0,  \\
& \nabla \cdot \bm{B} = 0,
\end{cases}\label{eq:stationary1991}
\end{align}
\end{small}
with homogeneous boundary conditions
\begin{small}
\begin{align*}
& \bm{u} = 0, \quad
\bm{n} \cdot \bm{B} = 0, \quad
\bm{n} \times ( \nabla \times \bm{B} ) = \bm{0}.
\end{align*}
\end{small}
The finite element space for {\small $(\bm{u}, p)$} is still {\small $\bm{V}_{h} \times Q_{h}$}. However, due to the boundary condition of {\small $\bm{B}$}, a special Sobol\'{e}v space is used,
\begin{small}
\begin{align*}
H_{n}^{1} (\Omega)^3 = \left\{ \bm{u} \in H^{1}(\Omega)^{3}, ~ \bm{n} \cdot \bm{u} \mid_{\partial \Omega} = 0  \right\},
\end{align*}
\end{small}
and norm for this space is
\begin{small}
\begin{align} \label{def:norm-H1n}
& \Vert \bm{u} \Vert^{2}_{n,1} = 
\frac{\alpha}{Rm} 
\left[
\Vert \nabla \cdot \bm{u} \Vert^{2}
+ \Vert \nabla \times \bm{u} \Vert^{2}
\right], 
~ \forall \bm{u} \in H_{n}^{1}(\Omega)^3.
\end{align}
\end{small}
Then the finite element space for {\small $\bm{B}$} is {\small $H_{n,h}^1(\Omega)^3 \subset H_n^1(\Omega)^3$}.

The full discretization based on symmetric Picard linearization of \eqref{eq:stationary1991} is: find {\small $(\bm{u}_{h}, \bm{B}_{h}, p_{h} ) \in \bm{V}_{h} \times H_{n,h}^{1}(\Omega)^{3} \times Q_{h}$} such that for any {\small $(\bm{v}_{h}, \bm{C}_{h}, q_{h} ) \in \bm{V}_{h} \times H_{n,h}^{1}(\Omega)^{3} \times Q_{h}$},
\begin{small}
\begin{align*}
\begin{cases}
& d( \bm{u}_{h}^{-},  \bm{u}_{h}^{-},  \bm{v}_{h} )     
  + \frac{1}{Re} (\nabla \bm{u}_{h}, \nabla \bm{v}_{h} ) 
  - \alpha (\nabla\times \bm{B}_{h},\bm{B}^{-}_{h} \times \bm{v}_{h} ) 
  - (p_{h},\nabla\cdot \bm{v}_{h}) 
 = (\bm{f}_{h},\bm{v}_{h}),  \\
& - \alpha ( \bm{u}_{h} \times\bm{B}^{-}_{h}, \nabla\times \bm{C}_{h} )
+ \frac{\alpha}{Rm} \left[ (\nabla\times \bm{B}_{h}, \nabla\times \bm{C}_{h} )
+(\nabla\cdot \bm{B}_{h}, \nabla\cdot \bm{C}_{h} ) \right]= {0}, \\ 
& (\nabla\cdot \bm{u}_{h}, q_{h}) =0,
\end{cases}
\end{align*}
\end{small}
where {\small $\bm{u}_{h}^{-}$} and {\small $\bm{B}_{h}^{-}$} stand for solutions from last nonlinear iteration step. The operator form of this system is (we flip the signs of equations to make the system symmetric):
\begin{small}
\begin{align*}
\mathcal{L}_{1} \bm{x} = \bm{F}_{1}  & \Rightarrow
\begin{pmatrix}
- Re^{-1} \Delta & - \mathrm{div}^{\ast} & \alpha \bm{B}^{-} \times \mathrm{curl} \\
- \mathrm{div} & 0 & 0 \\
\left( \alpha \bm{B}^{-} \times \mathrm{curl} \right)^{\ast} & 0 &  
-\frac{\alpha}{Rm} \left( \mathrm{div}^{\ast} \mathrm{div} + \mathrm{curl}^{\ast} \mathrm{curl} \right)
\end{pmatrix}
\begin{pmatrix}
\bm{u} \\ p \\ \bm{B}
\end{pmatrix} = \begin{pmatrix}
\bm{h}_{1} \\ g \\ -\bm{h}_{2}
\end{pmatrix}.
\end{align*}
\end{small}

In \cite{Gunzburger.M;Meir.A;Peterson.J.1991a}, it is proved that the above discretization is well-posed with respect to the canonical norms of {\small $\bm{V}_{h}$}, {\small $Q_{h}$} and the norm defined by \eqref{def:norm-H1n} for {\small $H_{n,h}^1(\Omega)^3$}.  This implies robust norm-equivalent preconditioner as follows:
\begin{small}
\begin{align*}
\mathcal{B}_{1} & = 
\begin{pmatrix}
 - Re^{-1} \Delta & 0 & 0 \\
 0 & \mathcal{I}_{2} & 0 \\
0 & 0 & \frac{\alpha}{Rm} \left( \mathrm{div}^{\ast} \mathrm{div} + \mathrm{curl}^{\ast} \mathrm{curl} \right)
\end{pmatrix}^{-1}. 
\end{align*}
\end{small}
As mentioned before, we can design FOV-equivalent preconditioner for the original non-symmetric system, which is  given by
\begin{small}
\begin{align*}
\mathcal{M}_{\mathcal{L}, 1} & = 
\begin{pmatrix}
 - Re^{-1} \Delta & 0 & 0 \\
 \mathrm{div} & \mathcal{I}_{2} & 0 \\
-\left( \alpha \bm{B}^{-} \times \mathrm{curl} \right)^{\ast} & 0 & 
\frac{\alpha}{Rm} 
\left( \mathrm{div}^{\ast} \mathrm{div} + \mathrm{curl}^{\ast} \mathrm{curl} \right) 
\end{pmatrix}^{-1}.
\end{align*}
\end{small}
Use the similar argument in previous sections, we can show that these preconditioner is uniform with respect to {\small $h$}. The first diagonal block is Poisson-like, for which multigrid methods are efficient. Jacobi iterative method or other simple iterative methods are efficient for the second block. And the third diagonal block is a vector Laplacian, for which multigrid methods are effective as well.  

\subsection{$H(\mathrm{curl})$ discretization for a non-stationary incompressible MHD system}
The mathematical model discussed in \cite{Bris.C;Lelievre.T.2006a,Prohl.A.2008a} is
\begin{small}
\begin{align}
\begin{cases}
& \frac{\partial \bm{u}}{\partial t} +  \bm{u}\cdot \nabla \bm{u} 
-\frac{1}{Re} \Delta \bm{u} + \nabla p 
-  \alpha (\nabla\times\bm{B}) \times \bm{B}=\bm{f}, \\
& \nabla\cdot \bm{u} =0, \\
& \frac{\partial \bm{B}}{\partial t} - \nabla\times (\bm{u}\times\bm{B})
+ \frac{1}{Rm} \nabla\times \nabla\times \bm{B} - \nabla r= 0, \\
& \nabla\cdot \bm{B} = 0,
\end{cases}\label{eq:mhd_schotzau}
\end{align}
\end{small}
with boundary conditions
\begin{small}
\begin{align*}
& \bm{u} = 0, \quad
\bm{n} \times \bm{B} = \bm{0}.
\end{align*}
\end{small}
The full discretization based on symmetric Picard linearization of \eqref{eq:mhd_schotzau} is: find {\small $(\bm{u}_{h}^{n}, p_{h}^{n}, \bm{B}_{h}^{n}, r_{h}^{n} ) \in \bm{V}_{h} \times Q_{h} \times \bm{V}_{h}^{c} \times H^{1}_{h}$} such that for any {\small $(\bm{v}_{h}, q_{h}, \bm{C}_{h}, l_{h} ) \in \bm{V}_{h} \times Q_{h} \times \bm{V}_{h}^{c} \times H^{1}_{h}$},
\begin{small}
\begin{align}
\begin{cases}
& k^{-1} \left( \boldsymbol{u}_h^{n}-\boldsymbol{u}_h^{n-1}, \boldsymbol{v}_{h} \right) 
+  d( \bm{u}_{h}^{n-1}, \bm{u}_{h}^{n-1}, \bm{v}_{h} )
+ {1 \over Re} \left(\nabla \boldsymbol{u}_h^{n}, \nabla \boldsymbol{v}_h \right)  \\
& \qquad
 - \alpha \left( ( \nabla \times \bm{B}_{h} ) \times \boldsymbol{B}_h^{n-1},\boldsymbol{v}_h \right)
 - \left( p_h^{n},\nabla\cdot \boldsymbol{v}_h \right) 
= \left( \boldsymbol{f}_h^{n},\boldsymbol{v}_h \right), \\
& \alpha k^{-1} \left( \boldsymbol{B}_h^{n} - \boldsymbol{B}_h^{n-1} , 
\boldsymbol{C}_{h} \right) 
- \alpha \left(  \boldsymbol{u}_h^{n} \times \bm{B}_{h}^{n-1}, \nabla \times \boldsymbol{C}_{h} \right) 
+ \frac{\alpha}{Rm} ( \nabla \times \bm{B}_{h}^{n}, \nabla \times \bm{C}_{h} ) \\
& \qquad
- \alpha ( \bm{C}_{h}, \nabla r_{h} )
= 0, \\
& \left( \nabla\cdot \boldsymbol{u}_h^{n}, q_h \right) =0, \\
& \alpha ( \bm{B}_{h}^{n}, \nabla l_{h} ) = 0.
\end{cases}\label{eq:weak_mhd_schotzau}
\end{align}
\end{small}

The operator form of discretization \eqref{eq:weak_mhd_schotzau} is (we flip the signs of equations to make the system symmetric):
\begin{small}
\begin{align*}
& \mathcal{L}_{2} \bm{x} = \bm{F} \Rightarrow
\begin{pmatrix}
k^{-1} \mathcal{I}_{1} - Re^{-1} \Delta &  - \mathrm{div}^{\ast} & \alpha \bm{B}^{-} \times \mathrm{curl} & 0 \\
- \mathrm{div} & 0 & 0 & 0 \\
( \alpha \bm{B}^{-} \times \mathrm{curl} )^{\ast} & 0 & 
- \left( \alpha k^{-1} \mathcal{I}_{3} + \frac{\alpha}{Rm} \mathrm{curl}^{\ast} \mathrm{curl} \right) 
 & \alpha \mathrm{grad} \\ 
0 &  0 & ( \alpha \mathrm{grad} )^{\ast} & 0 
\end{pmatrix} 
\begin{pmatrix}
\bm{u} \\ p \\ \bm{B} \\ r
\end{pmatrix} = \begin{pmatrix}
\bm{h}_{1} \\ g_{1} \\ -\bm{h}_{2} \\ g_{2} 
\end{pmatrix}.
\end{align*}
\end{small}

Based on the analysis in the literature \cite{Bris.C;Lelievre.T.2006a,Prohl.A.2008a}, we can prove that discretization \eqref{eq:weak_mhd_schotzau} is well-posed with respect to weighted norms in the finite element spaces {\small $\bm{V}_{h}, \ Q_{h},  \ \bm{V}_{h}^{c}$}, and {\small $ H^{1}_{h}$}. Therefore, robust norm-equivalent preconditioner for this system is:
\begin{small}
\begin{align*}
\mathcal{B}_{2} & = 
\begin{pmatrix}
k^{-1} \mathcal{I}_{1} - Re^{-1}\Delta + r \mathrm{div}^{\ast}\mathrm{div} & 0 & 0 & 0 \\
0 & k \mathcal{I}_{2} & 0 & 0 \\
0 & 0 & \alpha k^{-1} \mathcal{I}_{3} + \frac{\alpha}{Rm} \mathrm{curl}^{\ast} \mathrm{curl} & 0 \\
0 & 0 & 0 & k \alpha^{-1} ( \mathcal{I}_{4} - \Delta )
\end{pmatrix}^{-1}.
\end{align*}
\end{small}
And an FOV-equivalent preconditioner for the original non-symmetric system is:
\begin{small}
\begin{align*}
\mathcal{M}_{\mathcal{L}, 2 } & = 
\begin{pmatrix}
k^{-1} \mathcal{I}_{1} - Re^{-1}\Delta + r \mathrm{div}^{\ast}\mathrm{div} & 0 & 0 & 0 \\
\mathrm{div} & k \mathcal{I}_{2} & 0 & 0 \\
- ( \alpha \bm{B}^{-} \times \mathrm{curl} )^{\ast} & 0 &
\alpha k^{-1} \mathcal{I}_{3} + \frac{\alpha}{Rm} \mathrm{curl}^{\ast} \mathrm{curl}  & 0 \\
0 & 0 & ( \alpha \mathrm{grad} )^{\ast} & k \alpha^{-1} ( \mathcal{I}_{4} - \Delta )
\end{pmatrix}^{-1}.
\end{align*}
\end{small}

When {\small $k$} is sufficiently small, we can show these preconditioner is robust with respect to discretization parameters.  Multigrid methods are effective to precondition the first diagonal block. Jacobi method or other simple iterative methods are efficient for the second diagonal block. HX preconditioner is powerful for the third diagonal block.  And the fourth diagonal block corresponds to a Poisson problem for which multigrid methods are powerful as well.

% numerics
\section{Numerical experiments}\label{sec:numeric_experiments}
In this section, we carry out numerical experiments for both $2$D and $3$D incompressible MHD systems to demonstrate the robustness of the preconditioners introduced in previous sections.  As mentioned before, we mainly focus on the structure-preserving discretization without the stabilization term, namely, scheme \eqref{eqn:problem-a0-1}.

In both $2$D and $3$D implementation, we use {\small $P_2$}-{\small $P_0$} to discretize the velocity and the pressure, and use the lowest order elements to discretize the electric and the magnetic field. We implement the structure-preserving discretization based on the FEniCs \cite{Logg.A;Mardal.K;Wells.G.2012a} and the block preconditioners based on the FASP package \cite{Xu.J.a}. We run all the experiments on a Dell workstation with $12$ GB total memory.

\subsection{Implementation of block preconditioner}
In this section, we discuss implementation details of the proposed block preconditioners when they are used to precondition the Krylov methods such as MINRES, GMRES, and flexible GMRES (FGMRES). As we know, inverting block diagonal or triangular preconditioners ends up with inverting diagonal blocks. Therefore, the main implementation issue is how to invert those diagonal blocks. When inverting each diagonal block exactly, we call direct solvers implemented in the UMFPack package \cite{Davidson.P.2001a}.  While when inverting each diagonal block approximately, we mainly call preconditioned conjugate gradient (PCG) method with a relative large tolerance, since each block in our preconditioners is SPD. For example, the tolerance of PCG in terms of {\small $l_2$} norm of the relative residual is {\small $10^{-3}$}, which is sufficient to achieve \eqref{def:Qx}.

It is well-known that effective preconditioners for PCG method are necessary in order to achieve overall robustness of the preconditioners and reduce the computational cost as well. In both $2$D and $3$D, since the velocity block is a Poisson-like problem, multigrid methods are good candidates. In our experiments, we use AMG method implemented in FASP package.  For the pressure, the corresponding block is a mass matrix of finite element of {\small $Q_{h}$}, we use Jacobi preconditioner for it. And for the magnetic block, we update it using \eqref{eqn:vBm} without inverting the diagonal block explicitly.  This not only ensures the exactness of the divergence-free condition for the magnetic field but also saves the computation time.

However, the implementations of the diagonal block corresponding to the electric field are different in $2$D and $3$D.  This is because the vector electric field degenerates to a scalar in 2D.  In 3D, the diagonal block for the electric field {\small $\bm{E}$} corresponds to the following PDE 
\begin{small}
\begin{align*}
& \alpha^{2} \mathrm{curl} ~ \mathrm{curl} \bm{E} + s \bm{E} = \bm{f}.
\end{align*}
\end{small}
Therefore, HX preconditioner \cite{Hiptmair.R;Xu.J.2007a} implemented in FASP package is a powerful preconditioner for it.  However, in $2$D, the diagonal block of the electric field {\small $E$} corresponds to {\small $\alpha^{2} \mathrm{rot} ~ \mathrm{curl} E + s E = f$} on the continuous level. 
As {\small $\mathrm{rot}~\mathrm{curl} u = - \Delta u$}, this PDE is actually {\small $- \alpha^{2} \Delta E + s E = f$},
which is a Poisson-like problem.  Obviously, robust preconditioner for the above problem depends on the ratio between the parameter {\small $\alpha$} and {\small $s$}.  When {\small $\alpha$} is relatively large, {\small $-\Delta$} dominates and we use AMG as the preconditioner.  When {\small $\alpha$} is relatively small, the lower order term dominates and we use Jacobi as the preconditioner.  

\subsection{Convergence tests}
We first perform convergence tests to verify the correctness of our code.  For the sake of accuracy, we use $2$-step BDF for the temporal discretization.  All the physical parameters are set to be {\small $1$} in the convergence tests.  The initial conditions, boundary conditions, and right hand sides are computed based on the given exact solutions. The tolerance for the preconditioned Krylov methods is {\small $10^{-10}$}. 

\subsubsection{$2$D Convergence Test}
The exact solutions chosen for the $2$D convergence test are:
\begin{small}
\begin{align*}
\bm{u} & = \begin{pmatrix}
e^{t} \cos y \\ 0
\end{pmatrix}, 
\quad \bm{B} = \begin{pmatrix}
0 \\ \sin t \cos x
\end{pmatrix}, 
\quad p = -x \cos y, 
\quad E = \sin x. 
\end{align*}
\end{small}

\vskip-15pt
\begin{figure}[H]
    \centering
        \subfigure[Error versus mesh size $h$ ( $k = 0.01$ and $t = 0.1$)]{ \includegraphics[width=0.27\textwidth]{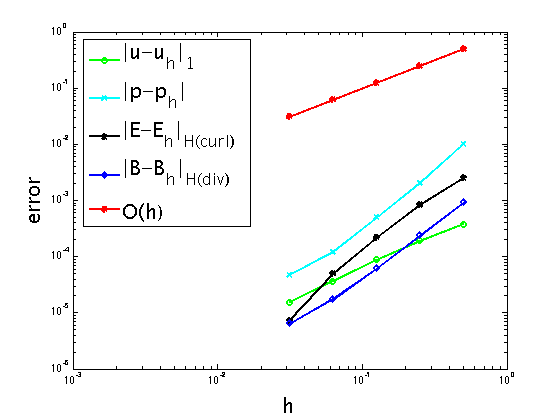} }
        \qquad
        \subfigure[Error versus time step size $k$ ($h = \nicefrac{1}{32}$ and $t = 0.8$)]{ \includegraphics[width=0.27\textwidth]{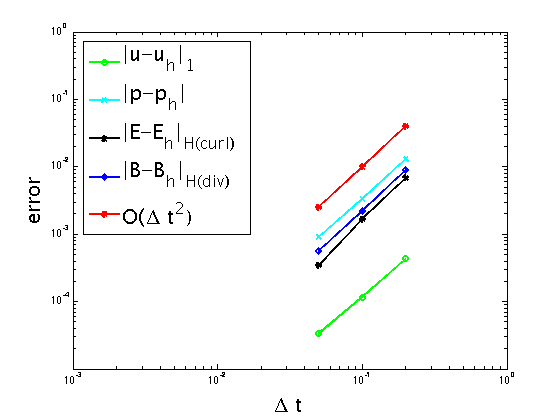} }
        \qquad
        \subfigure[$\Vert \mathrm{div} \bm{B} \Vert$ ($k=0.01$)]{ \includegraphics[width=0.27\textwidth]{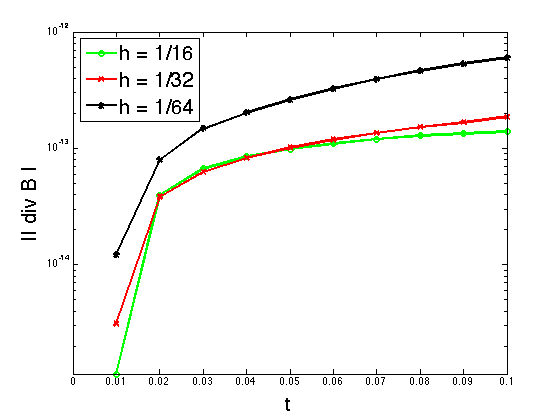} }
        \caption{Numerical results of $2$D convergence test. } \label{fig:convergence_test_rlt2D}
\end{figure}

Based on the results shown in Figure \ref{fig:convergence_test_rlt2D} (a) and (b), we can see that both spatial and temporal errors converge in the optimal order, i.e., {\small $\mathcal{O}(h)$} and {\small $\mathcal{O}(k^2)$}.  We also plot {\small $\Vert \mathrm{div} \bm{B} \Vert$} in Figure \ref{fig:convergence_test_rlt2D} (c) which is about {\small $10^{-12}$} and confirms the exactness of the divergence-free condition due to the structure-preserving discretization.  One comment is that {\small $\Vert \mathrm{div} \bm{B} \Vert$} increases as the time evolves.  This is due to the accumulation of the round-off error during the computation.  We can reduce such accumulation simply by updating {\small $\bm{B}_{h}^{n}$} with
\begin{small}
\begin{align}
& \bm{B}_{h}^{n} = \bm{B}_{h}^{0} - \nabla \times \left( \sum\limits_{i = 1}^{n} k \bm{E}_{h}^{i} \right),
\label{eq:correctB}
\end{align}
\end{small}
at the end of each time step.  \eqref{eq:correctB} is a mathematical equivalent formula to the original system. Therefore, we can preserve the divergence-free condition exactly on the discrete level. 

\subsubsection{$3$D Convergence Test}
The exact solution chosen for this test is:
\begin{small}
\begin{align*}
& \bm{u} = \begin{pmatrix}
e^{t} \cos y \\ 0 \\ 0
\end{pmatrix},
\quad \bm{E} = \begin{pmatrix}
0 \\ \cos x \\ 0
\end{pmatrix},  
\quad \bm{B} = \begin{pmatrix}
0 \\ 0 \\ \sin t \cos x
\end{pmatrix}, 
\quad p = -x \cos y. 
\end{align*} 
\end{small}

\vskip-15pt
\begin{figure}[H]
\centering
\subfigure[Error versus mesh size $h$ ($k = 0.01$ and $t=0.1$)]{ \includegraphics[width=0.27\textwidth]{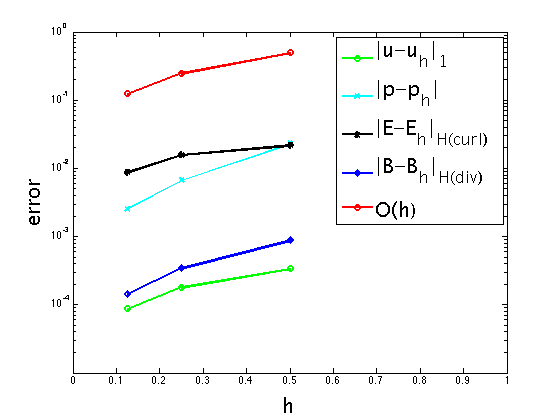} }
\qquad
\subfigure[Error versus time step size $k$ ($h = \nicefrac{1}{12}$ and $t=1$)]{ \includegraphics[width=0.27\textwidth]{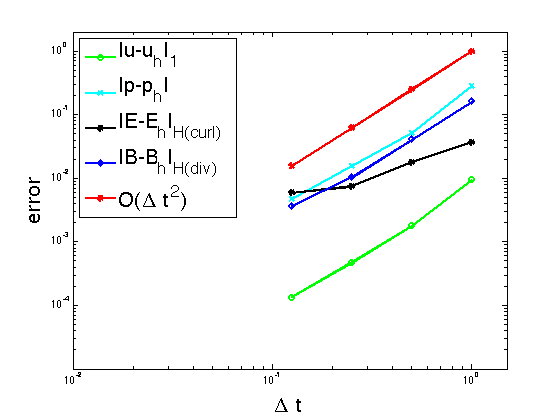} }
\qquad
\subfigure[$\Vert \mathrm{div} \bm{B} \Vert$ ($k=0.01$)]{ \includegraphics[width=0.27\textwidth]{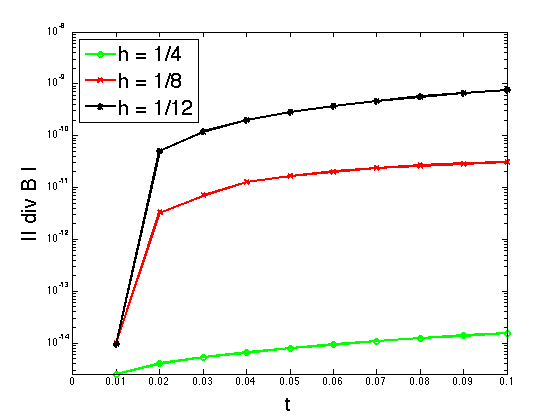} }
\caption{ Numerical results of $3$D convergence test. }  \label{fig:convergence_test_rlt3D}
\end{figure}

Based on the results shown in Figure \ref{fig:convergence_test_rlt3D}, we can see that both spatial and temporal errors converge in the optimal order and the divergence-free condition is preserved. 

\subsection{Benchmark problem: lid-driven cavity}
Next we consider the lid-driven cavity problem, which is a well-known benchmark problem for CFD simulation \cite{Salah.N;Soulaimani.A;Habashi.W.2001a}.  A constant magnetic field {\small $\boldsymbol{B}_{0}$} is applied and coupled with fluids.  When the external imposed magnetic field is zero, the problem becomes the classical hydrodynamic lid-driven cavity problem. Here, we assume that the cavity is a unit square in $2$D and a unit cubic in $3$D.  

\subsection{2D lid-driven cavity}
For the $2$D case, the fluid movement in the cavity is induced by the imposed boundary condition {\small $\boldsymbol{u} = \boldsymbol{u}_{0}$} on the boundary {\small $y = 1$} (figure \ref{fig:geo_liddriven2D}). The constant background magnetic field is {\small $\boldsymbol{B}_{0} = \left( 0, 1 \right)^{T}$}. The velocity field perturbs to the magnetic fields, and the Lorentz force acts on the fluid and changes the motion of the fluid flow. 

\vskip-20pt
\begin{center}
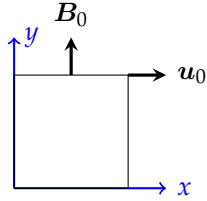

\begin{tikzpicture}
	[cube/.style = {very thick, black}, 
		axis/.style = {->, blue, thick},
		vector/.style = {-stealth, black, very thick}]
% draw the axes
\draw[axis] (0, 0) -- (2, 0) node[anchor=west]{{\small $x$}};
\draw[axis] (0, 0) -- (0, 2) node[anchor=west]{{\small $y$}};

% draw the cube
\begin{scope}[shift = {(0, 0)}]
	% draw the top & bottom
	\draw (0,0) rectangle (1.5,1.5);
\end{scope}

% draw B0
\coordinate (A) at (0.75, 1.5);
\coordinate (B) at (0.75, 2);
\draw[vector] (A) -- (B) node[above] {{\small $\boldsymbol{B}_{0}$}}; 

% draw u0
\coordinate (C) at (1.5, 1.5);
\coordinate (D) at (2, 1.5);
\draw[vector] (C) -- (D) node[right] {{\small $\boldsymbol{u}_{0}$}};
\end{tikzpicture}
\captionof{figure}{Geometry of lid driven cavity}\label{fig:geo_liddriven2D}
\end{center}

Again, we use symmetric Picard iteration in this benchmark test. Parameters for this test are:
\begin{itemize}
\item Time step size {\small $k = 0.01$}, time interval is {\small $[0, 10.0]$}. 
\item Uniform grid is used, with {\small $h = 1/100$}. Number of elements is $2 \times 10^{4}$.
\item The tolerance of preconditioned Krylov methods is {\small $10^{-8}$} and the tolerance of nonlinear iteration is {\small $10^{-6}$}.
\item {\small $\mu_{r} = 1$}, {\small $\sigma_r = 1$}, and {\small $s= 1$}.
\end{itemize}

We perform numerical tests of different Reynolds number ({\small $Re$}) and magnetic Reynolds number ({\small $Rm$}). In Figure \ref{fig:lidcavity_vstream}, we plot stream lines of velocity fields in different cases.  We can see the vortex appears as the Reynolds number increases. 

\vskip-12pt
\begin{figure}[H]
    \centering
        \subfigure[{\small $Re=1$}, {\small $Rm = 1$}]{ \includegraphics[width=0.19\textwidth]{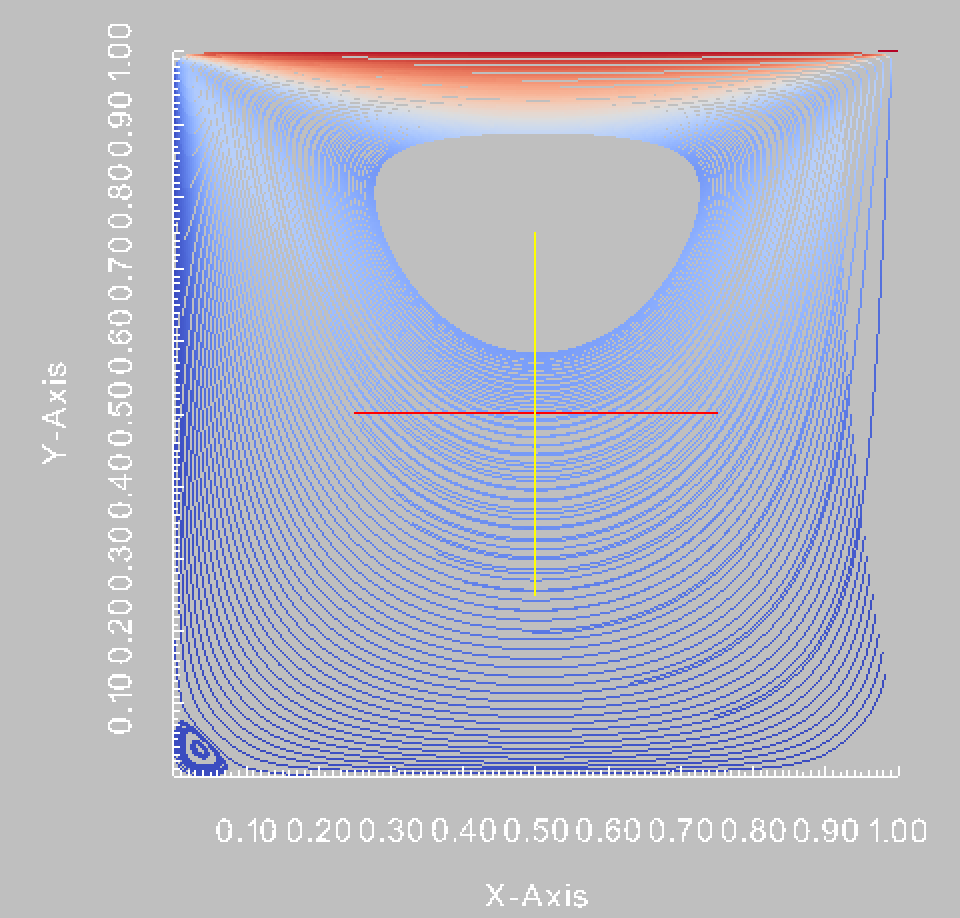} }
        \quad
        \subfigure[{\small $Re = 1$}, {\small $Rm = 400$}]{ \includegraphics[width=0.182\textwidth]{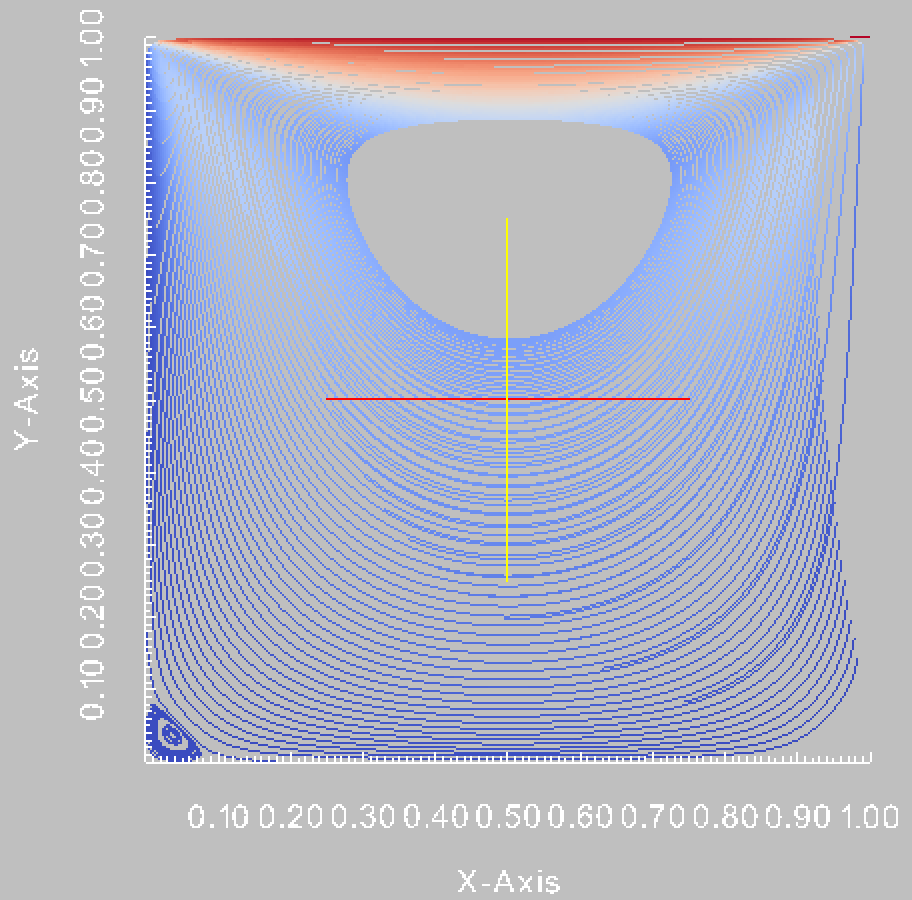} }
        \quad
        \subfigure[{\small $Re = 400$}, {\small $Rm = 1$}]{ \includegraphics[width=0.185\textwidth]{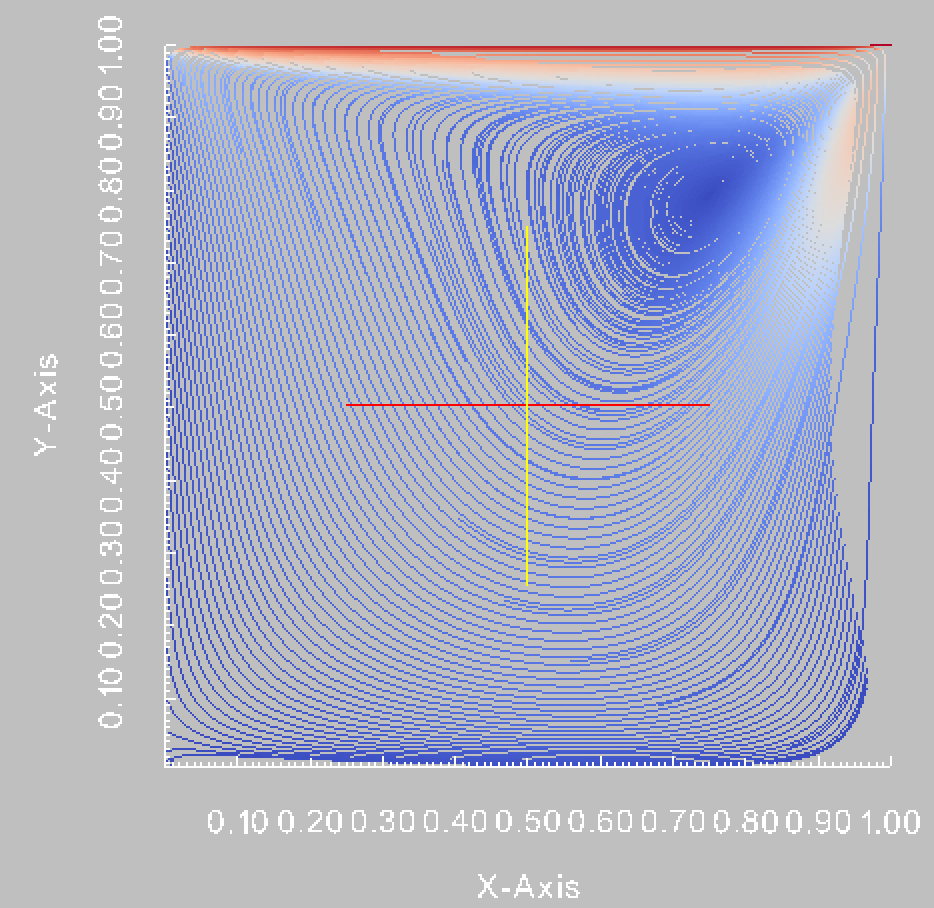} }
        \quad
        \subfigure[{\small $Re = 400$}, {\small $Rm = 400$}]{ \includegraphics[width=0.18\textwidth]{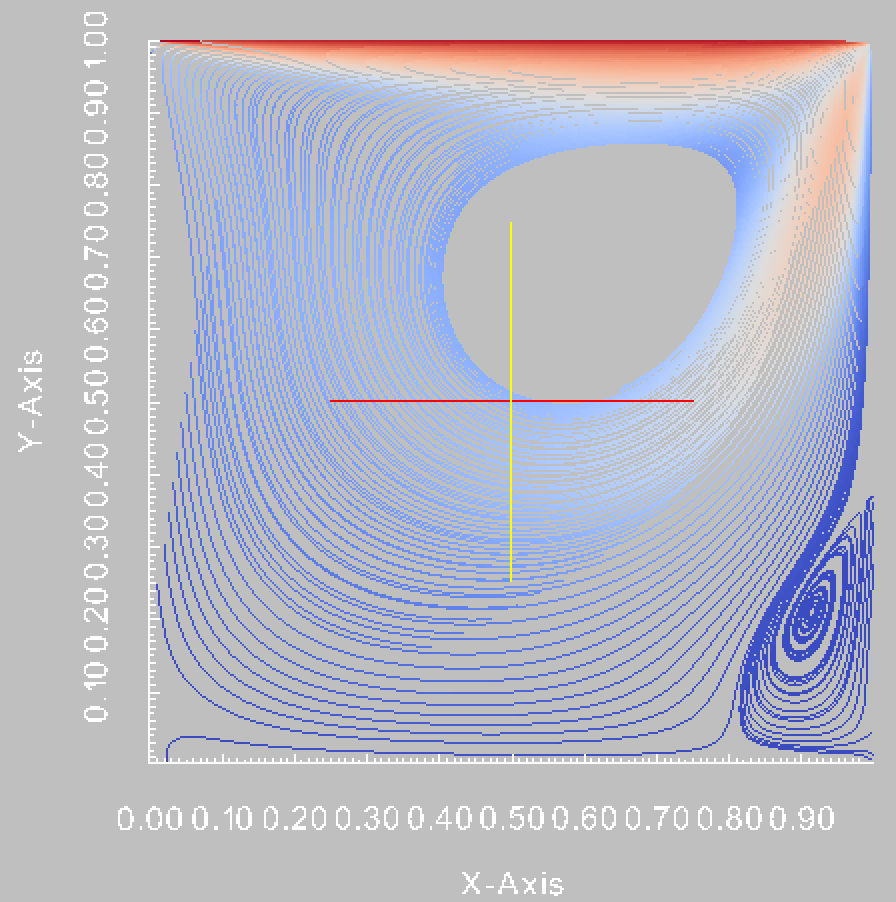} }
        \caption{Stream lines of velocity field. } \label{fig:lidcavity_vstream}
\end{figure}

In Figure \ref{fig:lidcavity_b}, we plot the distribution of the total magnetic fields for different {\small $Re$} and {\small $Rm$}. When {\small $Rm$} is relatively small, the total magnetic field is almost the same as the constant background field, while when {\small $Rm$} is relatively large, the distribution of total magnetic field is attached to the motion of the fluid. 

\vskip-12pt
\begin{figure}[H]
    \centering
        \subfigure[{\small $Re = 1$}, {\small $Rm = 1$}]{ \includegraphics[width=0.18\textwidth]{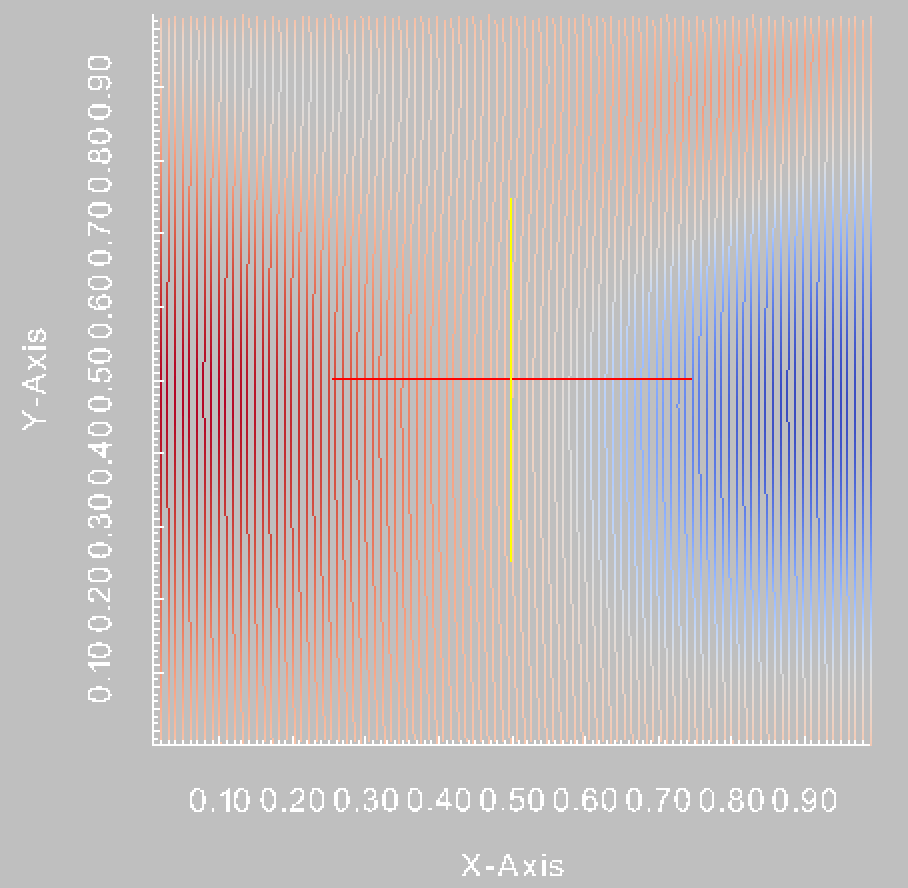} }
        \quad
        \subfigure[{\small $Re = 1$}, {\small $Rm = 400$}]{ \includegraphics[width=0.175\textwidth]{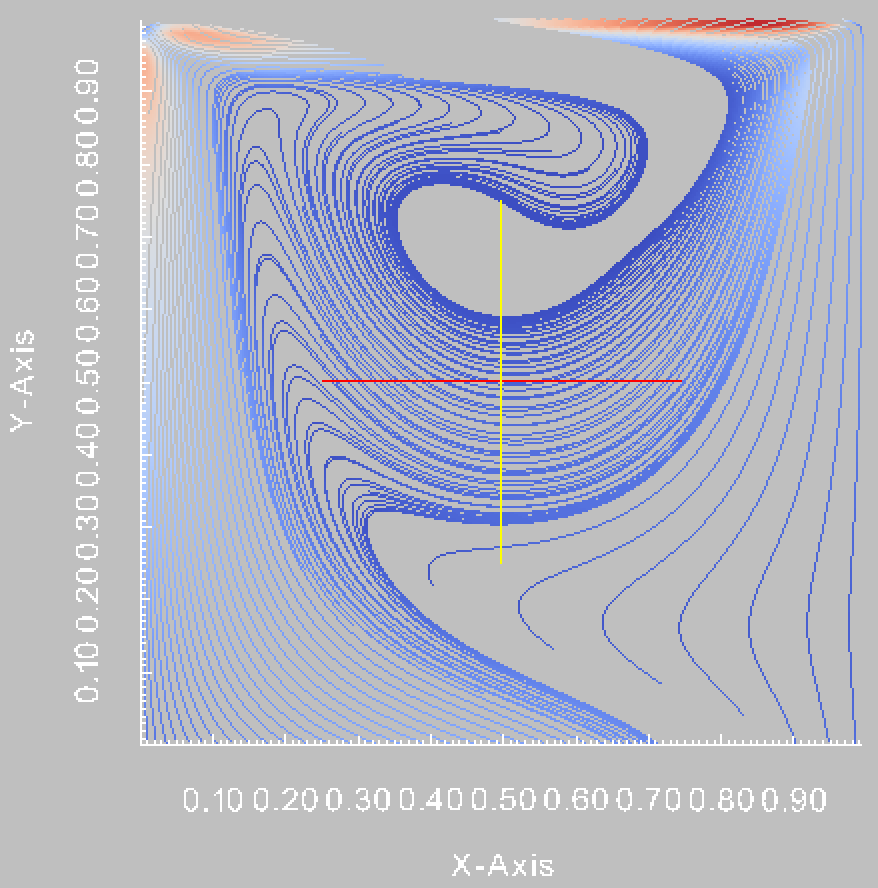} }
        \quad
        \subfigure[{\small $Re = 400$}, {\small $Rm = 1$}]{ \includegraphics[width=0.175\textwidth]{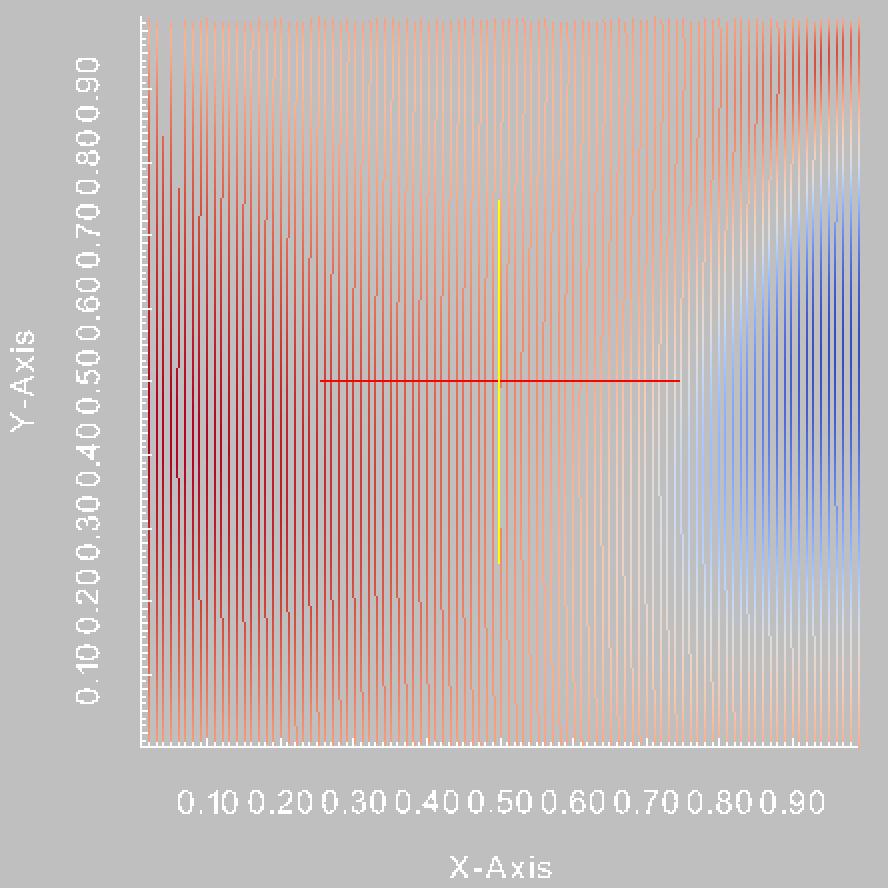} }
        \quad
        \subfigure[{\small $Re = 400$}, {\small $Rm = 400$}]{ \includegraphics[width=0.178\textwidth]{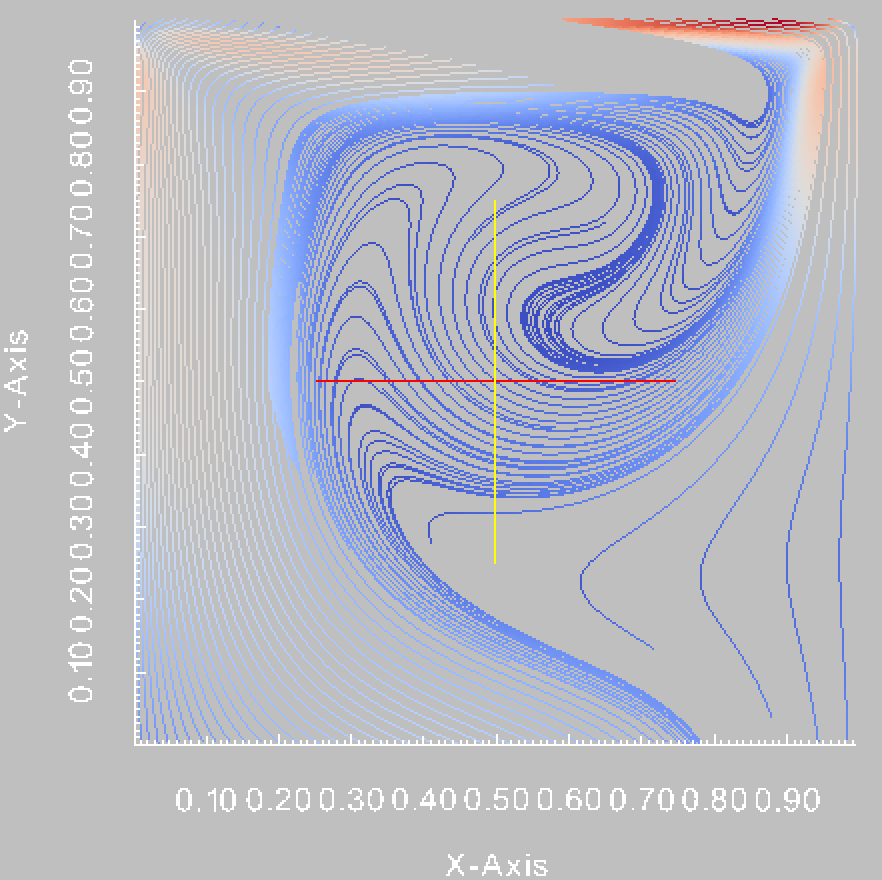} }
        \caption{Distribution of total magnetic field.} \label{fig:lidcavity_b}
\end{figure}

Next we demonstrate the performance of the proposed preconditioners {\small $\mathcal{D}$} \eqref{eqn:B-no-stable} and {\small $\mathcal{M}_{\mathcal{L}}$} \eqref{eqn:L-no-stable} in Table \ref{tab:2Dliddriven_picard}.  Note that, these two preconditioners requires to solve each diagonal block exactly.  The numbers in the table are the number of iterations for preconditioned Krylov methods.

%\vskip-10pt
\begin{scriptsize}
\begin{table}[H]
\centering 
	\begin{tabular}{l | c c c}
	\hline\hline
	& $\mathcal{D}$ (MINRES) & $\mathcal{D}$ (GMRES) & $\mathcal{M}_{\mathcal{L}}$ (GMRES)  \\ [1ex] 
	\hline 
	$Re=1$, $Rm=1$  & 32 & 27 & 7 \\ [1ex]
	% Entering 2nd row
	$Re=1$, $Rm=400$  & 33 & 26 & 6 \\ [1ex]
	% Entering 3rd row
	$Re=400$, $Rm=1$  & 25 & 21 & 5 \\ [1ex]
	% Entering 4th row
	$Re=400$, $Rm=400$  & 27 & 23 & 5 \\ [1ex]
	% [1ex] adds vertical space
	\hline % inserts single-line
	\end{tabular}
	\caption{Performance of solver (2D): $s = 1$} \label{tab:2Dliddriven_picard}
\end{table}
\end{scriptsize}

Besides the performance of solver, we also concern about the value of {\small $k_{0}$}, defined by \eqref{eq:def_k0}. So we compute the values of {\small $k_{0}$} for each time step and nonlinear step, and list the minimum value of {\small $k_{0}$} of all the iterations in table \ref{tab:2Dk0}. Based on the values of {\small $k_{0}$}, we can conclude that the refinement of the mesh does not influence the restriction on time step size significantly.

%\vskip-15pt
\begin{scriptsize}
\begin{table}[H]
\centering 
	\begin{tabular}{c c c c c}
	\hline\hline
	& $Re=1$, $Rm=1$ & $Re=1$, $Rm=400$ & $Re=400$, $Rm=1$ & $Re=400$, $Rm=400$ \\ [1ex] 
	\hline 
	h = 50 & 0.12 & 0.00084 & 0.12 & 0.0017 \\ [1ex]
	h = 100 & 0.12 & 0.00065 & 0.12 & 0.0012 \\ [1ex]
	h = 200 & 0.12 & 0.00057 & 0.12 & 0.00098 \\ [1ex]
	% [1ex] adds vertical space
	\hline % inserts single-line
	\end{tabular}
	\caption{Restriction on time step size: values of $k_{0}$} \label{tab:2Dk0}
\end{table}
\end{scriptsize}
\vskip-15pt
In order to further investigate the robustness of the proposed preconditioners with respect to the time step size {\small $k$} and the mesh size {\small $h$}, we list the number of iterations of the Krylov methods with different block preconditioners in Table \ref{tab:2Dexact_diagMINRES}, \ref{tab:2Dexact_diagGMRES}, \ref{tab:2Dexact_lower}, \ref{tab:2Dinexact_diag} and \ref{tab:2Dinexact_lower}. The tolerance of the preconditioned Krylov methods is $10^{-6}$. When the diagonal blocks are solved approximately by the PCG method, the preconditioners are actually changing in our implementation. Therefore, we use FGMRES method in this case to improve the robustness. Based on the results shown in the tables, we can conclude that our block preconditioners are effective and robust. 

\vskip-15pt
\begin{tiny}
\begin{table}[H]
\centering
\subtable[$Re = 1$, $Rm =1$]{
		\begin{tabular}{l | c c c}
		\hline\hline
		 \backslashbox{$\Delta t$}{$h$} & $\nicefrac{1}{32}$ & $\nicefrac{1}{64}$ & $\nicefrac{1}{128}$ 
		  \\ [0.5ex] 
		\hline 
		$0.02$ & 23 & 23 & 22 \\
		% Entering 2nd row
		$0.01$ & 24 & 24 & 22 \\
		% Entering 3rd row
		$0.005$ & 24 & 23 & 22 \\
		% Entering 4th row
		$0.0025$ & 23 & 23 & 22 \\
		% [1ex] adds vertical space
		\hline % inserts single-line
		\end{tabular}
	}
	\subtable[$Re = 1$, $Rm =400$]{
		\begin{tabular}{l | c c c}
		\hline\hline
		 \backslashbox{$\Delta t$}{$h$} & $\nicefrac{1}{32}$ & $\nicefrac{1}{64}$ & $\nicefrac{1}{128}$
		  \\ [0.5ex] 
		\hline 
		$0.02$ & 19 & 23 & 23 \\
		% Entering 2nd row
		$0.01$ & 17 & 19 & 21 \\
		% Entering 3rd row
		$0.005$ & 16 & 17 & 19 \\
		% Entering 4th row
		$0.0025$ & 14 & 16 & 17 \\
		% [1ex] adds vertical space
		\hline % inserts single-line
		\end{tabular}
	}
	\subtable[$Re = 400$, $Rm =1$]{
		\begin{tabular}{l | c c c}
		\hline\hline
		 \backslashbox{$\Delta t$}{$h$} & $\nicefrac{1}{32}$ & $\nicefrac{1}{64}$ & $\nicefrac{1}{128}$
		  \\ [0.5ex] 
		\hline 
		$0.02$ & 15 & 15 & 15 \\
		% Entering 2nd row
		$0.01$ & 13 & 15 & 15 \\
		% Entering 3rd row
		$0.005$ & 13 & 13 & 15 \\
		% Entering 4th row
		$0.0025$ & 13 & 13 & 13 \\
		% [1ex] adds vertical space
		\hline % inserts single-line
		\end{tabular}
	}
	\subtable[$Re = 400$, $Rm =400$]{
		\begin{tabular}{l | c c c}
		\hline\hline
		 \backslashbox{$\Delta t$}{$h$} & $\nicefrac{1}{32}$ & $\nicefrac{1}{64}$ & $\nicefrac{1}{128}$
		 \\ [0.5ex] 
		\hline 
		$0.02$ & 14 & 16 & 17 \\
		% Entering 2nd row
		$0.01$ & 10 & 14 & 16 \\
		% Entering 3rd row
		$0.005$ & 10 & 12 & 15 \\
		% Entering 4th row
		$0.0025$ & 9 & 10 & 13 \\
		% [1ex] adds vertical space
		\hline % inserts single-line
		\end{tabular}
	}
	\caption{Block diagonal preconditioner {\small $\mathcal{D}$} \eqref{eqn:B-no-stable} for MINRES method (diagonal blocks are solved exactly)}
	\label{tab:2Dexact_diagMINRES}
\end{table}
\end{tiny}

\vskip-25pt
\begin{tiny}
\begin{table}[H]
\centering
\subtable[$Re = 1$, $Rm =1$]{
		\begin{tabular}{l | c c c}
		\hline\hline
		 \backslashbox{$\Delta t$}{$h$} & $\nicefrac{1}{32}$ & $\nicefrac{1}{64}$ & $\nicefrac{1}{128}$ 
		  \\ [0.5ex] 
		\hline 
		$0.02$ & 20 & 18 & 17 \\
		% Entering 2nd row
		$0.01$ & 19 & 18 & 18 \\
		% Entering 3rd row
		$0.005$ & 21 & 20 & 18 \\
		% Entering 4th row
		$0.0025$ & 20 & 20 & 18 \\
		% [1ex] adds vertical space
		\hline % inserts single-line
		\end{tabular}
	}
	\subtable[$Re = 1$, $Rm =400$]{
		\begin{tabular}{l | c c c}
		\hline\hline
		 \backslashbox{$\Delta t$}{$h$} & $\nicefrac{1}{32}$ & $\nicefrac{1}{64}$ & $\nicefrac{1}{128}$
		  \\ [0.5ex] 
		\hline 
		$0.02$ & 16 & 17 & 15 \\
		% Entering 2nd row
		$0.01$ & 14 & 15 & 15 \\
		% Entering 3rd row
		$0.005$ & 14 & 14 & 15 \\
		% Entering 4th row
		$0.0025$ & 12 & 14 & 13 \\
		% [1ex] adds vertical space
		\hline % inserts single-line
		\end{tabular}
	}
	\subtable[$Re = 400$, $Rm =1$]{
		\begin{tabular}{l | c c c}
		\hline\hline
		 \backslashbox{$\Delta t$}{$h$} & $\nicefrac{1}{32}$ & $\nicefrac{1}{64}$ & $\nicefrac{1}{128}$
		  \\ [0.5ex] 
		\hline 
		$0.02$ & 13 & 11 & 11 \\
		% Entering 2nd row
		$0.01$ & 11 & 13 & 12 \\
		% Entering 3rd row
		$0.005$ & 9 & 12 & 12 \\
		% Entering 4th row
		$0.0025$ & 10 & 9 & 10 \\
		% [1ex] adds vertical space
		\hline % inserts single-line
		\end{tabular}
	}
	\subtable[$Re = 400$, $Rm =400$]{
		\begin{tabular}{l | c c c}
		\hline\hline
		 \backslashbox{$\Delta t$}{$h$} & $\nicefrac{1}{32}$ & $\nicefrac{1}{64}$ & $\nicefrac{1}{128}$
		 \\ [0.5ex] 
		\hline 
		$0.02$ & 8 & 11 & 9 \\
		% Entering 2nd row
		$0.01$ & 7 & 7 & 7 \\
		% Entering 3rd row
		$0.005$ & 6 & 7 & 7 \\
		% Entering 4th row
		$0.0025$ & 8 & 7 & 7 \\
		% [1ex] adds vertical space
		\hline % inserts single-line
		\end{tabular}
	}
	\caption{Block diagonal preconditioner {\small $\mathcal{D}$} \eqref{eqn:B-no-stable} for FGMRES method (diagonal blocks are solved exactly)}
	\label{tab:2Dexact_diagGMRES}
\end{table}
\end{tiny}

\vskip-25pt
\begin{tiny}
\begin{table}[H]
\centering
\subtable[$Re = 1$, $Rm = 1$]{
		\begin{tabular}{l | c c c}
		\hline\hline
		 \backslashbox{$\Delta t$}{$h$} & $\nicefrac{1}{32}$ & $\nicefrac{1}{64}$ & $\nicefrac{1}{128}$ 
		  \\ [0.5ex] 
		\hline 
		$0.02$ & 6 & 5 & 5 \\
		% Entering 2nd row
		$0.01$ & 5 & 5 & 5 \\
		% Entering 3rd row
		$0.005$ & 5 & 5 & 5 \\
		% Entering 4th row
		$0.0025$ & 5 & 5 & 5\\
		% [1ex] adds vertical space
		\hline % inserts single-line
		\end{tabular}
	}
	\subtable[$Re = 1$, $Rm = 400$]{
		\begin{tabular}{l | c c c}
		\hline\hline
		 \backslashbox{$\Delta t$}{$h$} & $\nicefrac{1}{32}$ & $\nicefrac{1}{64}$ & $\nicefrac{1}{128}$
		  \\ [0.5ex] 
		\hline 
		$0.02$ & 5 & 5 & 5 \\
		% Entering 2nd row
		$0.01$ & 5 & 5 & 5 \\
		% Entering 3rd row
		$0.005$ & 5 & 5 & 5 \\
		% Entering 4th row
		$0.0025$ & 5 & 5 & 5 \\
		% [1ex] adds vertical space
		\hline % inserts single-line
		\end{tabular}
	}
	\subtable[$Re = 400$, $Rm = 1$]{
		\begin{tabular}{l | c c c}
		\hline\hline
		 \backslashbox{$\Delta t$}{$h$} & $\nicefrac{1}{32}$ & $\nicefrac{1}{64}$ & $\nicefrac{1}{128}$
		  \\ [0.5ex] 
		\hline 
		$0.02$ & 4 & 4 & 4 \\
		% Entering 2nd row
		$0.01$ & 3 & 4 & 4 \\
		% Entering 3rd row
		$0.005$ & 3 & 3 & 3 \\
		% Entering 4th row
		$0.0025$ & 4 & 3 & 3 \\
		% [1ex] adds vertical space
		\hline % inserts single-line
		\end{tabular}
	}
	\subtable[$Re = 400$, $Rm = 400$]{
		\begin{tabular}{l | c c c}
		\hline\hline
		 \backslashbox{$\Delta t$}{$h$} & $\nicefrac{1}{32}$ & $\nicefrac{1}{64}$ & $\nicefrac{1}{128}$
		  \\ [0.5ex] 
		\hline 
		$0.02$ & 3 & 3 & 3\\
		% Entering 2nd row
		$0.01$ & 3 & 3 & 3 \\
		% Entering 3rd row
		$0.005$ & 3 & 3 & 3 \\
		% Entering 4th row
		$0.0025$ & 4 & 3 & 3 \\
		% [1ex] adds vertical space
		\hline % inserts single-line
		\end{tabular}
	}
	\caption{Block lower triangular preconditioner {\small $\mathcal{M}_{\mathcal{L}}$} \eqref{eqn:L-no-stable} for FGMRES method (diagonal blocks are solved exactly)}
	\label{tab:2Dexact_lower}
\end{table}
\end{tiny}

\vskip-25pt
\begin{tiny}
\begin{table}[H]
\centering
\subtable[$Re = 1$, $Rm =1$]{
		\begin{tabular}{l | c c c}
		\hline\hline
		 \backslashbox{$\Delta t$}{$h$} & $\nicefrac{1}{32}$ & $\nicefrac{1}{64}$ & $\nicefrac{1}{128}$ 
		  \\ [0.5ex] 
		\hline 
		$0.02$ & 21 & 25 & 43 \\
		% Entering 2nd row
		$0.01$ & 22 & 29 & 51 \\
		% Entering 3rd row
		$0.005$ & 24 & 34 & 60 \\
		% Entering 4th row
		$0.0025$ & 24 & 38 & 70 \\
		% [1ex] adds vertical space
		\hline % inserts single-line
		\end{tabular}
	}
	\subtable[$Re = 1$, $Rm =400$]{
		\begin{tabular}{l | c c c}
		\hline\hline
		 \backslashbox{$\Delta t$}{$h$} & $\nicefrac{1}{32}$ & $\nicefrac{1}{64}$ & $\nicefrac{1}{128}$ 
		  \\ [0.5ex] 
		\hline 
		$0.02$ & 16 & 22 & 43 \\
		% Entering 2nd row
		$0.01$ & 16 & 24 & 49 \\
		% Entering 3rd row
		$0.005$ & 16 & 25 & 57 \\
		% Entering 4th row
		$0.0025$ & 15 & 27 & 60 \\
		% [1ex] adds vertical space
		\hline % inserts single-line
		\end{tabular}
	}
	\subtable[$Re = 400$, $Rm =1$]{
		\begin{tabular}{l | c c c}
		\hline\hline
		 \backslashbox{$\Delta t$}{$h$} & $\nicefrac{1}{32}$ & $\nicefrac{1}{64}$ & $\nicefrac{1}{128}$ 
		  \\ [0.5ex] 
		\hline 
		$0.02$ & 16 & 24 & 46 \\
		% Entering 2nd row
		$0.01$ & 16 & 26 & 47 \\
		% Entering 3rd row
		$0.005$ & 18 & 24 & 49 \\
		% Entering 4th row
		$0.0025$ & 20 & 25 & 49 \\
		% [1ex] adds vertical space
		\hline % inserts single-line
		\end{tabular}
	}
	\subtable[$Re = 400$, $Rm =400$]{
		\begin{tabular}{l | c c c}
		\hline\hline
		 \backslashbox{$\Delta t$}{$h$} & $\nicefrac{1}{32}$ & $\nicefrac{1}{64}$ & $\nicefrac{1}{128}$ 
		 \\ [0.5ex] 
		\hline 
		$0.02$ & 13 & 23 & 46 \\
		% Entering 2nd row
		$0.01$ & 12 & 21 & 47 \\
		% Entering 3rd row
		$0.005$ & 12 & 20 & 47 \\
		% Entering 4th row
		$0.0025$ & 14 & 21 & 47 \\
		% [1ex] adds vertical space
		\hline % inserts single-line
		\end{tabular}
	}
	\caption{Block diagonal preconditioner {\small $\mathcal{M}$} \eqref{eqn:M-no-stable} for FGMRES method (diagonal blocks are solved approximately)}
	\label{tab:2Dinexact_diag}
\end{table}
\end{tiny}

\vskip-25pt
\begin{tiny}
\begin{table}[H]
\centering
\subtable[$Re = 1$, $Rm =1$]{
		\begin{tabular}{l | c c c}
		\hline\hline
		 \backslashbox{$\Delta t$}{$h$} & $\nicefrac{1}{32}$ & $\nicefrac{1}{64}$ & $\nicefrac{1}{128}$
		  \\ [0.5ex] 
		\hline 
		$0.02$ & 6 & 9 & 15 \\
		% Entering 2nd row
		$0.01$ & 6 & 9 & 17 \\
		% Entering 3rd row
		$0.005$ & 6 & 10 & 20 \\
		% Entering 4th row
		$0.0025$ & 6 & 11 & 22 \\
		% [1ex] adds vertical space
		\hline % inserts single-line
		\end{tabular}
	}
	\subtable[$Re = 1$, $Rm =400$]{
		\begin{tabular}{l | c c c}
		\hline\hline
		 \backslashbox{$\Delta t$}{$h$} & $\nicefrac{1}{32}$ & $\nicefrac{1}{64}$ & $\nicefrac{1}{128}$		  		\\ [0.5ex] 
		\hline 
		$0.02$ & 6 & 9 & 15 \\
		% Entering 2nd row
		$0.01$ & 6 & 9 & 17 \\
		% Entering 3rd row
		$0.005$ & 6 & 10 & 20 \\
		% Entering 4th row
		$0.0025$ & 6 & 11 & 22 \\
		% [1ex] adds vertical space
		\hline % inserts single-line
		\end{tabular}
	}
	\subtable[$Re = 400$, $Rm =1$]{
		\begin{tabular}{l | c c c}
		\hline\hline
		 \backslashbox{$\Delta t$}{$h$} & $\nicefrac{1}{32}$ & $\nicefrac{1}{64}$ & $\nicefrac{1}{128}$ 
		  \\ [0.5ex] 
		\hline 
		$0.02$ & 6 & 9 & 16 \\
		% Entering 2nd row
		$0.01$ & 6 & 10 & 17 \\
		% Entering 3rd row
		$0.005$ & 6 & 10 & 17 \\
		% Entering 4th row
		$0.0025$ & 7 & 11 & 18 \\
		% [1ex] adds vertical space
		\hline % inserts single-line
		\end{tabular}
	}
	\subtable[$Re = 400$, $Rm =400$]{
		\begin{tabular}{l | c c c}
		\hline\hline
		 \backslashbox{$\Delta t$}{$h$} & $\nicefrac{1}{32}$ & $\nicefrac{1}{64}$ & $\nicefrac{1}{128}$
		 \\ [0.5ex] 
		\hline 
		$0.02$ & 6 & 9 & 16 \\
		% Entering 2nd row
		$0.01$ & 6 & 10 & 17 \\
		% Entering 3rd row
		$0.005$ & 6 & 10 & 17 \\
		% Entering 4th row
		$0.0025$ & 7 & 10 & 18 \\
		% [1ex] adds vertical space
		\hline % inserts single-line
		\end{tabular}
	}
	\caption{Block lower triangular preconditioner {\small $\widehat{\mathcal{M}}_{\mathcal{L}}$} \eqref{eqn:barL-no-stable} for FGMRES method (diagonal blocks are solved approximately)} 
	\label{tab:2Dinexact_lower}
\end{table}
\end{tiny}

\vskip-20pt
\subsection{3D lid-driven cavity}
In $3$D case, the fluid movement in the cavity is induced by the imposed boundary condition {\small $\boldsymbol{u} = \boldsymbol{u}_{0} = \left( 1, 0, 0 \right)^{T}$} on the plane {\small $y = 1$}. The constant background magnetic field is {\small $\boldsymbol{B}_{0} = \left( 0, 1, 0 \right)^{T}$}. 

%see Figure \ref{fig:geo_liddriven3D}. 

% draw picture of lid-driven cavity here.
\begin{center}
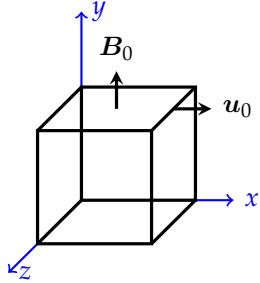

\begin{tikzpicture}
	[cube/.style = {very thick, black}, 
		axis/.style = {->, blue, thick},
		vector/.style = {-stealth, black, very thick}]
% draw the axes
\draw[axis] (0, 0, 0) -- (2, 0, 0) node[anchor=west]{{\small $x$}};
\draw[axis] (0, 0, 0) -- (0, 2.5, 0) node[anchor=west]{{\small $y$}};
\draw[axis] (0, 0, 0) -- (0, 0, 2.5) node[anchor=west]{{\small $z$}};

% draw the cube
\begin{scope}[shift = {(0, 0, 0)}]
	% draw the top & bottom
	\draw[cube] (0, 0, 0) -- (0, 1.5, 0) -- (1.5, 1.5, 0) -- (1.5, 0, 0) -- cycle;
	\draw[cube] (0, 0, 1.5) -- (0, 1.5, 1.5) -- (1.5, 1.5, 1.5) -- (1.5, 0, 1.5) -- cycle;
	
	% draw the edges of the cube
	\draw[cube] (0, 0, 0) -- (0, 0, 1.5);
	\draw[cube] (0,1.5, 0) -- (0, 1.5, 1.5);
	\draw[cube] (1.5, 0, 0) -- (1.5, 0, 1.5);
	\draw[cube] (1.5, 1.5, 0) -- (1.5, 1.5, 1.5);
\end{scope}

% draw B0
\coordinate (A) at (0.75, 1.5, 0.75);
\coordinate (B) at (0.75, 2, 0.75);
\draw[vector] (A) -- (B) node[above] {{\small $\boldsymbol{B}_{0}$}}; 

% draw u0
\coordinate (C) at (1.5, 1.5, 0.75);
\coordinate (D) at (2, 1.5, 0.75);
\draw[vector] (C) -- (D) node[right] {{\small $\boldsymbol{u}_{0}$}};
\end{tikzpicture}
\captionof{figure}{Geometry of lid driven cavity}\label{fig:geo_liddriven3D}
\end{center}

We also use symmetric Picard iteration to solve this $3$D benchmark problem. Since the motion of the fluids is uniform along the $z$-axis, the sectional view along $z$-axis of the motion of the fluids and the distribution of the total magnetic field are similar to those of the $2$D benchmark. Therefore, we only present the performance of our preconditioners in the $3$D benchmark. Parameters for the $3$D case are:
\begin{itemize}
\item Time step size {\small $\Delta t = 0.01$}. Time interval {\small $[0, 1.0]$}.
\item Uniform grid with mesh size {\small $h = \nicefrac{1}{12}$}.  Number of elements is $10368$.
\item The tolerance of the preconditioned Krylov methods is {\small $10^{-8}$} and the tolerance of nonlinear iteration is {\small $10^{-6}$}.
\item {\small $\mu_r = 1$}, {\small $\sigma_r = 1$}, {\small $s=1$}.
\end{itemize}
For $3$D, we use coarser mesh than the $2$D tests (but compatible problem size) due to the limited memory capacity of our workstation. We list the number of iterations of the Krylov solvers with preconditioners {\small $\mathcal{D}$} \eqref{eqn:B-no-stable} and {\small $\mathcal{M}_{\mathcal{L}}$} \eqref{eqn:L-no-stable} in Table \ref{tab:3Dliddriven_picard}.  Note that, we need to invert each diagonal block exactly for these preconditioners.

\vskip-8pt
\begin{scriptsize}
\begin{table}[H]
\centering 
	\begin{tabular}{l | c c}
	\hline\hline
	 & $\mathcal{D}$ (FGMRES) & $\mathcal{M}_{\mathcal{L}}$ (GMRES)  \\ [1ex] 
	\hline 
	$Re=1$, $Rm=1$ & 160 & 36 \\ [1ex]
	% Entering 2nd row
	$Re=1$, $Rm=400$ & 86 & 35 \\ [1ex]
	% Entering 3rd row
	$Re=400$, $Rm=1$ & 54 & 16 \\ [1ex]
	% Entering 4th row
	$Re=400$, $Rm=400$ & 34 & 16 \\ [1ex]
	% [1ex] adds vertical space
	\hline % inserts single-line
	\end{tabular}
	\caption{Performance of solver (3D): $s = 1$} \label{tab:3Dliddriven_picard}
\end{table}
\end{scriptsize}
\vskip-10pt
In order to further demonstrate the robustness of the preconditioner with respect to the time step size {\small $k$} and the mesh size {\small $h$}, we list the number of iterations of the Krylov methods with different block preconditioners in Table \ref{tab:3Dexact_diagMINRES}, \ref{tab:3Dexact_diagGMRES}, \ref{tab:3Dexact_lower}, \ref{tab:3Dinexact_diag}, and \ref{tab:3Dinexact_lower}. The tolerance of the preconditioned Krylov methods is {\small $10^{-6}$}.  {\small $\ast$} in the tables means that our workstation is out of memory during the test.  This usually happens when we call direct solvers.  In $3$D, diagonal blocks get bigger when the mesh size decreases, so we cannot call direct solvers due to the limitation of our workstation. When solving each diagonal block approximately, we set the tolerance of PCG {\small $10^{-3}$}. Based on the test results, we can conclude that our block preconditioners are effective and robust. 
\vskip-13pt
\begin{tiny}
\begin{table}[H]
\centering 
	\subtable[$Re = 1$, $Rm = 1$]{
		\begin{tabular}{l | c c c}
		\hline\hline
		\backslashbox{$\Delta t$}{$h$} & $\nicefrac{1}{4}$ & $\nicefrac{1}{8}$ & $\nicefrac{1}{16}$ \\ [0.5ex] 
		\hline 
		$0.02$ & 65 & 79 & $\ast$ \\
		% Entering 2nd row
		$0.01$ & 59 & 72 & $\ast$ \\
		% Entering 3rd row
		$0.005$ & 54 & 67 & $\ast$ \\
		% Entering 4th row
		$0.0025$ & 42 & 61 & $\ast$ \\
		% [1ex] adds vertical space
		\hline % inserts single-line
		\end{tabular}
	}
	\subtable[$Re = 1$, $Rm = 400$]{
		\begin{tabular}{l | c c c}
		\hline\hline
		\backslashbox{$\Delta t$}{$h$} & $\nicefrac{1}{4}$ & $\nicefrac{1}{8}$ & $\nicefrac{1}{16}$ \\ [0.5ex] 
		\hline 
		$0.02$ & 44 & 56 & $\ast$ \\
		% Entering 2nd row
		$0.01$ & 38 & 47 & $\ast$ \\
		% Entering 3rd row
		$0.005$ & 34 & 44 & $\ast$ \\
		% Entering 4th row
		$0.0025$ & 28 & 40 & $\ast$ \\
		% [1ex] adds vertical space
		\hline % inserts single-line
		\end{tabular}
	}
	\subtable[$Re = 400$, $Rm = 1$]{
		\begin{tabular}{l | c c c}
		\hline\hline
		\backslashbox{$\Delta t$}{$h$} & $\nicefrac{1}{4}$ & $\nicefrac{1}{8}$ & $\nicefrac{1}{16}$ \\ [0.5ex] 
		\hline 
		$0.02$ & 45 & 29 & $\ast$ \\
		% Entering 2nd row
		$0.01$ & 48 & 32 & $\ast$ \\
		% Entering 3rd row
		$0.005$ & 49 & 36 & $\ast$ \\
		% Entering 4th row
		$0.0025$ & 49 & 40 & $\ast$ \\
		% [1ex] adds vertical space
		\hline % inserts single-line
		\end{tabular}
	}
	\subtable[$Re = 400$, $Rm = 400$]{
		\begin{tabular}{l | c c c}
		\hline\hline
		\backslashbox{$\Delta t$}{$h$} & $\nicefrac{1}{4}$ & $\nicefrac{1}{8}$ & $\nicefrac{1}{16}$ \\ [0.5ex] 
		\hline 
		$0.02$ & 32 & 21 & $\ast$ \\
		% Entering 2nd row
		$0.01$ & 31 & 23 & $\ast$ \\
		% Entering 3rd row
		$0.005$ & 32 & 23 & $\ast$ \\
		% Entering 4th row
		$0.0025$ & 32 & 25 & $\ast$ \\
		% [1ex] adds vertical space
		\hline % inserts single-line
		\end{tabular}
	}
	\caption{Block diagonal preconditioner {\small $\mathcal{D}$} \eqref{eqn:B-no-stable} for MINRES method (diagonal blocks are solved exactly)}
	\label{tab:3Dexact_diagMINRES}
\end{table}
\end{tiny}

\vskip-20pt
\begin{tiny}
\begin{table}[H]
\centering 
	\subtable[$Re = 1$, $Rm = 1$]{
		\begin{tabular}{l | c c c}
		\hline\hline
		\backslashbox{$\Delta t$}{$h$} & $\nicefrac{1}{4}$ & $\nicefrac{1}{8}$ & $\nicefrac{1}{16}$ \\ [0.5ex] 
		\hline 
		$0.02$ & 54 & 62 & $\ast$ \\
		% Entering 2nd row
		$0.01$ & 52 & 49 & $\ast$ \\
		% Entering 3rd row
		$0.005$ & 49 & 47 & $\ast$ \\
		% Entering 4th row
		$0.0025$ & 42 & 47 & $\ast$ \\
		% [1ex] adds vertical space
		\hline % inserts single-line
		\end{tabular}
	}
	\subtable[$Re = 1$, $Rm = 400$]{
		\begin{tabular}{l | c c c}
		\hline\hline
		\backslashbox{$\Delta t$}{$h$} & $\nicefrac{1}{4}$ & $\nicefrac{1}{8}$ & $\nicefrac{1}{16}$ \\ [0.5ex] 
		\hline 
		$0.02$ & 34 & 40 & $\ast$ \\
		% Entering 2nd row
		$0.01$ & 32 & 28 & $\ast$ \\
		% Entering 3rd row
		$0.005$ & 29 & 27 & $\ast$ \\
		% Entering 4th row
		$0.0025$ & 24 & 27 & $\ast$ \\
		% [1ex] adds vertical space
		\hline % inserts single-line
		\end{tabular}
	}
	\subtable[$Re = 400$, $Rm = 1$]{
		\begin{tabular}{l | c c c}
		\hline\hline
		\backslashbox{$\Delta t$}{$h$} & $\nicefrac{1}{4}$ & $\nicefrac{1}{8}$ & $\nicefrac{1}{16}$ \\ [0.5ex] 
		\hline 
		$0.02$ & 30 & 19 & $\ast$ \\
		% Entering 2nd row
		$0.01$ & 42 & 22 & $\ast$ \\
		% Entering 3rd row
		$0.005$ & 46 & 25 & $\ast$ \\
		% Entering 4th row
		$0.0025$ & 49 & 28 & $\ast$ \\
		% [1ex] adds vertical space
		\hline % inserts single-line
		\end{tabular}
	}
	\subtable[$Re = 400$, $Rm = 400$]{
		\begin{tabular}{l | c c c}
		\hline\hline
		\backslashbox{$\Delta t$}{$h$} & $\nicefrac{1}{4}$ & $\nicefrac{1}{8}$ & $\nicefrac{1}{16}$ \\ [0.5ex] 
		\hline 
		$0.02$ & 21 & 13 & $\ast$ \\
		% Entering 2nd row
		$0.01$ & 23 & 14 & $\ast$ \\
		% Entering 3rd row
		$0.005$ & 27 & 16 & $\ast$ \\
		% Entering 4th row
		$0.0025$ & 29 & 18 & $\ast$ \\
		% [1ex] adds vertical space
		\hline % inserts single-line
		\end{tabular}
	}
	\caption{Block diagonal preconditioner {\small $\mathcal{D}$} \eqref{eqn:B-no-stable} for FGMRES method (diagonal blocks are solved exactly)}
	\label{tab:3Dexact_diagGMRES}
\end{table}
\end{tiny}

\vskip-20pt
\begin{tiny}
\begin{table}[H]
\centering 
	\subtable[$Re = 1$, $Rm = 1$]{
		\begin{tabular}{l | c c c}
		\hline\hline
		\backslashbox{$\Delta t$}{$h$} & $\nicefrac{1}{4}$ & $\nicefrac{1}{8}$ & $\nicefrac{1}{16}$ \\ [0.5ex] 
		\hline 
		$0.02$ & 15 & 16 & $\ast$ \\
		% Entering 2nd row
		$0.01$ & 15 & 14 & $\ast$ \\
		% Entering 3rd row
		$0.005$ & 14 & 13 & $\ast$ \\
		% Entering 4th row
		$0.0025$ & 12 & 13 & $\ast$ \\
		% [1ex] adds vertical space
		\hline % inserts single-line
		\end{tabular}
	}
	\subtable[$Re = 1$, $Rm = 400$]{
		\begin{tabular}{l | c c c}
		\hline\hline
		\backslashbox{$\Delta t$}{$h$} & $\nicefrac{1}{4}$ & $\nicefrac{1}{8}$ & $\nicefrac{1}{16}$ \\ [0.5ex] 
		\hline 
		$0.02$ & 15 & 16 & $\ast$ \\
		% Entering 2nd row
		$0.01$ & 15 & 13 & $\ast$ \\
		% Entering 3rd row
		$0.005$ & 14 & 13 & $\ast$ \\
		% Entering 4th row
		$0.0025$ & 12 & 13 & $\ast$ \\
		% [1ex] adds vertical space
		\hline % inserts single-line
		\end{tabular}
	}
	\subtable[$Re = 400$, $Rm = 1$]{
		\begin{tabular}{l | c c c}
		\hline\hline
		\backslashbox{$\Delta t$}{$h$} & $\nicefrac{1}{4}$ & $\nicefrac{1}{8}$ & $\nicefrac{1}{16}$ \\ [0.5ex] 
		\hline 
		$0.02$ & 10 & 6 & $\ast$ \\
		% Entering 2nd row
		$0.01$ & 11 & 7 & $\ast$ \\
		% Entering 3rd row
		$0.005$ & 13 & 8 & $\ast$ \\
		% Entering 4th row
		$0.0025$ & 15 & 9 & $\ast$ \\
		% [1ex] adds vertical space
		\hline % inserts single-line
		\end{tabular}
	}
	\subtable[$Re = 400$, $Rm = 400$]{
		\begin{tabular}{l | c c c}
		\hline\hline
		\backslashbox{$\Delta t$}{$h$} & $\nicefrac{1}{4}$ & $\nicefrac{1}{8}$ & $\nicefrac{1}{16}$ \\ [0.5ex] 
		\hline 
		$0.02$ & 10 & 6 & $\ast$ \\
		% Entering 2nd row
		$0.01$ & 11 & 7 & $\ast$ \\
		% Entering 3rd row
		$0.005$ & 13 & 8 & $\ast$ \\
		% Entering 4th row
		$0.0025$ & 15 & 9 & $\ast$ \\
		% [1ex] adds vertical space
		\hline % inserts single-line
		\end{tabular}
	}
	\caption{Block lower triangular preconditioner {\small $\mathcal{M}_{\mathcal{L}}$} \eqref{eqn:L-no-stable} for FGMRES method (diagonal blocks are solved exactly)}
	\label{tab:3Dexact_lower}
\end{table}
\end{tiny}

\vskip-20pt
\begin{tiny}
\begin{table}[H]
\centering 
	\subtable[$Re = 1$, $Rm = 1$]{
		\begin{tabular}{l | c c c}
		\hline\hline
		\backslashbox{$\Delta t$}{$h$} & $\nicefrac{1}{4}$ & $\nicefrac{1}{8}$ & $\nicefrac{1}{16}$ \\ [0.5ex] 
		\hline 
		$0.02$ & 56 & 64 & 89 \\
		% Entering 2nd row
		$0.01$ & 54 & 52 & 74 \\
		% Entering 3rd row
		$0.005$ & 50 & 51 & 56 \\
		% Entering 4th row
		$0.0025$ & 41 & 50 & 49 \\
		% [1ex] adds vertical space
		\hline % inserts single-line
		\end{tabular}
	}
	\quad
	\subtable[$Re = 1$, $Rm = 400$]{
		\begin{tabular}{l | c c c}
		\hline\hline
		\backslashbox{$\Delta t$}{$h$} & $\nicefrac{1}{4}$ & $\nicefrac{1}{8}$ & $\nicefrac{1}{16}$ \\ [0.5ex] 
		\hline 
		$0.02$ & 36 & 42 & 70 \\
		% Entering 2nd row
		$0.01$ & 34 & 30 & 49 \\
		% Entering 3rd row
		$0.005$ & 29 & 29 & 38 \\
		% Entering 4th row
		$0.0025$ & 25 & 27 & 26 \\
		% [1ex] adds vertical space
		\hline % inserts single-line
		\end{tabular}
	}
	\quad
	\subtable[$Re = 400$, $Rm = 1$]{
		\begin{tabular}{l | c c c}
		\hline\hline
		\backslashbox{$\Delta t$}{$h$} & $\nicefrac{1}{4}$ & $\nicefrac{1}{8}$ & $\nicefrac{1}{16}$ \\ [0.5ex] 
		\hline 
		$0.02$ & 39 & 21 & 14 \\
		% Entering 2nd row
		$0.01$ & 43 & 24 & 17 \\
		% Entering 3rd row
		$0.005$ & 47 & 32 & 21 \\
		% Entering 4th row
		$0.0025$ & 49 & 38 & 25 \\
		% [1ex] adds vertical space
		\hline % inserts single-line
		\end{tabular}
	}
	\quad
	\subtable[$Re = 400$, $Rm = 400$]{
		\begin{tabular}{l | c c c}
		\hline\hline
		\backslashbox{$\Delta t$}{$h$} & $\nicefrac{1}{4}$ & $\nicefrac{1}{8}$ & $\nicefrac{1}{16}$ \\ [0.5ex] 
		\hline 
		$0.02$ & 24 & 14 & 10 \\
		% Entering 2nd row
		$0.01$ & 25 & 16 & 12 \\
		% Entering 3rd row
		$0.005$ & 27 & 19 & 14 \\
		% Entering 4th row
		$0.0025$ & 32 & 21 & 16 \\
		% [1ex] adds vertical space
		\hline % inserts single-line
		\end{tabular}
	}
	\caption{Block diagonal preconditioner {\small $\mathcal{M}$} \eqref{eqn:M-no-stable} for FGMRES method (diagonal blocks are solved approximately)}
	\label{tab:3Dinexact_diag}
\end{table}
\end{tiny}

\vskip-25pt
\begin{tiny}
\begin{table}[H]
\centering
\subtable[$Re = 1$, $Rm =1$]{
		\begin{tabular}{l | c c c}
		\hline\hline
		 \backslashbox{$\Delta t$}{$h$} & $\nicefrac{1}{4}$ & $\nicefrac{1}{8}$ & $\nicefrac{1}{16}$ \\ [0.5ex] 
		\hline 
		$0.02$ & 16 & 18 & 25 \\
		% Entering 2nd row
		$0.01$ & 16 & 15 & 21 \\
		% Entering 3rd row
		$0.005$ & 16 & 15 & 17 \\
		% Entering 4th row
		$0.0025$ & 13 & 15 & 14 \\
		% [1ex] adds vertical space
		\hline % inserts single-line
		\end{tabular}
	}
	\quad
	\subtable[$Re = 1$, $Rm =400$]{
		\begin{tabular}{l | c c c}
		\hline\hline
		 \backslashbox{$\Delta t$}{$h$} & $\nicefrac{1}{4}$ & $\nicefrac{1}{8}$ & $\nicefrac{1}{16}$ \\ [0.5ex] 
		\hline 
		$0.02$ & 16 & 18 & 24 \\
		% Entering 2nd row
		$0.01$ & 16 & 15 & 20 \\
		% Entering 3rd row
		$0.005$ & 16 & 15 & 17 \\
		% Entering 4th row
		$0.0025$ & 13 & 15 & 14 \\
		% [1ex] adds vertical space
		\hline % inserts single-line
		\end{tabular}
	}
	\quad
	\subtable[$Re = 400$, $Rm =1$]{
		\begin{tabular}{l | c c c}
		\hline\hline
		 \backslashbox{$\Delta t$}{$h$} & $\nicefrac{1}{4}$ & $\nicefrac{1}{8}$ & $\nicefrac{1}{16}$ \\ [0.5ex] 
		\hline 
		$0.02$ & 11 & 7 & 5 \\
		% Entering 2nd row
		$0.01$ & 13 & 8 & 6 \\
		% Entering 3rd row
		$0.005$ & 15 & 10 & 7 \\
		% Entering 4th row
		$0.0025$ & 16 & 11 & 8 \\
		% [1ex] adds vertical space
		\hline % inserts single-line
		\end{tabular}
	}
	\quad
	\subtable[$Re = 400$, $Rm =400$]{
		\begin{tabular}{l | c c c}
		\hline\hline
		 \backslashbox{$\Delta t$}{$h$} & $\nicefrac{1}{4}$ & $\nicefrac{1}{8}$ & $\nicefrac{1}{16}$ \\ [0.5ex] 
		\hline 
		$0.02$ & 11 & 7 & 5 \\
		% Entering 2nd row
		$0.01$ & 13 & 8 & 6 \\
		% Entering 3rd row
		$0.005$ & 15 & 9 & 7 \\
		% Entering 4th row
		$0.0025$ & 16 & 11 & 7 \\
		% [1ex] adds vertical space
		\hline % inserts single-line
		\end{tabular}
	}
	\caption{Block lower triangular preconditioner {\small $\widehat{\mathcal{M}}_{\mathcal{L}}$} \eqref{eqn:barL-no-stable} for FGMRES method (diagonal blocks are solved approximately)}
	\label{tab:3Dinexact_lower}
\end{table}
\end{tiny}

% conclusion.
\section{Conclusions}
In the numerical simulations for the incompressible MHD system, the most time-consuming part is solving the linear system.  In this paper, we design several new block preconditioners for the linear systems resulting from the structure-preserving finite element discretization proposed in \cite{Hu.K;Ma.Y;Xu.J.2014a}.  We develop these preconditioners from the perspective of the functional and PDE analysis.  By rigorously proving the well-posedness of the discretization schemes and carefully studying the mapping property of the linearized operators between appropriate Sobol\'{e}v spaces with proper weighted norms following the framework in \cite{Mardal.K;Winther.R.2011a,Loghin.D;Wathen.A.2004a}, we develop two types of preconditioners: norm-equivalent preconditioners (for example, block diagonal preconditioners) and FOV-equivalent preconditioners (for example, block triangular preconditioners). By revisiting the inf-sup conditions of the discrete systems under the new weighted norms, we theoretically verify the robustness of the preconditioners when the time step size is small enough, and further prove that the resulting preconditioned Krylov iterative methods (for example, MINRES and GMRES methods) converge uniformly with respect to the mesh size, time step size, and physical parameters including the relative electrical conductivity {\small $\sigma_{r}$}, the relative magnetic permeability {\small $\mu_{r}$}, the coupling number {\small $s$}, and the magnetic Reynolds number  {\small $Rm$}. We also verify the theoretical conclusions by numerical experiments.

Another contribution of this paper is that we improve the analysis of FOV-equivalent preconditioners proposed in \cite{Loghin.D;Wathen.A.2004a}.  Their analysis requires scaling parameters in front of the diagonal blocks in {\small $\widehat{\mathcal{M}}_{\mathcal{L}}$} under certain constrains, which usually are difficult to choose in practice. In our analysis, with the help of an appropriate norm $(\cdot, \cdot)_{ \mathcal{M}^{-1} }$, we are able to remove those unnecessary scaling parameters, which makes the theoretical results consistent with practical implementation and observations. 

Finally we would like to point out that we have not considered the convection dominant cases in either the design of discretization or the design of preconditioner.  For convection dominant cases (that could happen in \eqref{eq:dimensionless1} and or \eqref{eq:dimensionless3}), we need to use special discretization methods such as upwinding, streamline diffusion or other stabilization techniques and we also need to use tailored techniques for designing preconditioners for the resulting discrete systems.  This is a topic of our ongoing investigations.  We expect that for all practical purposes, preconditioners developed in this paper can be extended to certain convection dominant cases without much difficulty but we do not at all expect that it is easy to extend our theory to such cases as rigorous theoretical results that are uniform with respect to the magnitude of the convection are extremely rare for either continuous or discrete case in any circumstances.  Despite the limitation of our results for the convection dominant case, we trust that the rigorous theoretical results and the supporting numerical experiments presented in this paper will help shed some light on certain aspects of the highly complicated MHD systems.

\section*{Acknowledgements}
The authors would like to thank Lu Wang for a lot of useful discussions and suggestions.

\bibliographystyle{plain}
\bibliography{MHD}{}

\end{document}